%% file: main.tex
\documentclass[12pt, reqno]{amsart}
\pdfoutput=1
\usepackage{mathpreamble}

\title[Statistical convex-cocompactness for non-orientable mapping class groups]{Statistical convex-cocompactness for mapping class groups of non-orientable surfaces}
\author{Sayantan Khan}
\address{Department of Mathematics, University of Michigan, Ann Arbor, MI}
\email{\href{mailto:saykhan@umich.edu}{saykhan@umich.edu}}
\thanks{}
\urladdr{\url{https://www.sayantankhan.io}}


\date{\today}

\begin{document}
\begin{abstract}
  We show that a finite volume deformation retract $\systole(\no_g)/\mcg(\no_g)$ of the moduli space $\mo(\no_g)$ of non-orientable surfaces $\no_g$ behaves like the convex core of $\mo(\no_g)$, despite not even being quasi-convex.
  We then show that geodesics in the convex core leave compact regions with exponentially low probabilities, showing that the action of $\mcg(\no_g)$ on $\systole(\no_g)$ is \emph{statistically convex-cocompact}.
  Combined with results of Coulon and Yang, this shows that the growth rate of orbit points under the mapping class group action is purely exponential, pseudo-Anosov elements in mapping class groups of non-orientable surfaces are exponentially generic, and the action of mapping class group on the limit set in the horofunction boundary is ergodic with respect to the Patterson-Sullivan measure.
  A key step of our proof relies on \emph{complexity length}, developed by Dowdall and Masur, which is an alternative notion of distance on Teichmüller space that accounts for geodesics that spend a considerable fraction of their time in the thin part.
\end{abstract}
\maketitle





\input{introduction.tex}
\input{preliminaries.tex}
\input{weak-convex-core.tex}
\input{thin-part.tex}
\input{entropy-equality.tex}
\input{linear-gap.tex}

\appendix

\input{appendix.tex}

\printbibliography

\end{document}

%% file: introduction.tex
\section{Introduction}
\label{sec:introduction}



Mapping class groups $\mcg(\no_g)$ of non-orientable surfaces $\no_g$ have many similarities to infinite covolume geometrically finite Fuchsian groups, in a manner similar how mapping class groups of orientable surfaces sometimes behave like lattices in $\mathrm{SL}_2(\mathbb{R})$.

\begin{enumerate}[(i)]
\item The mapping class groups $\mcg(\no_g)$ are finitely presented.
\item The action of $\mcg(\no_g)$ on the Teichmüller space $\teich(\no_g)$ has infinite $\nu_N$-covolume, where $\nu_N$ is the generalization of the Weil-Petersson volume form on Teichmüller spaces of non-orientable surfaces (\cite[Theorem 17.1]{gendulphe2017whats} and \cite{norbury2008lengths}).
\item The limit set of $\mcg(\no_g)$ in the Thurston boundary of $\teich(\no_g)$ is $\pml^+(\no_g)$, i.e.\ the projective measured laminations that have no two-sided components  (\cite{erlandsson2023mapping} and \cite{limitsetkhan}), which is a subset of the boundary with zero Lebesgue measure (\cite{norbury2008lengths}).
\item The Teichmüller geodesic flow is not ergodic with respect to any Borel measure on the unit cotangent bundle with full support (\cite[Proposition 17.5]{gendulphe2017whats}).
\item There exists an $\mcg(\no_g)$-equivariant deformation retract of $\teich(\no_g)$ to $\systole(\no_g)$, the subset where no one-sided curve is shorter than $\vept > 0$. The action of $\mcg(\no_g)$ on $\systole(\no_g)$ is finite $\nu_N$-covolume \cite[Proposition 19.1]{gendulphe2017whats}.
\end{enumerate}

With these similarities, one might expect that $\systole(\no_g)$ serves as a convex core of $\teich(\no_g)$, and the geodesic flow restricted to tangent directions whose forward and backward end points lie in the limit set will be ergodic, with respect to a finite geodesic flow invariant measure supported only on these tangent directions.

A prior result of the author shows that this is not the case: $\systole(\no_g)$ is not even quasi-convex, i.e.\ geodesic segments whose endpoints lie in $\systole(\no_g)$ can travel arbitrarily far from $\systole(\no_g)$.

\begin{theorem}[Theorem 5.2 of \cite{limitsetkhan}]
  For all $\vept > 0$, and all $D > 0$, there exists a Teichmüller geodesic segment whose endpoints lie in $\systole(\no_g)$, such that some point in the interior of the geodesic segment is more than distance $D$ from $\systole(\no_g)$.
\end{theorem}

However, we show that this failure of quasi-convexity is not a serious obstruction to understanding geodesic segments whose endpoints lie in $\systole(\no_g)$.

\begingroup
\def\thetheorem{\ref{thm:weak-convexity}}
\begin{theorem}
  For any $\vepd > 0$, there exists constants $\vept^{\prime}$ and $c$, such that any geodesic segment $\gamma$, whose length is more than $c$, with endpoints in $\systole(\no_g)$, for $0 < \vept < \vept^{\prime}$, can be homotoped to a segment relative to endpoints to lie entirely within $\systole(\no_g)$, such that the length of the homotoped segment $\gamma^{\prime}$ satisfies the following inequality.
  \begin{align*}
    \ell(\gamma^{\prime}) \leq \ell(\gamma) \cdot (1 + \vepd)
  \end{align*}
\end{theorem}
\addtocounter{theorem}{-1}
\endgroup

Theorem \ref{thm:weak-convexity} shows that $\systole(\no_g)$, despite not being convex, almost behaves like the convex core of $\teich(\no_g)$: it is a metric subset of $\teich(\no_g)$ (with respect to the induced path metric) which is distorted by an arbitrarily small amount.
We call $\systole(\no_g)$ the \emph{weak convex core} of $\teich(\no_g)$, and focus our attention on this subspace as a metric space, where the metric is the induced path metric, which we denote by $d_{\vept}$ (in all of our results that follow, we are using this metric on $\systole(\no_g)$, and not the Teichmüller metric).
If we restrict our attention to the cotangent directions in $\teich(\no_g)$ along which the geodesic flow does not eventually leave $\systole(\no_g)$, we can use those cotangent directions to define a geodesic flow for $\systole(\no_g)$.
We call this collection of restricted directions the restricted cotangent bundle over $\systole(\no_g)$.


In light of this, we restrict our attention to $\systole(\no_g)$, and the $\mcg(\no_g)$ action on $\systole(\no_g)$.
Since the action of $\mcg(\no_g)$ on $\systole(\no_g)$ is finite $\nu_N$-covolume (but not cocompact), one might try to prove that the action is
{like}
the action of lattices in $\mathrm{SL}_2(\mathbb{R})$ on $\mathbb{H}$.
However, the results on lattices (and Teichmüller spaces of orientable surfaces) rely on having a measure preserving $\mathrm{SL}_2(\mathbb{R})$ action on the unit tangent bundle (respectively on the moduli space of quadratic differentials), and use the interplay between the geodesic flow and the horocycle flow.

For non-orientable surfaces, we do not have an analog of the horocycle flow on the space of quadratic differentials, so we cannot hope to directly import the techniques from the orientable case.
However, \textcite{10.1093/imrn/rny001} introduced a notion of \emph{statistically convex-cocompact action}, which can replace the notion of a lattice-like action for our setting.
In the setting of $\systole(\no_g)$, proving statistical convex-cocompactness is equivalent to proving that geodesic segments between $\mcg(\no_g)$ orbit points in $\systole(\no_g)$ enter the thin part (i.e.\ the region in $\systole(\no_g)$ where some two-sided curve is short) with low probabilities.

Our main result is that this holds for the $\mcg(\no_g)$ action on $\systole(\no_g)$.
\begin{theorem}[Corollary of Theorems \ref{thm:entropy-equality-implies-scc} and \ref{thm:entropy-equality}]
  \label{thm:statistical-convex-cocompactness}
  The action of $\mcg(\no_g)$ on $(\systole(\no_g), d_{\vept})$ is statistically convex-cocompact.
\end{theorem}

Using Theorem \ref{thm:weak-convexity} and a result of \textcite{minsky1996quasi}, we have that the projection of balls disjoint from axes of the pseudo-Anosov elements have bounded diameter.
This means pseudo-Anosov elements are \emph{strongly contracting} and most reducible elements are not (see Lemma \ref{lem:strongly-contract-class}).

Coulon (\cite{coulon2022patterson}, \cite{coulon2024ergodicity}) and Yang (\cite{10.1093/imrn/rny001}, \cite{yang2020genericity}) prove fairly general results in the setting of statistically convex cocompact group actions with strongly contracting elements.
We state these results in our setting, where $\mcg(\no_g)$ plays the role of the group, $\systole(\no_g)$ is the metric space upon which $\mcg(\no_g)$ acts via a statistically convex cocompact action (and $d_{\vept}$ denotes the distance function), and pseudo-Anosovs play the role of strongly contracting elements.
We also denote balls of radius $R$ with respect to the metric $d_{\vept}$ centered at $p$ as $B_R(p)$.

\begin{theorem}[Purely exponential growth (Theorem B of \cite{10.1093/imrn/rny001})]
  There exist positive constants $A$ and $B$ such that the following bounds hold for the cardinality of the $\mcg(\no_g)$ orbit of a point $p \in \systole(\no_g)$ in a ball of radius $R$.
  \begin{align*}
    A \exp(hR) \leq \#\left( \mcg(\no_g) \cdot p \cap B_R(p) \right) \leq B \exp(hR)
  \end{align*}
  Here, $h$ is the critical exponent for the group action.
\end{theorem}

We also have that pseudo-Anosov elements are exponentially generic with respect to the induced metric on $\systole(\no_g)$.

\begin{theorem}[Exponential genericity of contracting elements (Theorem 1.4 of \cite{yang2020genericity})]
  \label{thm:pure-exponential}
  For a point $p \in \systole(\no_g)$, let $N(R)$ denote the number of mapping class elements $\gamma$ such that $d_{\vept}(p, \gamma p) \leq R$, and let $N_{nc}(R)$ denote the reducible mapping class elements $\gamma$ that leave a two-sided curve invariant, and $d_{\vept}(p, \gamma p) \leq R$.
  Then there exists a positive constant $c$ such that the following holds for large enough $R$.
  \begin{align*}
    \frac{N_{nc}(R)}{N(R)} \leq \exp(-cR)
  \end{align*}
\end{theorem}

After picking a basepoint $p \in \systole(\no_g)$, we can replicate the classical construction of Patterson-Sullivan measures \cite{sullivan1979density} to get a measure $\nu$ supported on the limit set $\pml^+(\no_g)$.
However, for Theorem \ref{thm:hts-dichotomy}, we need to consider the limit set in an alternative compactification, namely the horofunction compactification, and consider the Patterson-Sullivan measure on the limit set in this compactification.
We will abuse notation, and denote the Patterson-Sullivan measure on this limit set by $\nu$ as well.
Since we have that the action is statistically convex cocompact, and we have plenty of strongly contracting elements, the results of Coulon (\cite{coulon2022patterson},\cite{coulon2024ergodicity}) let us say a lot about the Patterson-Sullivan measure $\nu$.
We have the Hopf-Tsuji-Sullivan dichotomy (or rather, the divergent half of the result).

\begin{theorem}[Hopf-Tsuji-Sullivan dichotomy (Theorem 1.1 of \cite{coulon2024ergodicity})]
  \label{thm:hts-dichotomy}
  For the action of $\mcg(\no_g)$ on $\systole(\no_g)$, the following are equivalent.
  \begin{enumerate}[(i)]
  \item The Poincaré series for $\mcg(\no_g)$ diverges at the critical exponent.
  \item The measure $\nu$ gives full measure to the radial limit set.
  \item The diagonal action on $\Lambda \times \Lambda$ is ergodic with respect to the product measure $\nu \otimes \nu$, where $\Lambda$ is the limit set of $\mcg(\no_g)$ in the horofunction boundary.
  \item The geodesic flow on the unit tangent bundle $(\Lambda \times \Lambda \times \mathbb{R})/\mcg(\no_g)$ is ergodic with respect to the Bowen-Margulis measure.
  \end{enumerate}
\end{theorem}

It follows as a corollary of Theorem \ref{thm:pure-exponential} that the Poincaré series for $\mcg(\no_g)$ diverges at the critical exponent, and as a result, we have all the other equivalent criteria that appear in the Hopf-Tsuji-Sullivan dichotomy.



\begin{remark}
  In practice, it is preferable to work with the Thurston boundary of Teichmüller space rather than the horofunction boundary, since the Thurston boundary is a much better understood object.

  For statistically convex-cocompact actions of subgroups of mapping class groups (with respect to the Teichmüller metric) (see \textcite{gekhtman2023dynamics}, and \textcite{CGTY}), one can replace the horofunction boundary in the statement of the Hopf-Tsuji-Sullivan dichotomy with the Thurston boundary using the following sequence of results.
  \begin{enumerate}[(i)]
  \item \textcite{miyachi2008teichmuller} proves that there is an $\mcg(\no_g)$ equivariant bijective map between the uniquely ergodic points in the horofunction boundary, and the uniquely ergodic points in the Thurston boundary.
  \item By condition (ii) of the Hopf-Tsuji-Sullivan dichotomy, the radial limit points in the horofunction boundary have full measure.
  \item A geodesic converging to a radial limit point stays within bounded distance of the mapping class group orbit, and by Masur's criterion for unique ergodicity \cite{masur1982interval}, we have that the corresponding limit point in the boundary must be uniquely ergodic.
  \item Consequently, the uniquely ergodic limit points have full measure in the horofunction boundary, and we can thus replace that with the Thurston boundary, since they agree on a full measure subset.
  \end{enumerate}

  In our setting, i.e. the $\mcg(\no_g)$ action on $\systole(\no_g)$ with the $d_{\vept}$ metric, we do not know if there is an $\mcg(\no_g)$-equivariant map between the uniquely ergodic points in the Thurston boundary with a full measure subset of the horofunction boundary.
  We do not even have a very good description of the horofunction boundary of $(\systole(\no_g), d_{\vept})$ in terms of measured foliations on $\no_g$.
  One way to obtain ergodicity of the $\mcg(\no_g)$ action on $\pml^+(\no_g)$ would be to get an affirmative answer to the following question.
\begin{question}[Restatement of Question \ref{ques:full-stat-core}]
  \label{ques:full-scc}
  Is the action of $\mcg(\no_g)$ on $(\teich(\no_g), d_T)$ statistically convex cocompact, where $d_T$ is the Teichmüller metric?
\end{question}
We expect the answer to this question is yes, despite the random walk methods of our proof not working in this setting.
\end{remark}

\subsection*{Why we care about Patterson-Sullivan theory for $\mcg(\no_g)$}

\subsubsection*{Counting functions}

Understanding the dynamics of the geodesic flow over the moduli space of \emph{orientable surfaces} has led to solutions for two counting problems: one on the moduli space of orientable hyperbolic surfaces, and one on orientable hyperbolic surfaces themselves.

\begin{enumerate}[(i)]
\item Counting closed curves in moduli space: Via techniques originally introduced to Margulis in his thesis \cite{margulis2004some}, one can reduce counting closed curves, which are conjugacy classes of mapping class group orbit points, to understanding the geodesic flow over the moduli space.
  The number of closed curves of length at most $R$, which we denote by $N(R)$ has the following asymptotics (see \cite{eskinmirzakhani}).
  \begin{align}
    \label{eq:counting-closed}
    N(R) \sim \frac{\exp(hR)}{hR}
  \end{align}
  Here, the symbol $\sim$ means that the ratio of the two quantities approaches a positive constant as $R$ goes to $\infty$, and $h$ is the volume growth entropy of $\teich(\os_{g})$, which is $6g-6$.
\item Counting \emph{simple} closed curves on orientable hyperbolic surfaces: \textcite{mirzakhani2008growth} proved that the counting function $M(R)$ that counts \emph{simple} closed curves satisfies a polynomial asymptotic.
  \begin{align}
    \label{eq:counting-simple-closed}
    M(R) \sim R^{h}
  \end{align}
  Here, $h$ is again the volume growth entropy, i.e.\ $6g-6$.
\end{enumerate}

For non-orientable surfaces, the counting function does not behave like the orientable version.
\textcite{gendulphe2017whats} showed that the counting function $N_{\mathrm{no}}(R)$ and $M_{\mathrm{no}}(R)$, which is the versions of the functions $N(R)$ and $M(R)$ for non-orientable surfaces satisfy the following asymptotic.
\begin{align*}
  N_{\mathrm{no}}(R) &= o\left( \frac{\exp((3g-6)R)}{(3g-6)R} \right) \\
  M_{\mathrm{no}}(R) &= o(R^{3g-6})
\end{align*}
These asymptotics raise the question of whether there is an exponent $h < 3g-6$ for which the non-orientable versions of \eqref{eq:counting-closed} and \eqref{eq:counting-simple-closed} continue to hold.
For $g = 3$, \textcite{10.1093/imrn/rny112} obtained precise asymptotics, and in this case, the growth rate is a non-integer exponent smaller than $3g - 6 = 3$.

If we can upgrade ergodicity of the geodesic flow to mixing of the geodesic flow, we can count lattice points and their conjugacy classes to obtain a non-orientable version of \eqref{eq:counting-closed} where the role of $h$ is played by the critical exponent for the group action of $\mcg(\no_g)$ with respect to the Teichmüller metric.

If the answer to Question \ref{ques:full-scc} is yes, we can use the results of \textcite{CGTY} to establish mixing of the geodesic flow with respect to the Bowen-Margulis measure, and use that to count lattice points.

To explain how ergodicity of the $\mcg(\no_g)$ action on $\pml^+(\no_g)$ might help count simple closed curves on $\no_g$, we outline Mirzakhani's original proof of the fact for orientable surfaces (see \cite{mirzakhani2008growth} for the original proof, and \cite{2022arXiv220204156A} for a gentler exposition).

\begin{proof}[Sketch of simple closed curve counting in the orientable case]
  The proof proceeds in 3 steps.
  \begin{enumerate}[Step 1:]
  \item For any simple closed curve $\gamma$ and any $L > 0$, consider the measure $\mu_L$ on the space $\ml(\os_g)$ of measured laminations.
    \begin{align*}
      \mu_L \coloneqq \frac{1}{L^{6g-6}} \sum_{\alpha \in \mcg(\os_g)} \delta_{\frac{1}{L}\alpha \gamma}
    \end{align*}
  \item Letting $L$ go to $\infty$, $\left\{ \mu_L \right\}$ converges to measure $\mu$ that is $\mcg(\os_g)$-invariant.
    By ergodicity of the $\mcg(\os_g)$-action on $\ml(\os_g)$ with respect to the Thurston measure, we have that the limiting measure $\mu$ is a constant multiple $c$ times the Thurston measure.
  \item To show that the constant $c$ is positive, one needs to average over the moduli space $\mo(\os_g)$, using Mirzakhani's integration formula for the Weil-Petersson volume form.
  \end{enumerate}
\end{proof}

To replicate this proof in the non-orientable setting, we pick the original simple closed curve $\gamma$ to be a \emph{two-sided} curve, and replace the exponent $6g-6$ with $h + 1$, where $h$ is the critical exponent of $\mcg(\no_g)$.
With this replacement, we have the following question.
\begin{question}
  Do the sequence of measures $\mu_L$ converge to a locally finite measure supported on $\ml^+(\no_g)$?
\end{question}

\textcite{erlandsson2023mapping} show that the $\mcg(\no_g)$ orbit closure of such a point indeed is $\ml^+(\no_g)$, but the question about convergence of measures is still open.
Since our results suggest that the $\mcg(\no_g)$ action on $\pml^+(\no_g)$ is ergodic with respect to the Patterson-Sullivan measures (and answering Question \ref{ques:full-scc} would imply ergodicity), we can ask the following questions about the limiting measure $\mu$ as well as the product of the Patterson-Sullivan measures with the Lebesgue measure.

\begin{question}
  Is the limiting measure $\mu$ absolutely continuous with respect to the ergodic measure on $\ml^+(\no_g)$ obtained by taking a product of Patterson-Sullivan measure and the Lebesgue measure?
\end{question}
If the answer to the question is yes, then one will have completed Step 2 of the proof for non-orientable surfaces.

To make Step 3 work for non-orientable surfaces, one needs to construct a recursive formula for the volumes of $\systole(\no_g)$: this has been done by \textcite{stanford2023mirzakhani}.






\subsubsection*{Geometric finiteness for mapping class subgroups}

One can think of $\mcg(\no_g)$ as a subgroup of $\mcg(\os_{g-1})$ (where $\os_{g-1}$ is the orientation double cover of $\no_g$), where the embedding is obtained by lifting mapping classes on $\no_g$ to orientation preserving mapping classes on $\os_{g-1}$.
The image of $\mcg(\no_g)$ is an infinite-index subgroup, and stabilizes an isometrically embedded copy of $\teich(\no_g)$ inside $\teich(\os_{g-1})$.

For subgroups of mapping class groups, the notion of convex-cocompactness was introduced by \textcite{farb2002convex}: these groups have good properties with respect to their dynamics on the Teichmüller space.
A natural generalization of these subgroups, inspired by the Kleinian setting, is the notion of geometric finiteness.
While there is not universally agreed upon notion of geometric finiteness for mapping class subgroups, the following two classes are subgroups are considered to be geometrically finite by any reasonable definition.

\begin{enumerate}[(i)]
\item Veech groups: These are stabilizers of Teichmüller discs in $\teich(\os_{g-1})$ which are finitely generated.
  They are lattices in $\mathrm{SL}_2(\mathbb{R})$, and their action on the Teichmüller discs they stabilize is well understood via hyperbolic geometry.
\item Combinations of Veech groups: \textcite{leininger2006combination} show that if two Veech groups $H$ and $K$ share a maximal parabolic subgroup $A$, the subgroup they generate is $H \ast_A K$ (after possibly conjugating by a pseudo-Anosov).
\end{enumerate}

The key emphasis with these two examples is that there are only finitely many cusps, i.e.\ finitely many conjugacy classes of reducible elements.
However, that is not the case for $\mcg(\no_g)$, it stabilizes an isometrically embedded sub-manifold, and yet there are infinitely many conjugacy classes of reducible elements.
Despite having infinitely many ``cusps'', our results show that it is still possible to do Patterson-Sullivan theory on $\mcg(\no_g)$, which is a departure from the Fuchsian/Kleinian setting, where finite Bowen-Margulis measure requires finitely many cusps.


\subsection*{Idea behind proofs of main theorems}

In this subsection, we outline the key ideas behind the proof of the main theorems.

\subsubsection*{Weak convexity of $\systole(\no_g)$}
We construct a projection map from $\teich(\no_g)$ to $\systole(\no_g)$ which takes any one-sided curve of length less than $\vept$ and increases its length to $\vept$, while keeping the lengths and twists of other curves constant.
We then use Minsky's product region theorem to show that this projection map increases distance by only a factor of $(1 + \vepd)$, where $\vepd$ can be picked to be arbitrarily small.

\subsubsection*{Statistical convexity of $\systole(\no_g)$}
To show that geodesics in $\systole(\no_g)$ stay away from the thin part, we construct a random walk on $\systole(\no_g)$, and compute the probability of a single step of the random walk entering the thin part, and show that this probability is small.
Estimating this probability reduces to computing an average over a ball in $\mathbb{H}$ because of Minsky's product region theorem.
The random walk argument gives us that the number of geodesics of length at most $R$ entering the thin part is at most $\exp((\hNP - 1)R)$, where $\hNP$ is the discrete analog of the volume growth entropy of $\systole(\no_g)$.
However, the total number of geodesics of length at most $R$ grows like $\exp((\hLP)R)$, where $\hLP$ is the growth rate of the number of lattice points.
To show that the probability of a geodesic entering the thin part is exponentially small, we need to relate the two entropy terms, and show that $\hLP > \hNP - 1$.

\subsubsection*{Showing $\hLP = \hNP$}
We prove entropy equality by inducting on the complexity of the surface.
We first show it for surfaces with Euler characteristic equal to $-1$ using direct methods, and reduce the inductive step to proving an estimate on complexity length for geodesic segments that spend a definite fraction of their time in thin part.

\subsubsection*{Complexity length estimate}

In this section, we
show that geodesic segments that spend a small but definite fraction of time near their end in the thin part are rare.
We do this by showing that $\hNP$ for a proper subsurface is strictly smaller than $\hNP$ for the entire surface, and use the machinery of complexity length (due to \textcite{dowdall2023lattice}), which builds upon Minsky's product region theorem and hierarchical hyperbolicity of Teichmüller space, to show that geodesic segments ending in the thin part are rare.

\subsection*{Acknowledgements}

The author would like to thank his advisor Alex Wright for constant support for the duration of this project.
The author would also like to thank Spencer Dowdall, Howard Masur, Kasra Rafi, and Jacob Russell for explaining various aspects of hierarchical hyperbolicity.
The author would also like to thank Ilya Gekhtman for comments on an earlier draft of the paper.
The work done in this paper was supported by the Rackham Predoctoral Fellowship for the academic year 2022-2023.





%% file: preliminaries.tex
\section{Preliminaries}
\label{sec:preliminaries}

\subsection{Non-Orientable Surfaces}
\label{sec:non-orient-surf}

Similar to orientable surfaces, compact non-orientable surfaces with (possibly empty) boundary are classified by their \emph{demigenus} and number of boundary components.
The demigenus of a non-orientable surfaces is the number of copies of $\mathbb{RP}^2$ that need to be connect-summed in order to get the non-orientable surface.
An alternative way to construct non-orientable surfaces is to start with an orientable surface, and attach \emph{crosscaps}: a crosscap is attached by deleting the interior of an embedded disc, and gluing the $S^1$ boundary of that disc to itself via the antipodal map.

To unify notation between orientable and non-orientable surfaces, we will denote a compact surface with boundary using $\os_{g,b,c}$, which denotes a surface of genus $g$, with $b$ boundary components, and $c$ crosscaps attached.
With this notation, a non-orientable surface $\no_{g, b}$ of demigenus $g$ with $b$ boundary components is $\os_{\frac{g-1}{2}, b, 1}$ if $g$ is odd, and $\os_{\frac{g-2}{2}, b, 2}$ if $g$ is even.

One can classify simple closed curves on a non-orientable surface into two categories based on the topology of their tubular neighbourhoods.
\begin{description}
\item[Two-sided curves] These are curves whose tubular neighbourhoods are homeomorphic to cylinders.
\item[One-sided curves] These are curves whose tubular neighbourhoods are homeomorphic to Möbius bands.
\end{description}

The orientation double cover of $\no_g$ is $\os_{g-1}$, where $p$ denotes the covering map: the one-sided curves on $\no_g$ lift to a single curve on $\os_{g-1}$ that is twice as long, and the two-sided curves on $\no_g$ lift to two disjoint curves on $\os_{g-1}$, both of which are the same length as the original curve.
We also have an orientation reversing deck transformation $\iota$ on $\os_{g-1}$ corresponding to the covering map.
The map $\iota$ swaps the lifts of the two-sided curves, and leaves the lifts of the one-sided curves invariant.

The subgroup $\pi_1(\os_{g-1}) < \pi_1(\no_g)$ is a \emph{characteristic} subgroup, i.e.\ left invariant by an automorphism of $\pi_{1}(\no_g)$ induced by a homeomorphism, and consequently, self-homeomorphisms of $\no_g$ have a unique orientation preserving lift to self-homeomorphisms of $\os_{g-1}$, giving us an embedding $p^\ast$ of mapping class groups, induced by the covering map $p$.
\begin{align*}
  p^{\ast}: \mcg(\no_g) \hookrightarrow \mcg(\os_{g-1})
\end{align*}
The image of $\mcg(\no_g)$ is an infinite-index subgroup of $\mcg(\os_{g-1})$.
The lifting map also induces an embedding of the corresponding Teichmüller spaces, where the image of $\teich(\no_g)$ is the locus left invariant by $\iota^{\ast}$, where $\iota^{\ast}$ is the deck transformation induced map on $\teich(\os_{g-1})$.
\begin{align*}
  p^{\ast}: \teich(\no_g) \hookrightarrow \teich(\os_{g-1})
\end{align*}
This embedding is isometric, i.e.\ Teichmüller geodesics joining points in the image of $\teich(\no_g)$ stay within the image of $\teich(\no_g)$.

These facts present an alternative way of thinking about mapping class groups and Teichmüller spaces of non-orientable surfaces.
They can be thought of as a special infinite index subgroup of $\mcg(\os_{g-1})$, and a isometrically embedded totally real submanifold of $\teich(\os_{g-1})$.
We will use this point of view to prove some of the metric properties of $\teich(\no_g)$ we will require, but for most other applications, we prefer to think of $\teich(\no_g)$ and $\mcg(\no_g)$ as independent objects, without embedding them in other spaces.

The Teichmüller space for non-orientable surfaces can be given Fenchel-Nielsen coordinates using a pants decomposition for $\no_g$: the only difference from the orientable setting is that for all the one-sided curves in the pants decomposition, there is only one coordinate, associated to the length of the one-sided curve, rather than both the twist and length.
This means that Teichmüller spaces of non-orientable surfaces can have odd $\mathbb{R}$-dimension.

Since these Teichmüller spaces of non-orientable surfaces can have odd dimension, we no longer have a symplectic structure, and a corresponding volume form.
However, the image of $\teich(\no_g)$ in $\teich(\os_{g-1})$ is a Lagrangian submanifold, and consequently a Lagrangian volume form.
This Lagrangian volume form $\nu_N$ has a particularly nice description in terms of a pants decomposition $\mathcal{P}$, due to \textcite{norbury2008lengths}.

Let $\mathcal{P}$ be a pants decomposition for $\no_g$: $\nu_N$ is defined in terms of the lengths and twists of curves in $\mathcal{P}$.

\begin{align*}
  \nu_N = \left( \bigwedge_{\text{$\gamma_i$ one-sided}} \coth(\ell(\gamma_i)) d\ell(\gamma_i) \right) \wedge \left( \bigwedge_{\text{$\gamma_i$ two-sided}} d\tau(\gamma_i) \wedge d\ell(\gamma_i) \right)
\end{align*}
Here $\ell(\gamma_i)$ denotes the length of the curve $\gamma_i$, and $\tau(\gamma_i)$ denotes the twist, when $\gamma_i$ is two-sided.

Similar to Wolpert's magic formula, the $\mu_N$ has the following properties.
\begin{itemize}
\item[-] The form $\nu_N$ does not depend on the choice of pants decomposition.
\item[-] $\nu_N$ is $\mcg(\no_g)$ invariant, up to sign.
\end{itemize}
This lets us use the absolute value of $\nu_N$ as a volume form on the quotient $\teich(\no_g) / \mcg(\no_g)$.
We will, for notational convenience, use $\nu_N$ to mean $\left| \nu_N \right|$.

With respect to $\nu_N$, the action of $\mcg(\no_g)$ on $\teich(\no_g)$ is infinite covolume: the same also holds for the geodesic flow invariant volume on the full\footnote{Full referring to the entire unit cotangent bundle as opposed to the restricted unit cotangent bundle.} unit cotangent bundle.
Furthermore, the set of cotangent directions in which the geodesic flow recurs to the $\mcg(\no_g)$-cocompact part of $\teich(\no_g)$ has $\nu_N$-measure $0$: this is due to \textcite{norbury2008lengths} (see \textcite{gendulphe2017whats} for more analogies with infinite covolume Fuchsian groups).

\subsection{Critical Exponents and Patterson-Sullivan Theory}
\label{sec:crit-expon-patt}

In this section, we outline techniques that are used to deal with infinite-covolume group actions on non-positively curved metric spaces, i.e.\ Patterson-Sullivan theory.
For the sake of concreteness, we will state most results in this section for infinite-covolume geometrically finite Fuchsian groups, and specify a generalized theorem/conjecture for the setting of mapping class groups.

Let $\Gamma$ be an infinite-covolume geometrically finite Fuchsian group.
Geometric finiteness in this context means that the surface $\mathbb{H}/\Gamma$ is composed of \emph{finitely} many components outside of a large enough compact set, where each component is isometric to one of the following regions.
\begin{enumerate}[(i)]
\item Cusps: A cusp is the quotient of a horoball (i.e.\ $\left\{ \mathrm{Im}(z) > t_0 \right\}$ with the upper half plane model) with respect to an isometry of the form $
  \begin{pmatrix}
    1 & t \\
    0 & 1
  \end{pmatrix}
  $.
\item Flares: A flare is quotient of the region $\left\{ \mathrm{Re}(z) > 0 \right\}$ with respect to an isometry of the form $
  \begin{pmatrix}
    q & 0 \\
    0 & \frac{1}{q}
  \end{pmatrix}
  $.
\end{enumerate}
Note that it is the flares of the hyperbolic surface that make its volume infinite: each of the cusps has finite hyperbolic volume.

The presence of flares also means that the limit set of $\Gamma$, i.e.\ the set $\overline{\Gamma p} \cap \partial \mathbb{H}$ (for any $p \in \mathbb{H}$) is a measure $0$ subset of the boundary (with respect to the usual Lebesgue measure on $S^1$), as well as forcing the Liouville measure, which is a geodesic flow invariant measure on the unit tangent bundle $S^1 \Gamma / \mathbb{H}$ to be infinite.

Since most results from ergodic theory need a finite flow-invariant measure, the Liouville measure does not work for these infinite-covolume groups.
The fix to this problem is to construct a new (family of) measure(s) $\left\{ \mu_q \right\}$ on the boundary, which replaces the Lebesgue measure, with respect to which the limit set has full measure, and then use that measure to construct a finite geodesic flow invariant measure on the unit tangent bundle.

The family of measures on the boundary is called the Patterson-Sullivan measure, and the corresponding measure on the unit tangent bundle is called the Bowen-Margulis-Sullivan measure.

\subsubsection{Construction of Patterson-Sullivan measures}
\label{sec:constr-patt-sull}

We begin by picking a basepoint $p \in \mathbb{H}$, and a parameter $h > 0$, and consider the measure $\mu_q^h$, for $q \in \mathbb{H}$.
\begin{align*}
  \mu_q^h \coloneqq \frac{\sum_{\gamma \in \Gamma} \exp(-h d(p, \gamma q)) \delta_{\gamma q}}{\sum_{\gamma \in \Gamma} \exp(-h d(p, \gamma p))}
\end{align*}
Here, $\delta_{\gamma q}$ denotes the Dirac mass at $\delta_{\gamma q}$, and $d(p, \gamma q)$ denotes the hyperbolic distance between $p$ and $\gamma q$.

For large enough $h$, the denominator of the expression is a convergent sum, and the resulting measure has total mass that only depends on the choice of $p$ and $q$.

Conversely, for small enough values of $h > 0$, the sum in the denominator diverges, and the measure $\mu_q^h$ is not well defined.
To see this, one can use the ping pong lemma to embed a copy of the free group $F_2$ in $\Gamma$, and show that for this copy of $F_2$, there is a small enough $h$ to make the sum diverge.
We can now define the critical exponent $h_{\Gamma}$ of the group $\Gamma$.
\begin{definition}[Critical exponent]
  The critical exponent $h_{\Gamma}$ is the infimum of all the values of $h$ for which the following infinite sum converges.
  \begin{align*}
    \sum_{\gamma \in \Gamma} \exp(-h d(p, \gamma p))
  \end{align*}
\end{definition}

Note that for $h = h_{\Gamma}$, it is possible for the exponential sum to converge or diverge.
If the sum converges at the critical exponent, the group $\Gamma$ is said to be of \emph{convergent type}, and if it diverges, the group $\Gamma$ is of \emph{divergent type}.

Since the measures $\mu_q^h$ are well-defined for $h > h_{\Gamma}$, and their mass is uniformly bounded (where the bound only depends on $q$), we have that for some sequence of $h \searrow h_{\Gamma}$, the sequence of measures $\mu_q^h$ converges to some limiting measure $\mu_q$.
This family of limiting measures $\left\{ \mu_q \right\}$ is called a Patterson-Sullivan measure.
The Patterson-Sullivan measure $\left\{ \mu_q \right\}$ is not unique \emph{a priori}, since picking different sequences $h \searrow h_{\Gamma}$ might lead to different limiting measures.

In practice, the uniqueness of the Patterson-Sullivan measure follows from the ergodicity of the geodesic flow with respect to the Bowen-Margulis-Sullivan measure constructed from a given Patterson-Sullivan measure.
We will skip the construction of the Bowen-Margulis-Sullivan measure $\mu_{\mathrm{BMS}}$, since the specifics of the construction are not relevant for the remainder of the paper.
We refer the reader to \textcite{quint2006overview} for the construction of $\mu_{\mathrm{BMS}}$.


\subsubsection{Some results in Patterson-Sullivan theory}
\label{sec:some-results-patt}

The question of finiteness and ergodicity of the Bowen-Margulis-Sullivan measure is equivalent to several other conditions, some of which are easier to check in some examples.

\begin{theorem}[Hopf-Tsuji-Sullivan dichotomy; \textcite{sullivan1979density}]
  \label{thm:hts-dich}
  For a geometrically finite group $\Gamma$, the following conditions are equivalent.
  \begin{enumerate}[(i)]
  \item The group is of divergent type.
  \item The Bowen-Margulis-Sullivan measure is finite.
  \item The geodesic flow is ergodic with respect to the Bowen-Margulis-Sullivan measure.
  \end{enumerate}
\end{theorem}

The ergodicity of the geodesic flow with respect to $\mu_{\mathrm{BMS}}$ can be upgraded to mixing if the length spectrum of $\mathbb{H}/\Gamma$ generates a dense subgroup of $\mathbb{R}$.

\begin{theorem}[\textcite{babillot2002mixing}]
  \label{thm:mixing-h2}
  If $\mu_{\mathrm{BMS}}$ is finite, and the lengths of the closed geodesics on $\mathbb{H}/\Gamma$ generate a dense subgroup of $\mathbb{R}$, then the geodesic flow is mixing with respect to $\mu_{\mathrm{BMS}}$.
\end{theorem}

One can then combine Theorem \ref{thm:mixing-h2} with the following result of Roblin to count lattice points where the logarithmic error goes to $0$.

\begin{theorem}[\textcite{roblin2003ergodicite}]
  \label{thm:counting-h2}
  Let $B_p(R)$ denote the lattice point counting function.
  \begin{align*}
    B_p(R) \coloneqq \#\left( \gamma \in \Gamma \mid d(p, \gamma p) \leq R \right)
  \end{align*}
  Then there exists a constant $C$, which is the $\mu_{\mathrm{BMS}}$-volume of the unit tangent bundle of $\mathbb{H}/\Gamma$, such that $B_p(R)$ can be approximated in the following manner.
  \begin{align*}
    \lim_{R \to \infty} \log\left( \frac{C \exp(h_{\Gamma} R)}{B_p(R)} \right) = 0
  \end{align*}
\end{theorem}

\subsubsection{Extending these results to subgroups of mapping class groups}
\label{sec:extend-these-results}

In \cite{10.1093/imrn/rny001}, Yang outlined a criterion for a \emph{non-elementary group with contracting element} acting on metric space to be of divergent type: the action must be \emph{statistically convex-cocompact}.
In the context of subgroups of mapping class groups, a subgroup is non-elementary if it contains two non-commuting pseudo-Anosov elements.

To explain what a statistically convex-cocompact action is, we first need to describe what is means for a subset of a metric space to be statistically convex.

Let $X$ be a metric space with a group $G$ acting on it, and let $Y$ be a subset of $X$ which is invariant under the $G$-action, i.e.\ we have a $G$-action on $Y$ as well, and let $p$ be a point in $Y$.
One can consider two kinds of counting functions for the $G$-action on $Y$.
\begin{align*}
  N_p(R) &\coloneqq \left\{ \gamma \in G \mid d(p, \gamma p) \leq R \right\} \\
\end{align*}
The function $N_p$ is the standard lattice point counting function.
We also want to look at those lattice points that detect a failure of convexity of $Y$: we call these points \emph{concave lattice points}.
\begin{definition}[$s$-Concave lattice points]
  A lattice point $\gamma p$ is $s$-concave if some geodesic segment $\kappa$ starting in a ball of radius $s$ centered at $p$ and ending in ball of radius $s$ centered at $\gamma p$ stays outside the set $Y$.

  The path obtained by joining $p$ to the starting point of $\kappa$, then following $\kappa$, and then joining the end point of $\kappa$ to $\gamma p$ is called the \emph{concavity detecting path for $\gamma p$}.
\end{definition}
For our applications, the precise value of $s$ will not be very important: we fix it to be twice the diameter of the compact set $\thick(\no_g)/\mcg(\no_g)$ (any value larger that the diameter of $\thick(\no_g)/\mcg(\no_g)$ will work though).

Let $M_p(R)$ denote the counting function for concave lattice points.
Let $h$ and $h_c$ be the exponential growth rates for $N_p(R)$ and $M_p(R)$.
\begin{align*}
  h &\coloneqq \lim_{R \to \infty} \frac{\log\left( N_p(R) \right)}{R} \\
  h_c &\coloneqq \lim_{R \to \infty} \frac{\log\left( M_p(R) \right)}{R}
\end{align*}

\begin{definition}[Statistically convex subset]
  \label{defn:statistical-convex-subset}
  The subset $Y$ is said to be statistically convex if $h_c < h$.
\end{definition}

\begin{definition}[Statistically convex-cocompact action]
  The action of $G$ on $X$ is statistically convex-cocompact if there exists some $G$-invariant subset $Y$ such that $Y$ is statistically convex, and the action of $G$ on $Y$ is cocompact.
\end{definition}

In \cite{10.1093/imrn/rny001}, Yang shows that when a non-elementary group acts statistically convex-cocompactly on a space, the group is of divergent type.

\textcite{coulon2024ergodicity} shows that for groups with strongly contracting elements that act statistically convex-cocompactly, a version of the Hopf-Tsuji-Sullivan dichotomy (Theorem \ref{thm:hts-dich}) holds.
Combining this with Yang's result of the group being of divergence type, one can conclude that the Bowen-Margulis-Sullivan measure on the unit cotangent bundle has finite mass and the geodesic flow is ergodic.


In the remainder of this paper, we show that the action of $\mcg(\no_g)$ on $\systole(\no_g)$ (the subset of $\teich(\no_g)$ where the one-sided curves cannot be shorter than $\vept$) is statistically convex-cocompact.





\subsection*{List of notation}
\begin{itemize}
\item[-] $\os_g$: An orientable surface of genus $g$.
\item[-] $\os_{g,b,c}$: A surface of genus $g$ with $b$ boundary components, and $c$ crosscaps attached.
\item[-] $\no_g$: A non-orientable surface of genus $g$: this is the same as $\os_{\frac{g-1}{2}, 0, 1}$ if $g$ is odd, and $\os_{\frac{g-2}{2}, 0, 2}$ if $g$ is even.
\item[-] $\teich(S)$: The Teichmüller space of the surface $S$.
\item[-] $\systole(S)$: The one-sided systole superlevel set in $\teich(S)$.
\item[-] $\nu_N$: The Lagrangian volume form on $\teich(\no_g)$.
\item[-] $B_{\tau}(x)$: A ball of radius $\tau$ (with respect to the Teichmüller metric) centered at $x$.
\item[-] $B_{\tau}^{\vept}(x)$: A ball of radius $\tau$ (with respect to the induced path metric on $\systole(\no_g)$) centered at $x$.
\item[-] $A_{\tau}$: The averaging operator on a ball of radius $\tau$.
\item[-] $\hLP(\teich(S))$: The exponential growth rate for the mapping class group orbit of a point $x$ in $\teich(S)$.
\item[-] $\hLP(H)$: For a subgroup $H$ of $\mcg(S)$, this is the exponential growth rate of for the $H$-orbit of a point $x$ in $\teich(S)$.
\item[-] $\net$: An $(\vepn, 2\vepn)$-net.
\item[-] $\hNP(\core(\teich(S)))$: This is the exponential growth rate for the net points in an $(\vepn, 2 \vepn)$-net in the weak convex core of $\teich(S)$. The value of $\vepn$ is usually clear from the context.
\item[-] $\pitchfork$: $U \pitchfork V$ denotes that the surfaces $U$ and $V$ are transverse.
\item[-] $\pitchfork_{W}$: $U \pitchfork_W V$ denotes that $U$ and $V$ are transverse when restricted to any subsurface of $W$ which intersects both $U$ and $V$ non-trivially.
\item[-] $U \lessdot V$: The Behrstock partial order for transverse subsurfaces $U$ and $V$.
\item[-] $\emul$: We say $a \emul b$ if $a$ and $b$ are equal up to a multiplicative error of $k$ and an additive error of $c$, where $k$ and $c$ are some fixed constants.
\end{itemize}


%% file: weak-convex-core.tex
\section{The Weak Convex Core of $\teich(\no_g)$}
\label{sec:weak-convex-core}

\subsection{Issues with Geometric Finiteness and Statistical Convex-Cocompactness}
\label{sec:issu-with-geom}

In order to show that the action of $\mcg(\no_g)$ on $\teich(\no_g)$ is geometrically finite (in the sense of Fuchsian groups), we need to exhibit a \emph{convex core}, i.e.\ a convex subset of $\teich(\no_g)$ on which the action of $\mcg(\no_g)$ is finite covolume.
Similarly, to show that the action of $\mcg(\no_g)$ on $\teich(\no_g)$ is statistically convex-cocompact, we need to exhibit a \emph{statistical convex core}, which is a \emph{statistically convex subset} (see Definition \ref{defn:statistical-convex-subset}) of $\teich(\no_g)$ on which $\mcg(\no_g)$ acts cocompactly.

A candidate for the convex core was suggested by \textcite{gendulphe2017whats}, namely the \emph{one-sided systole superlevel set} $\systole(\no_g)$.

\begin{definition}[One-sided systole superlevel set]
  The one-sided systole superlevel set is the subset of $\teich(\no_g)$ where no one-sided curve is shorter than $\vept$. This set is denoted $\systole(\no_g)$.
\end{definition}

The subset $\systole(\no_g)$ has several properties that suggest it should be the convex core for the $\mcg(\no_g)$ action.
\begin{itemize}
\item[-] The space $\teich(\no_g)$ $\mcg(\no_g)$-equivariantly deformation retracts onto the subset $\systole(\no_g)$ (Proposition 19.2 of \cite{gendulphe2017whats}).
\item[-] The $\mcg(\no_g)$ action on $\systole(\no_g)$ has finite $\nu_N$-covolume, where $\nu_N$ is the non-orientable analog of the Weil-Petersson volume form (Proposition 19.1 of \cite{gendulphe2017whats}).
\end{itemize}

However, the subset $\systole(\no_g)$ fails to be convex, in a very strong sense, as we show in a prior paper.
\begin{theorem}[Theorem 5.2 of \cite{limitsetkhan}]
  For all $\vept > 0$, and all $D > 0$, there exists a Teichmüller geodesic segment whose endpoints lie in $\systole(\no_g)$ such that some point in the interior of the geodesic is more than distance $D$ from $\systole(\no_g)$.
\end{theorem}

With this obstruction, we see that $\systole(\no_g)$ will not work as a convex core for a geometrically finite action.
There are two directions one could go in with this obstruction in mind, which we phrase as open questions.

If we wish to show that $\mcg(\no_g)$ acts geometrically finitely, this is the question we need to answer.
\begin{question}
  \label{ques:geom-finite-core}
  Does there exist some other subset of $\teich(\no_g)$ that is finite $\nu_N$-covolume, convex, and an $\mcg(\no_g)$-equivariant deformation retract of $\teich(\no_g)$?
\end{question}

Alternatively, if we wish to show that $\mcg(\no_g)$ acts statistically convex-cocompactly on $\teich(\no_g)$, where $\thick(\no_g)$ acts as the statistical convex core, this is the question we need to answer.

\begin{question}
  \label{ques:full-stat-core}
  Is $\thick(\no_g)$ statistically convex?
\end{question}

We suspect the answer to Question \ref{ques:full-stat-core} is yes, despite our methods not working.
Our methods for proving statistical convexity rely on proving recurrence of random walks on $\teich(\no_g)$, and random walks on all of $\teich(\no_g)$ have poor recurrence properties when they enter regions where one-sided curves are short.
When the random walks enter these one-sided thin regions, they behave like symmetric random walks on $\mathbb{Z}$, which we know do not have very strong recurrence properties.
We explain this in more detail in Section \ref{sec:why-approach-fails}, when we set up the machinery of Foster-Lyapunov-Margulis functions.

\subsection{A Weaker Notion of Convexity}
\label{sec:weak-noti-conv}

Rather than directly answering questions \ref{ques:geom-finite-core} or \ref{ques:full-stat-core}, we define an even weaker notion of convexity as an intermediate goal.
In the next subsection, we will show that $\systole(\no_g)$ satisfies this weaker notion of convexity.
In this section, we define the notion, and explain why this weaker notion of convexity is still sufficient for the purposes of Patterson-Sullivan theory.

\begin{definition}[Weak convexity]
  A subset $S$ of a geodesic metric space $X$ is said to be $\vepd$-weak convex (for $\vepd > 0$) if there exists a constant $t > 0$ such that for any pair of points $x$ and $y$ in $S$, any geodesic path $\gamma$ joining $x$ and $y$ longer than $t$ can be homotoped to a path $\gamma^{\prime}$ joining $x$ and $y$ such that $\gamma^{\prime}$ lies entirely within $S$, and the lengths of $\gamma$ and $\gamma^{\prime}$ satisfy the following inequality.
  \begin{align*}
    \ell(\gamma^{\prime}) \leq (1 + \vepd) \ell(\gamma)
  \end{align*}
\end{definition}

\begin{remark}
  Strictly speaking, we should call a subset $(\vepd, t)$-weak convex, as the constant $t$ is part of the data that makes a set weak convex.
  However, the constant $t$ will not matter for us, so we suppress it in all mentions of weak convexity.
\end{remark}

An $\vepd$-weak convex subset is a subset which, while not entirely undistorted, has bounded distortion with respect to the ambient metric space at large enough scale.
Weak convexity also interacts well with results from Patterson-Sullivan theory.
Suppose we have a discrete group $G$ acting properly discontinuously on a metric space $X$, and let $X_{\vepd}$ be an $\vepd$-weak convex subset of $X$ upon which $G$ also acts.
If the critical exponent for the $G$ action on $X$ is $\delta$, and the corresponding exponent for $X_{\vepd}$ is $\delta_{\vepd}$, we immediately get the following estimate for $\delta$.
\begin{align*}
  \delta \leq \delta_{\vepd} (1 + \vepd)
\end{align*}

For tangent directions along which the Teichmüller geodesic stays in $\systole(\no_g)$ for arbitrarily large times, we can consider the Teichmüller geodesic flow, and reparameterize the flow speed such that the following equation holds.
\begin{align*}
  d_{\vept}(v, g_{\tau} v) = \tau
\end{align*}
If we can establish mixing results for this reparameterized geodesic flow, we can use the techniques of \textcite{roblin2003ergodicite} to count lattice points in $X_{\vepd}$, the counting results translate into estimates for lattice points in $X$.
\begin{align*}
  \#\left( \text{Lattice points in $X$ within distance $R$} \right) \leq \#\left( \text{Lattice points in $X_{\vepd}$ within distance $R(1 + \vepd)$} \right)
\end{align*}

We can get even better estimates if rather than having a single $\vepd$-weak convex subset, we have an family of subsets, such that the $\vepd$ goes to $0$, and the union of the weak convex subsets is the entire space.

\begin{definition}[Exhaustion by weak convex subsets]
  A metric space $X$ is said to be exhausted by weak convex subsets if there exists a nested family of subsets $\left\{ X_i \right\}$, such that $X_i$ is $\varepsilon_i$-weak convex, where $\varepsilon_i$ goes to $0$, and $\bigcup_{i=1}^{\infty} X_i = X$.
\end{definition}

When there is an exhaustion by weak convex subsets, one can get arbitrarily good bounds for the critical exponent for the $G$ action on $X$.

\subsection{Weak Convexity for $\systole(\no_g)$}
\label{sec:weak-conv-syst}

In this section, we will show that $\systole(\no_g)$ is an $\vepd$-weak convex subset of $\teich(\no_g)$, and that $\teich(\no_g)$ can be exhausted by the subsets $\systole(\no_g)$ as $\vept$ goes to $0$.

\begin{theorem}
  \label{thm:weak-convexity}
  For any $\vepd > 0$, there exists a $\vept > 0$ such that $\systole(\no_g)$ is a $\vepd$-weak convex subset of $\teich(\no_g)$.
\end{theorem}

The key ingredient in the proof of Theorem \ref{thm:weak-convexity} is a version of Minsky's product region theorem \cite[Theorem 6.1]{1077244446} for non-orientable surfaces, which we prove in Section \ref{sec:minskys-prod-regi}.

Let $\gamma = \left\{ \gamma_1, \ldots, \gamma_j, \ldots, \gamma_k \right\}$ be a multicurve on a non-orientable surface $\no_g$, where for $i \leq j$, $\gamma_i$ is a two-sided curve, and for $i > j$, $\gamma_i$ is a one-sided curve.
Let $X_\gamma$ denote the metric space obtained as the $\sup$-product $\teich(\no_g \setminus \gamma) \times \mathbb{H}_1 \times \cdots \times \mathbb{H}_j \times (\mathbb{R}_{>0})_{j+1} \times \cdots \times (\mathbb{R}_{>0})_k$, where the $\mathbb{H}_i$ are copies of the upper half plane with the hyperbolic metric, and $\mathbb{R}_{>0}$ is the set of positive real numbers, where the distance between $x$ and $y$ is $\left| \log \left( \frac{x}{y} \right) \right|$.
For any pants decomposition that contains $\gamma$, we consider the Fenchel-Nielsen coordinate systems associated to the pants decomposition.
We have a map $\Pi$ from $\teich(\no_g)$ to $X_\gamma$, which is called the \emph{product region projection map}.
\begin{definition}[Product region projection map]
  The product region projection map $\Pi: \teich(\no_g) \to X_\gamma$ is defined in the following manner.
  \begin{itemize}
  \item The $\teich(\no_g \setminus \gamma)$-coordinate is obtained by setting the lengths to $0$ of all the curves in $\gamma$ to get a punctured hyperbolic surface.
  \item The $\mathbb{H}_i$-coordinate is $\left( t, \frac{1}{\ell} \right)$, where $t$ is the \emph{twist} (i.e. the twist coordinate in the Fenchel-Nielsen coordinate system) of the two-sided curve $\gamma_i$, and $\ell$ is the hyperbolic length.
  \item The $(\mathbb{R}_{>0})_i$ coordinate is $\frac{1}{\ell}$, where $\ell$ is the hyperbolic length of the one-sided curve $\gamma_i$.
  \end{itemize}
\end{definition}

We define a metric on the product $X_\gamma$ as the supremum of the metrics on each of the components, where the metric on $\teich(\no_g \setminus \gamma)$ is the Teichmüller metric, the metric on the $\mathbb{H}_i$ components is the hyperbolic metric, and the metric on the $(\mathbb{R}_{>0})_i$ is given by $d(x,y) = \left| \log\left( \frac{x}{y}  \right) \right| $, i.e.\ the restriction of the hyperbolic metric in $\mathbb{H}$ to a vertical line.

We consider the restriction of $\Pi$ to the thin region of Teichmüller space, denoted $\teich_{\gamma \leq \vept}(\no_g)$, which is the region where all curves in $\gamma$ have hyperbolic length at most $\vept$.

\begingroup
\def\thetheorem{\ref{thm:prno}}
\begin{theorem}[Product region theorem for non-orientable surfaces]
  For any $c >0$, there exists  $\vept^{\prime} > 0$, such that for any $\vept < \vept^{\prime}$, the restriction of $\Pi$ to $\teich_{\gamma \leq \vept}(\no_g)$ is an isometry with additive error at most $c$, i.e.\ the following holds for any $x$ and $y$ in $\teich_{\gamma \leq \vept}(\no_g)$.
  \begin{align*}
    \left| d(x, y) - d_{X_{\gamma}}(\Pi(x), \Pi(y)) \right| \leq c
  \end{align*}

\end{theorem}
\addtocounter{theorem}{-1}
\endgroup

We can now prove Theorem \ref{thm:weak-convexity}.

\begin{proof}[Proof of Theorem \ref{thm:weak-convexity}]
  We begin by picking a small constant $\vept^{\prime} > 0$ and $\delta > 0$.
  We will fix the values of these constants at the end of the proof.
  Let $[x,y]$ be a geodesic segment that starts and ends in $\systolearg{\vept^{\prime}}(\no_g)$.
  Let $\{p_i\}$ be points on $[x,y]$, such that $x = p_0$, $d(p_i, p_{i+1}) = \delta$, and $d(p_n, y) \leq \delta$, where $p_n$ is the last of the $p_i$'s.

  The first step of our proof is modifying the path $[x,y]$ and estimating the length of the modified path.
  We do so by constructing new points $p_i^{\prime}$, where $p_i^{\prime}$ is obtained from $p_i$ by increasing the length of any one-sided curve that is shorter than $\vept^{\prime}$ to $\vept^{\prime}$.
  This ensures that the endpoints of the segments $[p_i^{\prime}, p_{i+1}^{\prime}]$ are in $\systolearg{\vept^{\prime}}(\no_g)$.
  Estimating $d(p_i^{\prime}, p_{i+1}^{\prime})$ splits up into two cases.
  \begin{enumerate}[(i)]
  \item When $p_i = p_i^{\prime}$ and $p_{i+1} = p_{i+1}^{\prime}$: In this case $d(p_i^{\prime}, p_{i+1}^{\prime}) = \delta$, by construction.
  \item When at least one of $p_i$ and $p_{i+1}$ are not equal to $p_i^{\prime}$ and $p_{i+1}^{\prime}$: In this case, we can assume without loss of generality that both $p_i \neq p_i^{\prime}$ and $p_{i+1} \neq p_{i+1}^{\prime}$.
    If that is not the case, and say $p_i \neq p_i^{\prime}$ and $p_{i+1} = p_{i+1}^{\prime}$, we replace $p_{i+1}$ with the last point $y$ on $[p_i, p_{i+1}]$ that is outside $\systolearg{\vept^{\prime}}(\no_g)$.
    The interval $[y, p_{i+1}]$ can be treated as in case (i), and we focus on $[p_i, y]$.

    We have that the interior of $[p_i, p_{i+1}]$ and $[p_i^{\prime}, p_{i+1}^{\prime}]$ both lie in the region where some one-sided curve $\gamma$ is shorter than $\vept^{\prime}$.
    We invoke Theorem \ref{thm:prno} to estimate distances in this region: we have a constant $c(\vept^{\prime})$ that depends on $\vept^{\prime}$ such that following holds.
    \begin{align}
      \label{eq:p-estimate}
      \left| d(p_i, p_{i+1}) - \sup\left( d_{\teich(\no_g \setminus \gamma)}(\Pi(p_i), \Pi(p_{i+1})), \left| \log \left( \frac{\ell_{p_i}(\gamma)}{\ell_{p_{i+1}}(\gamma)}  \right) \right| \right) \right| \leq c(\vept^{\prime})
    \end{align}
    Observe that when we replace $p_i$ by $p_i^{\prime}$ and $p_{i+1}$ by $p_{i+1}^{\prime}$, the first argument $\sup$ stays the same, and the second argument becomes $0$.
    \begin{align}
      \label{eq:p-prime-estimate}
      \left| d(p_i^{\prime}, p_{i+1}^{\prime}) - \sup\left( d_{\teich(\no_g \setminus \gamma)}(\Pi(p_i^{\prime}), \Pi(p_{i+1}^{\prime})), 0 \right) \right| \leq c(\vept^{\prime})
    \end{align}
    This leads to the following estimate for $d(p_i^{\prime}, p_{i+1}^{\prime})$.
    \begin{align}
      \label{eq:length-estimate-homotope}
      d(p_i^{\prime}, p_{i+1}^{\prime}) \leq \delta + 2c(\vept^{\prime})
    \end{align}
  \end{enumerate}
  We construct a new path $\lambda$ by joining $p_i^{\prime}$'s, and $p_n$ to $y$.
  If we let $l$ denote the length of $[x,y]$, we get the following estimate for $\ell(\lambda)$ using \eqref{eq:length-estimate-homotope}.
  \begin{align*}
    \ell(\lambda) \leq l \left( 1 + \frac{2c(\vept^{\prime})}{\delta} \right)
  \end{align*}
  We now pick a value of $\delta$ small enough such that along each of the segments $[p_i, p_{i+1}]$, there is at least one one-sided curve that stays short throughout, and then we pick $\vept^{\prime}$ small enough so that $c(\vept^{\prime})$ is small enough to make $\frac{2c(\vept^{\prime})}{\delta} < \vepd$.

  We now need to show that this new path stays within $\systole(\no_g)$ for some $\vept < \vept^{\prime}$.
  We already have that $x$, $y$ and all the $p_{i}^{\prime}$ are in $\systolearg{\vept^{\prime}}(\no_g)$ and thus in $\systole(\no_g)$.
  For the interior of the geodesic segments $[p_i^{\prime}, p_{i+1}^{\prime}]$, since the endpoints are in $\systolearg{\vept^{\prime}}(\no_g)$, and the length of the segments is no more than $\delta(1 + \vepd)$, we have that there exists some $\vept$ such that $[p_i^{\prime}, p_{i+1}^{\prime}]$ lies in $\systole(\no_g)$.

  Finally, we have to deal with geodesic segments $[w,z]$ which start or end in $\systole(\no_g) \setminus \systolearg{\vept^{\prime}}(\no_g)$.
  We do so by increasing the lengths of short one-sided curves on $w$ and $z$ to $\vept^{\prime}$ if there are any curves shorter than $\vept^{\prime}$.
  Let the modified points be $w^{\prime}$ and $z^{\prime}$: we first construct a path joining $w^{\prime}$ and $z^{\prime}$ that stays within $\systole(\no_g)$ using our construction, and then prepend that path with a path joining $w$ with $w^{\prime}$ and append a path joining $z^{\prime}$ to $z$.
  This new path joining $w$ to $z$ stays entirely withing $\systole(\no_g)$, but we now incur a fixed additive error along with our multiplicative error as well.
  However, if the path is long enough, the additive error can be absorbed in the multiplicative error, with a slightly worse constant.
  We do that, and the threshold for the path being long enough is the constant $t$ that appears in our definition of weak convexity.
  This proves the result.
\end{proof}

\begin{remark}
  We emphasize that the key step in the above proof is going from \eqref{eq:p-estimate} to \eqref{eq:p-prime-estimate}, where the $\log\left( \frac{\ell_{p_i}(\gamma)}{\ell_{p_i^{\prime}}(\gamma)}  \right)$ term becomes $0$.
  This is only possible because there cannot be any twisting around a one-sided curve $\gamma$, so the projection map that sends $p_i$ to $p_i^{\prime}$ and $p_{i+1}$ to $p_{i+1}^{\prime}$ is distance reducing.
  If one tried to use the same proof strategy to show that the thick part of $\teich(\os)$, for any orientable or non-orientable surface $\os$ is weak convex in $\teich(\os)$, the step we described would be the point of failure.
  In particular, if there's a twist along $\gamma$, going from \eqref{eq:p-estimate} to \eqref{eq:p-prime-estimate} will not be distance reducing, and will exponentially increase the distance, leading the estimate to fail.
\end{remark}

Now that we have established that $\systole(\no_g)$ is $\vepd$-weak convex, we can justifiably call it the weak convex core of $\teich(\no_g)$.
For the remainder of this paper, we fix $\vepd < \left(  \frac{1}{6g-12} \right)^2$, and a value of $\vept$ such that $\systole(\no_g)$ is $\vepd$-weak convex.

\begin{definition}[Weak convex core of $\teich(\no_g)$]
  We call $\systole(\no_g)$ the weak convex core of $\teich(\no_g)$, and denote it $\core(\teich(\no_g))$.
\end{definition}

We now also provide a partial classification of the \emph{strongly contracting} elements of $\mcg(\no_g)$ for the metric space $\systole(\no_g)$.

\begin{definition}[Strongly contracting element]
  An infinite order element $\gamma$ in $\mcg(\no_g)$ is said to be strongly contracting if there exists a $p \in \systole(\no_g)$ such that the following two conditions hold.
  \begin{enumerate}[(i)]
  \item $\left\{ \gamma^i p \right\}_{i \in \mathbb{Z}}$ quasi-isometrically embeds in $\systole(\no_g)$.
  \item For any ball of radius $R$ disjoint from $\left\{ \gamma^i p \right\}$, its closest point projection onto $\left\{ \gamma^i p \right\}$ has uniformly bounded diameter.
  \end{enumerate}
\end{definition}

\begin{lemma}[Partial classification of strongly contracting elements]
  \label{lem:strongly-contract-class}
  Let $\gamma$ be an infinite order element in $\mcg(\no_g)$.
  \begin{enumerate}[(i)]
  \item If $\gamma$ is pseudo-Anosov, then $\gamma$ is strongly contracting.
  \item If $\gamma$ leaves a two-sided curve invariant, then $\gamma$ is not strongly contracting.
  \end{enumerate}
\end{lemma}

\begin{proof}
  We deal with the two cases separately.
  \begin{description}
  \item[Case (i)] In this case, we pick $p$ to lie along the axis of the pseudo-Anosov $\gamma$.
    Passing to the orientable double cover, we have that the closest point projection of any ball disjoint from the axis has bounded diameter, by \textcite{minsky1996quasi}.
    Since $\teich(\no_g)$ isometrically embeds inside the Teichmüller space of the double cover, we have the claim for $\teich(\no_g)$.
    To show now that the result holds for $\systole(\no_g)$, observe that the induced metric on $\systole(\no_g)$ is minimally distorted from the metric on $\teich(\no_g)$, by Theorem \ref{thm:weak-convexity}.
    One of two things can happen: the axis of $\gamma$ lies in $\systole(\no_g)$, or it lies outside.
    In the first case, we project as usual, and by Theorem \ref{thm:weak-convexity}, the size of the projection increases by a bounded multiplicative factor: the metric $d_{\vept}$ differs from the Teichmüller metric $d_T$ by a multiplicative factor of $(1 + \vepd)$, and \textcite{Sisto+2018+79+114} shows that the property of being a contracting subset is a quasi-isometry invariant.
    In the second case, we first create a new axis, by increasing the lengths of all one-sided curves on the old axis to $\vept$ (while keeping all the other lengths and twists with respect to a Fenchel-Nielsen coordinate system constant): we call this map $\pi$.
    This construction gives us a new axis that lies in $\systole(\no_g)$.
    We then get a projection of a ball by first projecting to the old axis, and then composing that with the $\pi$.
    This composed map still has bounded diameter because the first projection does, and the second projection increases distances by a bounded multiplicative factor.
  \item[Case (ii)] In this case, we need to show there is no choice of $p$ such that the projection onto $\left\{ \gamma^i p \right\}$ has bounded diameter.
    Suppose there is such a $p$: we will construct a family of balls disjoint from $\gamma$ with arbitrarily large closest point projections onto $\gamma$.
    We can assume without loss of generality that at $p$, one of the invariant two-sided curves is short.
    If not, we can create a new point $p^{\prime}$ where this is the case, and the orbits $\left\{ \gamma^i p \right\}$ and $\left\{ \gamma^i p^{\prime} \right\}$ have unbounded closed point projections onto each other, and if a family of balls has unbounded closest point projections on $\left\{ \gamma^i p^{\prime} \right\}$, it will also have unbounded closest point projections on $\left\{ \gamma^i p \right\}$.

    Since we have that some two-sided curve $\kappa$ is short at $p$, and $\gamma$ leaves $\kappa$ invariant, we have that the entire orbit $\left\{ \gamma^i p \right\}$ lies in the product region associated to $\kappa$.
    In this product region, we can verify the claim via Minsky's product region theorem using a family of balls of radius $R$ centered at $q$, where $q$ is a point in the product region where the length of $\gamma$ is $\exp(-R)$ times smaller than the length of $\gamma$ at $p$.
    For each $R$, such a ball is disjoint from $\left\{ \gamma^i p \right\}$, and will have projection diameter $2R - c^{\prime}$, where $c^{\prime}$ is some constant that only depends on $\vept$ (see \autoref{fig:non-contracting}).
    \begin{figure}[h]
      \centering
    \def\svgscale{1}
    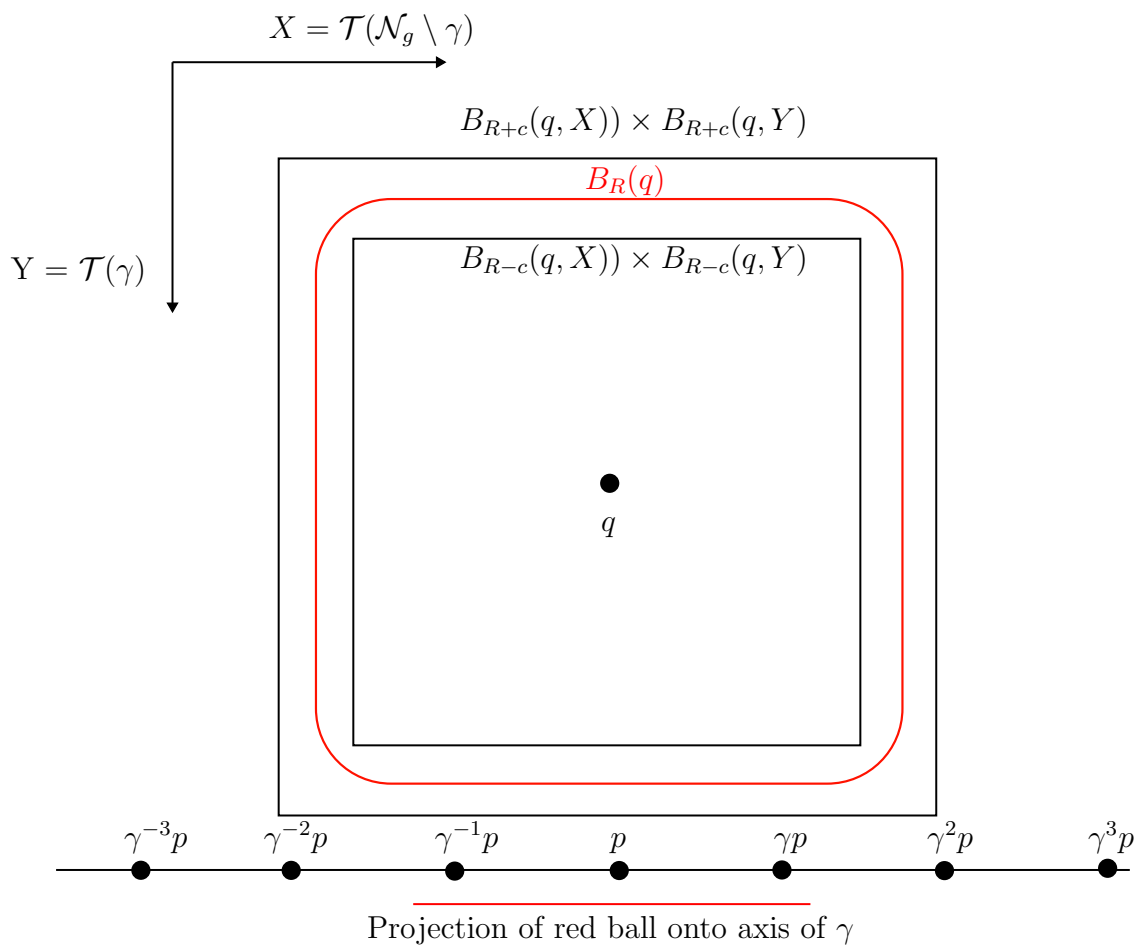

      \caption{Ball in product region with large projection onto axis.}
      \label{fig:non-contracting}
    \end{figure}
  \end{description}
  This proves the theorem in the two cases we specified.
\end{proof}

\begin{remark}
  The only case the above classification does not deal with is the case where $\gamma$ is a pseudo-Anosov on a subsurface that is the complement of only one-sided curves.
\end{remark}


%% file: images/non-contracting.pdf_tex
\begingroup%
  \makeatletter%
  \providecommand\color[2][]{%
    \errmessage{(Inkscape) Color is used for the text in Inkscape, but the package 'color.sty' is not loaded}%
    \renewcommand\color[2][]{}%
  }%
  \providecommand\transparent[1]{%
    \errmessage{(Inkscape) Transparency is used (non-zero) for the text in Inkscape, but the package 'transparent.sty' is not loaded}%
    \renewcommand\transparent[1]{}%
  }%
  \providecommand\rotatebox[2]{#2}%
  \newcommand*\fsize{\dimexpr\f@size pt\relax}%
  \newcommand*\lineheight[1]{\fontsize{\fsize}{#1\fsize}\selectfont}%
  \ifx\svgwidth\undefined%
    \setlength{\unitlength}{472.15116246bp}%
    \ifx\svgscale\undefined%
      \relax%
    \else%
      \setlength{\unitlength}{\unitlength * \real{\svgscale}}%
    \fi%
  \else%
    \setlength{\unitlength}{\svgwidth}%
  \fi%
  \global\let\svgwidth\undefined%
  \global\let\svgscale\undefined%
  \makeatother%
  \begin{picture}(1,0.73216665)%
    \lineheight{1}%
    \setlength\tabcolsep{0pt}%
    \put(0,0){\includegraphics[width=\unitlength,page=1]{non-contracting.pdf}}%
    \put(0.2560344,0.72250871){\makebox(0,0)[lt]{\lineheight{1.25}\smash{\begin{tabular}[t]{l}$X = \teich(\no_g \setminus \gamma)$\end{tabular}}}}%
    \put(0.05,0.5298685){\makebox(0,0)[lt]{\lineheight{1.25}\smash{\begin{tabular}[t]{l}Y = $\teich(\gamma)$\end{tabular}}}}%
    \put(0.52127056,0.32872956){\makebox(0,0)[lt]{\lineheight{1.25}\smash{\begin{tabular}[t]{l}$q$\end{tabular}}}}%
    \put(0.52870085,0.07632489){\makebox(0,0)[lt]{\lineheight{1.25}\smash{\begin{tabular}[t]{l}$p$\end{tabular}}}}%
    \put(0.65901816,0.07632489){\makebox(0,0)[lt]{\lineheight{1.25}\smash{\begin{tabular}[t]{l}$\gamma p$\end{tabular}}}}%
    \put(0.78155087,0.07632489){\makebox(0,0)[lt]{\lineheight{1.25}\smash{\begin{tabular}[t]{l}$\gamma^2 p$\end{tabular}}}}%
    \put(0.90995241,0.07632489){\makebox(0,0)[lt]{\lineheight{1.25}\smash{\begin{tabular}[t]{l}$\gamma^3 p$\end{tabular}}}}%
    \put(0.38858665,0.07632489){\makebox(0,0)[lt]{\lineheight{1.25}\smash{\begin{tabular}[t]{l}$\gamma^{-1} p$\end{tabular}}}}%
    \put(0.25200979,0.07632489){\makebox(0,0)[lt]{\lineheight{1.25}\smash{\begin{tabular}[t]{l}$\gamma^{-2} p$\end{tabular}}}}%
    \put(0.14048808,0.07632489){\makebox(0,0)[lt]{\lineheight{1.25}\smash{\begin{tabular}[t]{l}$\gamma^{-3} p$\end{tabular}}}}%
    \put(0.4081173,0.53868238){\makebox(0,0)[lt]{\lineheight{1.25}\smash{\begin{tabular}[t]{l}$B_{R-c}(q, X)) \times B_{R-c}(q, Y)$\end{tabular}}}}%
    \put(0.4081173,0.64860999){\makebox(0,0)[lt]{\lineheight{1.25}\smash{\begin{tabular}[t]{l}$B_{R+c}(q, X)) \times B_{R+c}(q, Y)$\end{tabular}}}}%
    \put(0.5081173,0.60181192){\makebox(0,0)[lt]{\lineheight{1.25}\smash{\begin{tabular}[t]{l}{\color{red}$B_R(q)$}\end{tabular}}}}%
    \put(0.33448193,0.00304985){\makebox(0,0)[lt]{\lineheight{1.25}\smash{\begin{tabular}[t]{l}Projection of red ball onto axis of $\gamma$\end{tabular}}}}%
  \end{picture}%
\endgroup%

%% file: thin-part.tex
\section{Geodesics in the Thin Part of $\core(\teich(\no_g))$}
\label{sec:recurr-rand-walks}

Inspired by Theorem \ref{thm:weak-convexity}, we will focus our attention on $\core(\teich(\no_g))$ instead of the entirety of $\teich(\no_g)$.
In this section, we will begin a proof of the fact that the action of $\mcg(\no_g)$ on $\core(\teich(\no_g))$ is statistically convex-cocompact (we abbreviate that to SCC for the remainder of the paper).

To show that the $\mcg(\no_g)$ action is SCC, we need to exhibit a subset of $\core(\teich(\no_g))$ which has the following two properties.
\begin{enumerate}[(i)]
\item The action of $\mcg(\no_g)$ on the subset is cocompact.
\item The subset is statistically convex.
\end{enumerate}
We claim that the subset $\thick(\no_g)$ satisfies these properties.
\begin{align*}
  \thick(\no_g) \coloneqq \left\{ z \in \teich(\no_g) \mid \text{No curve on $z$ is shorter than $\vept$} \right\}
\end{align*}

Although we have defined $\thick(\no_g)$ as a subset of $\teich(\no_g)$, it is also a subset of $\core(\teich(\no_g)) = \systole(\no_g)$: this follows from its very definition, which is a more restrictive version of the definition of $\systole(\no_g)$.
The action of $\mcg(\no_g)$ on $\thick(\no_g)$ is also cocompact, because the quotient is the thick part of the moduli space, which is known to be compact.

Showing that $\thick(\no_g)$ is a statistically convex subset of $\core(\teich(\no_g))$ requires more work.
We begin by rephrasing what it means for $\thick(\no_g)$ to be statistically convex in a form that's more convenient for our methods.

Consider the metric $d_{\vept}$ on $\core(\teich(\no_g))$, defined by the following formula.
\begin{align*}
  d_{\vept}(x, y) \coloneqq \inf\left( \ell(\lambda) \mid \text{$\lambda$ is a path in $\core(\teich(\no_g))$ joining $x$ and $y$} \right)
\end{align*}
This metric is not the same as the usual Teichmüller metric $d$, but by Theorem \ref{thm:weak-convexity}, we can make the ratio of these two metrics arbitrarily close to $1$ by picking $\vept$ small enough.

We now define lattice point entropy, and entropy for \concave lattice points.
\begin{definition}[Lattice point entropy for $\core(\teich(\no_g))$]
  Let $p$ be a point in $\thick(\no_g)$, and let $N_p(R, \vept)$ be the lattice point counting function.
  \begin{align*}
    N_p(R, \vept) \coloneqq \#\left( \gamma \in \mcg(\no_g) \mid d_{\vept}(p, \gamma p) \leq R \right)
  \end{align*}
  The lattice point entropy $\hLP(\core(\teich(\no_g)), \vept)$ is the following quantity.
  \begin{align*}
    \hLP(\core(\teich(\no_g)), \vept) \coloneqq \lim_{R \to \infty} \frac{\log N_p(R, \vept)}{R}
  \end{align*}
\end{definition}

\begin{remark}
  The lattice point entropy is a well-defined quantity since we have that $N_p(R, \vept)$ is a sub-multiplicative function, and therefore $\log N_p(R, \vept)$ is sub-additive, and the limit is well defined by Fekete's lemma.
\end{remark}

 We define a variable $s = 2 \diam \left( \thick(\no_g)/\mcg(\no_g) \right)$, and recall the definition of $s$-\concave lattice points.
\begin{definition}[\Concave lattice points]
 A lattice point $\gamma p$ is $s$-concave if some geodesic segment $\kappa$ starting in a ball of radius $s$ centered at $p$ and ending in ball of radius $t$ centered at $\gamma p$ stays outside the set $\thick(\no_g)$.

  The path obtained by joining $p$ to the starting point of $\kappa$, then following $\kappa$, and then joining the end point of $\kappa$ to $\gamma p$ is called the \emph{concavity detecting path for $\gamma p$}.
\end{definition}

\begin{definition}[Entropy for \concave lattice points]
  Let $M_p(R, \vept)$ be the counting function for \concave lattice points.
  \begin{align*}
    M_p(R, \vept) \coloneqq \#\left( \gamma \in \mcg(\no_g) \mid \text{$d_{\vept}(p, \gamma p) \leq R$ and $\gamma p$ is \concave}  \right)
  \end{align*}
  The entropy for \concave lattice points $\hLPb(\core(\teich(\no_g)), \vept)$ is the following quantity.
  \begin{align*}
    \hLPb(\core(\teich(\no_g)), \vept) \coloneqq \lim_{R \to \infty} \frac{\log M_p(R, \vept)}{R}
  \end{align*}
\end{definition}

The statistical convexity of $\thick(\no_g)$ is equivalent to the following statement, which states that the entropy for \concave lattice points is strictly lower than entropy for all lattice points.

\begin{theorem}[Statistical convexity]
  \label{thm:stat-convex}
  For $\vept > 0$ small enough, the following inequality holds.
  \begin{align*}
    \hLPb(\core(\teich(\no_g)), \vept) < \hLP(\core(\teich(\no_g)), \vept)
  \end{align*}
\end{theorem}

We prove Theorem \ref{thm:stat-convex} by constructing a random walk on $\core(\teich(\no_g))$ and proving a similar entropy gap between all random walk trajectories and the random walk trajectories that spend their time outside $\thick(\no_g)$.

\subsection{Construction of Random Walk}
\label{sec:constr-rand-walk}

Let $\mathfrak{p}$ be the projection map from $\core(\teich(\no_g))$ to $\core(\teich(\no_g)) / \mcg(\no_g)$, and $\vepn > 0$ be a fixed constant.

\begin{definition}[$(\vepn, 2\vepn)$-net in $\core(\teich(\no_g))$]
  Let $\mathfrak{M}$ be a subset of $\core(\teich(\no_g)) / \mcg(\no_g)$ satisfying the following two conditions.
  \begin{enumerate}[(i)]
  \item If $z_1$ and $z_2$ lie in $\mathfrak{M}$, then $d_{\vept}(z_1, z_2) \geq \vepn$.
  \item For any $z_1$ in $\core(\teich(\no_g)) / \mcg(\no_g)$, there exists $z_2 \in \mathfrak{M}$ such that $d_{\vept}(z_1, z_2) \leq 2 \vepn$.
  \end{enumerate}
  An $(\vepn, 2\vepn)$-net $\net$ in $\core(\teich(\no_g))$ is $\mathfrak{p}^{-1}(\mathfrak{M})$ for any subset $\mathfrak{M}$ satisfying the above conditions.
\end{definition}

The random walk is defined in terms of a net $\net$ and a parameter $\tau > 0$: we pick a starting point $r_0$ (which we call step $0$) for the random walk from one of the net points, and $r_n$ is picked uniformly at random amongst all the net points that are within distance $\tau$ of $r_{n-1}$.

We will be interested in counting the number of $n$-step trajectories of the random walk as a function of $n$ and $\tau$.
The count will also involve the exponential growth rate of the number of net points in a ball of radius $R$, which we call the \emph{net point entropy} $\hNP(\core(\teich(\no_g)), \vept)$.

\begin{definition}[Net point entropy]
  Let $K_p(R, \vept)$ be the counting function for net points, where $p \in \core(\teich(\no_g))$.
  \begin{align*}
    K_p(R, \vept) \coloneqq \#\left( y \in \net \mid d_{\vept}(p, y) \leq R \right)
  \end{align*}
  The net point entropy $\hNP(\core(\teich(\no_g)), \vept)$ is the following function defined in terms of $K_p$.
  \begin{align*}
    \hNP(\core(\teich(\no_g)), \vept) \coloneqq \lim_{R \to \infty} \frac{\log K_p(R, \vept)}{R}
  \end{align*}
\end{definition}

Note that $\hNP(\core(\teich(\no_g)), \vept)$ does not depend on the choice of the actual net, nor does it depend on the parameter $\vepn$. Two different nets with different choices of $\vepn$ will have counting functions that differ by at most a constant multiplicative term, which will not change the value of $\hNP(\core(\teich(\no_g)), \vept)$.
This follows from Proposition \ref{prop:uniform-volume-bound}: let $n_1$ and $n_2$ be the number of net points of two different nets contained in a ball $B_R(z)$ of radius $R$ centered at a point $z$.
We get a lower bound for $\nu_N(B_{R+\vepn}(z))$ by adding up the $\nu_N$ volumes of balls of radius $\vepn$ around each point in the first net.
\begin{align}
  \label{eq:volume-lower-bound-net}
  n_1 c_1(\vepn) \leq \nu_N(B_{R+\vepn}(z))
\end{align}
We get an upper bound for $\nu_N(B_{R+\vepn}(z))$ by adding up the $\nu_N$ volumes of radius $2 \vepn$ around each point in the second net.
\begin{align}
  \label{eq:volume-upper-bound-net}
   \nu_N(B_{R+\vepn}(z)) \leq n_2 c_2(2 \vepn)
\end{align}
Here, $c_1$ and $c_2$ are functions that appear in the statement of Proposition \ref{prop:uniform-volume-bound}.
Combining \eqref{eq:volume-lower-bound-net} and \eqref{eq:volume-upper-bound-net}, as well as using the fact that $c_1(\vepn)$ is positive gives us the claim.

The above argument also shows that the number of net points in a ball of radius $R$ is equal (up to multiplicative errors) to the Norbury measure $\nu_N$.
In Section \ref{sec:constr-marg-funct}, we will focus our attention on averaging functions with respect to this measure, instead of the uniform measure obtained via the net points.

We will replicate the proof of Theorem 1.2 of \textcite{eskinmirzakhani}, where they construct a random walk on a net, and use that to count \concave trajectories.
The key difficulty that comes up in our proof and which does not come up in their proof is the fact that they get an estimate for the cardinality of \concave trajectories (and therefore \concave lattice points) in terms of $\hNP(\teich(\os_g))$, which they know is the same as $\hLP(\teich(\os_g))$ (i.e.\ $6g-6$) by Theorem 1.2 of \textcite{10.1215/00127094-1548443}.

Since we are working with non-orientable surfaces, we cannot invoke Theorem 1.2 of \textcite{10.1215/00127094-1548443}, and instead need to relate $\hLP$ and $\hNP$ more directly: this is what we do in Sections \ref{sec:equal-latt-point} and \ref{sec:line-gap-compl}.

\subsection{Construction of the Foster-Lyapunov-Margulis Function}
\label{sec:constr-marg-funct}

One of the ways to show that a random walk on a non-compact space avoids the complement of a compact region with high probability is to construct a proper function $\flm$ on the space which satisfies a certain inequality when averaged over one step of the random walk.
See \textcite{EskinMozes+2022+342+361} for an exposition on the construction of these functions as well as some applications to dynamics and random walks.

\begin{definition}[Averaging operator]
  Let $\tau > 0$ be the parameter associated to the random walk, and $f$ be any real valued function $\core(\teich(\no_g))$.
  Then the action of the averaging operator $A_{\tau}$ on $f$ is given by the following formula.
  \begin{align*}
    (A_{\tau}f)(x) \coloneqq \frac{1}{\nu_N\left( B_{\tau}^{\vept}(x) \right)} \left( \int_{B_{\tau}^{\vept}(x)} f(z) d\nu_N(z) \right)
  \end{align*}
  Here, $B_{\tau}^{\vept}(x)$ is a ball of radius $\tau$ around $x$ with respect to the metric $d_{\vept}$.
\end{definition}

A Foster-Lyapunov-Margulis function is a function that has strong decay properties when the operator $A_{\tau}$ is applied to it.

\begin{definition}[Foster-Lyapunov-Margulis function]
  \label{defn:flm}
  A proper function $f$ on $\core(\teich(\no_g))$ quotiented by the $\mcg(\no_g)$-action is called a Foster-Lyapunov-Margulis function if, if there exists a polynomial $p$, and a compact subset $W_0$ of $\core(\teich(\no_g))$, and functions $b(x)$ and $c(x)$, such that $A_{\tau}f$ satisfies the following inequality.
  \begin{align*}
    (A_{\tau}f)(x) \leq c(x) f(x) + b(x)
  \end{align*}
  Furthermore, $b(x)$ is a bounded function that is supported within a compact set $W_0$, and $c(x)$ satisfies the following inequality for all $x$ outside of $W_0$.
  \begin{align*}
    c(x) \leq p(\tau) \cdot \exp(-\tau)
  \end{align*}
\end{definition}

Consider the function $\flm$, defined on $\core(\teich(\no_g))$ in terms of the length of the shortest two-sided curve on the surface.
\begin{align*}
  \flm(x) \coloneqq \sqrt{\frac{1}{\inf_{\text{$\gamma$ two-sided}}\ell_{\gamma}(x)}}
\end{align*}
This function is a proper function on $\systole(\no_g)/\mcg(\no_g)$, since the sub-level sets of this function are regions in $\systole(\no_g)$ where the hyperbolic lengths of all curves are bounded from below.

\begin{proposition}
  \label{prop:flm-is-flm}
  The function $\flm$ is a Foster-Lyapunov-Margulis function on $\core(\teich(\no_g))$ with respect to $A_{\tau}$, for large values of $\tau$.
\end{proposition}

\begin{proof}
  Let $W_0$ be the region of $\core(\teich(\no_g))$ where all two-sided curves are longer than $\vept$.
  We divide $\core(\teich(\no_g))$ into three regions, and prove the estimate for $(A_\tau \flm)(x)$ for $x$ in these three regions.
  The regions $R_1$, $R_2$ and $R_3$ are defined in the following manner.
  \begin{itemize}
  \item[-] $R_1$: The subset $R_1$ is defined in the following manner.
    \begin{align*}
      R_1 \coloneqq \left\{ x \mid \text{Shortest curve for any $z \in B_{\tau}^{\vept}(x)$ is $\gamma$ and $\ell_z(\gamma) < \vept$ for $z \in B_{\tau}^{\vept}(x)$} \right\}
    \end{align*}
    This is the set of points $x$ such that there exists a unique curve $\gamma$ which is the shortest curve at all points in $B_{\tau}^{\vept}(x)$, and the length of $\gamma$ is less than $\vept$ for all points in $B_{\tau}^{\vept}(x)$.
  \item[-] $R_2$: This subset is $R_0 \setminus R_1$, where $R_0$ is defined in the following manner.
    \begin{align*}
      R_0 \coloneqq \left\{ x \mid \text{Shortest curve for any $z \in B_{\tau}^{\vept}(x)$ shorter than $\vept$} \right\}
    \end{align*}
    $R_2$ is the region where multiple curves can be simultaneously short.
  \item[-] $R_3$: This is all of $\core(\teich(\no_g))$ with $R_1$ and $R_2$ removed.
\end{itemize}
  \begin{enumerate}[(i)]
  \item Proof for $x \in R_1$: In this case, the entire ball $B_{\tau}^{\vept}$ is contained in the product region where $\gamma$ stays short.
    By Theorem \ref{thm:prno}, there exists a constant $c(\vept)$ such that the ball $B_{\tau}^{\vept}(x)$ contains, and is contained inside a product of balls in $\core(\teich(\no_g \setminus \gamma))$ and $\mathbb{H}$ (which corresponds to length and twist around $\gamma$).
    \begin{align*}
      B_{\tau - c(\vept)}(x, \core(\teich(\no_g \setminus \gamma))) \times B_{\tau - c(\vept)}(\mathbb{H}) &\subset B_{\tau}^{\vept}(x) \\
      &\subset B_{\tau + c(\vept)}(x, \core(\teich(\no_g \setminus \gamma))) \times B_{\tau + c(\vept)}(\mathbb{H})
    \end{align*}
    Instead of computing the average of $f$ over $B_{\tau}^{\vept}(x)$, we can compute it over the product of the balls as described above.
    To do so, we need to verify that the measure on the product of the two balls is the product of the measures on the individual balls: we do this in the proof of Proposition \ref{prop:uniform-volume-bound}: specifically \eqref{eq:measure-of-product-is-product-of-measure}.
    Since the volumes of these balls grow exponentially with respect to radius, computing the average over the product of balls will give us an average that differs from the true average by a bounded multiplicative constant.
    This constant will be one of the terms that contribute to $c^{\prime}$ in Definition \ref{defn:flm}.
    Furthermore, note that the function $\flm$ is constant along the $\core(\teich(\no_g \setminus \gamma))$ component, since $\gamma$ is the shortest curve in the product of balls.
    It thus suffices to compute the average of $\flm$ on a ball in $\mathbb{H}$.

    Parameterizing $\mathbb{H}$ as the upper half plane with coordinates $z = (z_{\mathrm{real}}, z_{\mathrm{im}})$, the function $\flm(z)$ is the square root of the second coordinate, i.e.\ $\flm(z) = \sqrt{z_{\mathrm{im}}}$.
    The average of this function over a sphere is well-understood (see \cite[Lemma 4.2]{EskinMozes+2022+342+361}).
    We recall the estimate here for the reader's convenience: denoting the averaging operator over a sphere of radius $\tau$ in $\mathbb{H}$ by $B_{\tau}$, the estimate is as follows.
    \begin{align*}
      (B_{\tau}\flm)(z) \leq c^{\prime \prime} \exp(-\tau) \flm(z)
    \end{align*}
    We use the spherical average to compute the average over a ball by taking a weighted average of the spherical averages.
    Doing so gives the following estimate for $(A_{\tau}\flm)(z)$ (where $c^{\prime}$ is some fixed constant).
    \begin{align*}
      (A_{\tau}\flm)(z) \leq c^{\prime} \tau \exp(-\tau) \flm(z)
    \end{align*}
    Since we have already established that the value $\flm(x)$ only depends on depends on what happens in the $\mathbb{H}$-coordinate, namely $z$, we get a corresponding inequality for $x$, which proves the result in this case.
    \begin{align*}
      (A_{\tau}\flm)(x) \leq c^{\prime} \tau \exp(-\tau) \flm(x)
    \end{align*}
  \item Proof for $x \in R_2$: In this case, let $\left\{ \gamma_1, \ldots, \gamma_k \right\}$ be the two-sided curves that become shorter at some point in $B_{\tau}^{\vept}(x)$.
    We have that the ball lies in the product region where all the curves $\left\{ \gamma_1, \ldots, \gamma_k \right\}$ are short simultaneously.
    We have that there exists some constant $c_g$, depending only on $g$, such that $k \leq c_g$, since we cannot have too many curves being short simultaneously.

    Similar to the previous case, changing only the $\core(\teich(\no_g \setminus \bigcup_{i=1}^k \gamma_i))$ coordinate will not change the value of the function $\flm$, so it suffices to focus our attention on the coordinates $\prod_{i=1}^k \mathbb{H}_i$, where each $\mathbb{H}_i$ corresponds to the length and twist around $\gamma_i$.

    Let $z_{i, \mathrm{im}}$ be the imaginary part of the $i$\textsuperscript{th} copy of $\mathbb{H}$ in $\prod_{i=1}^k \mathbb{H}_i$.
    The function $\flm$ on $\prod_{i=1}^k \mathbb{H}_i$ is given by the following formula.
    \begin{align*}
      \flm(x) = \max_{i} \sqrt{z_{i, \mathrm{im}}}
    \end{align*}
    Since averaging this function over a product of balls is somewhat tedious, we relate it to a different function $\flm^{\prime}$ that is easier to average.
    \begin{align*}
      \flm^{\prime}(x) \coloneqq \sum_i \sqrt{z_{i, \mathrm{im}}}
    \end{align*}
    These two functions are equal, up to a constant multiplicative error.
    \begin{align*}
      \frac{\flm^{\prime}(x)}{c_g} \leq \flm(x) \leq \flm^{\prime}(x)
    \end{align*}
    This means we can prove the averaging estimate for $\flm^{\prime}$, and the same estimate will hold for $\flm$, with a slightly worse multiplicative constant.

    Furthermore, since $z_{i, \mathrm{im}}$ is constant along balls in $\mathbb{H}_j$ for $j \neq i$, it suffices to average just each term of the sum in the corresponding $\mathbb{H}_i$.
    We do so, using the same estimate from the proof in the $R_1$ case.
    \begin{align*}
      (A_{\tau}\flm^{\prime})(x) \leq c^{\prime} \tau \exp(-\tau) \flm^{\prime}(x)
    \end{align*}
    Replacing $\flm^{\prime}$ with $\flm$ gives us the inequality we want, and proves the result in this case.
    \begin{align*}
      (A_{\tau}\flm)(x) \leq  (c_g \cdot c^{\prime}) \tau \exp(-\tau) \flm(x)
    \end{align*}
  \item Proof for $x \in R_3$: Note that the region $R_3$ is compact, which means the function $\flm$ is bounded in this region, and consequently, there exists a uniform upper bound for $A_\tau \flm$ as well.
    Let us denote the uniform upper bounding function by $b(x)$: we can modify this function to be compactly supported by multiplying it with a bump function that is $1$ on a $\tau$-neighbourhood of $R_3$, and decays to $0$ outside.
    By construction, we have for $x \in R_3$, $(A_{\tau}\flm)(x) \leq b(x)$.
    This proves the result for $x \in R_3$.
  \end{enumerate}
  Putting together the estimate from the three cases, we get the standard form of the inequality (which holds for any $x \in \core(\teich(\no_g))$).
  \begin{align*}
    (A_{\tau} \flm)(x) \leq c(x) \flm(x) + b(x)
  \end{align*}
  Here, $c(x) \coloneqq (c_g \cdot c^{\prime}) \tau \exp(-\tau)$, and $b(x)$ is the function from the proof in case $R_3$.
\end{proof}

\subsection{Recurrence for Random Walks and Geodesic Segments}
\label{sec:recurr-rand-walks-1}


\subsubsection{Recurrence for random walks}
\label{sec:recurr-rand-walks-2}

In this section, we will count the number of random walk trajectories that are $\mathfrak{s}$-\emph{concave}, where $\mathfrak{s} = \lceil \frac{\diam(\thick(\no_g))}{\tau} \rceil + 1$.

\begin{definition}[\Concave trajectories]
  A trajectory $\left( r_0, r_1, \ldots, r_{n-1} \right)$ is said to be $\mathfrak{s}$-\concave if all the points in the trajectory except the first $\mathfrak{s}$ points and the last $\mathfrak{s}$ points are at least $\tau$-distance away from $\thick(\no_g)$.

  We call these middle points the \emph{concave trajectory points}.
\end{definition}
From Proposition \ref{prop:flm-is-flm}, we have that for any of the \concave trajectory points $r_i$, the following decay estimate for $A_{\tau}\flm(r_i)$ holds.
\begin{align*}
  (A_{\tau} \flm)(r_i) \leq c^{\prime} \tau \exp(-\tau) \flm(r_i)
\end{align*}
If $r_i$ is not a \concave trajectory point, we only have a weaker estimate in terms of a uniformly bounded function $b(x)$.
\begin{align*}
  (A_{\tau} \flm)(r_i) \leq b(x)
\end{align*}


Fix an $r_0$ in $\thick(\no_g)$, and let $\traj(n, \tau)$ denote the collection of $n$-step \concave random walk trajectories which start at $r_0$.
We will use the term $r$ to denote trajectories in $\traj(n, \tau)$, and $r_i$ to denote the $i$\textsuperscript{th} step of the trajectory $r$.

\begin{proposition}
  \label{prop:rw-recurrence}
  For any $\veperr > 0$, there exists a $\tau > 0$ large enough, and a constant $C \gg 0$,  such that the following bound on $\left| \traj(n, \tau) \right|$ holds.
  \begin{align*}
    \left| \traj(n, \tau) \right| \leq C \exp((\hNP - 1 + \veperr)n\tau)
  \end{align*}
  Here, $\hNP = \hNP(\core(\teich(\no_g)), \vept)$.
\end{proposition}

\begin{proof}
  Since the $\flm(x)$ has a positive lower bound $C_l$ as $x$ varies over $\core(\teich(\no_g))$ (which comes from the Bers constant associated to $\no_g$), we can estimate $\left| \traj(n, \tau) \right|$ by summing up $\flm(r_{n-1})$ over all the trajectories in $\traj(n, \tau)$.
  \begin{align*}
   \left| \traj(n, \tau) \right| \leq \frac{1}{C_l} \sum_{r \in \traj(n, \tau)} \flm(r_{n-1})
  \end{align*}
  It therefore will suffice to estimate $\sum_{r \in \traj(n, \tau)} \flm(r_{n-1})$: we do so by conditioning on the previous step of the random walk over and over again until we get to the first step $r_0$.

  We have the following recursive inequality for $\sum_{r \in \traj(n, \tau)} \flm(r_{n-i})$.
  \begin{align}
    \sum_{r \in \traj(n-i+1, \tau)} \flm(r_{n-1}) &= \sum_{r \in \traj(n-i, \tau)} \left( \sum_{\substack{y \in \net \\  d_{\vept(y, r_{n-i-1}) \leq \tau} } } \flm(y) \right) \label{eq:recursive-ineq1} \\
                                              &\leq \sum_{r \in \traj(n-i, \tau)} C \int_{B_{\tau}^{\vept}(r_{n-i-1})} \flm(y) d\nu_N(y) \label{eq:recursive-ineq2} \\
    &= \sum_{r \in \traj(n-i, \tau)} C \nu_N(B_{\tau}^{\vept}(r_{n-i-1})) (A_{\tau}\flm)(r_{n-i-1})
  \end{align}
  Here, we go from \eqref{eq:recursive-ineq1} to \eqref{eq:recursive-ineq2} by integrating the indicator function supported in a ball of radius $\frac{\vepn}{2}$ around each net point, and using Proposition \ref{prop:uniform-volume-bound} to uniformly bound the integral of the indicator independent of the basepoint.

  If $n-i-1$ is one of the first $\mathfrak{s}$ or last $\mathfrak{s}$ indices, we have the following inequality.
  \begin{align}
    \label{eq:flm-bad}
    (A_{\tau}\flm)(r_{n-i-1}) \leq b(r_{n-i-1}) \flm(r_{n-i-1})
  \end{align}

  Otherwise $r_{n-i-1}$ is a \concave trajectory point and we have strong bounds on $(A_{\tau} \flm)(r_{n-i-1})$.
  \begin{align}
    \label{eq:margulis-estimate-strong-bound}
    (A_{\tau}\flm)(r_{n-i-1}) \leq c^{\prime} \tau \exp(-\tau) \flm(r_{n-i-1})
  \end{align}
  Combining \eqref{eq:recursive-ineq1} and \eqref{eq:flm-bad} for the case of non concave trajectory points, we get an estimate we need to repeat $2 \mathfrak{s}$ times.
  \begin{align}
    \label{eq:to-iterate-bad}
    \sum_{r \in \traj(n-i+1, \tau)} \flm(r_{n-i}) \leq B \exp((\hNP + \veperr^{\prime})\tau) \left( \sum_{r \in \traj(n-i, \tau)} \flm(r_{n-i-1}) \right)
  \end{align}
  Here, $B$ is the maximum value the function $b$ takes over its compact support.

  Combining \eqref{eq:recursive-ineq1} and \eqref{eq:margulis-estimate-strong-bound}, we get the estimate we need to repeat $n - 2\mathfrak{s}$ times.
  \begin{align}
    \label{eq:to-iterate}
    \sum_{r \in \traj(n-i+1, \tau)} \flm(r_{n-i}) \leq C^{\prime} \tau \exp(-\tau) \exp((\hNP + \veperr^{\prime})\tau) \left( \sum_{r \in \traj(n-i, \tau)} \flm(r_{n-i-1}) \right)
  \end{align}
  Here, we upper bound the $\nu_N(B_{\tau}^{\vept}(r_{n-2}))$ with $C^{\prime \prime} \exp((\hNP + \veperr^{\prime})\tau)$, where $\veperr^{\prime}$ is a constant smaller than $\veperr$, and $C^{\prime \prime}$ is some large constant.
  The term $C^{\prime}$ is equal to $C C^{\prime \prime} c^{\prime}$.

  We now iterate \eqref{eq:to-iterate-bad} $2\mathfrak{s}$ times and \eqref{eq:to-iterate} $n-2\mathfrak{s}$ times.
  \begin{align*}
    \sum_{r \in \traj(n, \tau)} \flm(r_{n-1}) &\leq (Be)^{2\mathfrak{s}} \cdot \left( C^{\prime} \tau \right)^{n- 2 \mathfrak{s}} \exp((\hNP - 1 + \veperr^{\prime})n\tau) \flm(r_0) \\
                                              &= \flm(r_0) \\
                                              &\cdot \exp\left( \left( \hNP - 1 + \veperr^{\prime} + \frac{\log\left( C^{\prime} \tau  \right)\left(  \frac{n- 2\mathfrak{s}}{n} \right) + \log\left( Be \right)\left( \frac{2\mathfrak{s}}{n} \right)}{\tau} \right) n\tau \right)
  \end{align*}
  By picking $\tau$ large enough so that $\mathfrak{s} \leq 4$ and  $\veperr^{\prime} + \frac{\log\left( C^{\prime} \tau  \right)\left(  \frac{n- 2\mathfrak{s}}{n} \right) + \log\left( Be \right)\left( \frac{2\mathfrak{s}}{n} \right)}{\tau} < \veperr$, we get the claimed result.
\end{proof}

Proposition \ref{prop:rw-recurrence} tells us that a random walk on $\core(\teich(\no_g))$ is biased away from the thin part of $\core(\teich(\no_g))$.
It does so by proving strong upper bounds on the probability that a random walk trajectory with $n$ steps stays in the thin part is less than $\exp((-1 + \veperr)n\tau)$: in other words, a random walk returns to $\thick(\no_g)$ with high probability.

\subsubsection{Why the random walk approach fails for $\teich(\no_g)$}
\label{sec:why-approach-fails}

If we wanted to make Proposition \ref{prop:rw-recurrence} work on $\teich(\no_g)$, we would need to similarly show the random walk on $\teich(\no_g)$ is recurrent in a similarly strong sense: i.e.\ the probability of a length $n$ trajectory staying in the thin part decays exponentially in $n$.
A consequence of this requirement is that the expected return time to the thick part is finite.

Unlike $\core(\teich(\no_g))$, $\teich(\no_g)$ has two kinds of thin regions.
\begin{itemize}
\item[-] Thin region where only two-sided curves get short.
\item[-] Thin region where some one-sided curve also gets short.
\end{itemize}

It is the second kind of thin region that poses a problem for $\teich(\no_g)$.
Minsky's product region theorem (Theorem \ref{thm:prno}) tells us that up to additive error, the metric on these thin regions looks like a product of metrics on some copies of $\mathbb{R}$ (corresponding to the one-sided short curves), some copies of $\mathbb{H}$ (corresponding to the two-sided short curves), and a Teichmüller space of lower complexity.
Since the random walk is controlled by the metric, the random walk on this product metric space is a product of random walks on each of the components.

In particular, the random walk on the $\mathbb{R}$ component is a symmetric random walk on a net in $\mathbb{R}$: i.e.\ a symmetric random walk on $\mathbb{Z}$.
Symmetric random walks on $\mathbb{Z}$ are known to be recurrent, but only in a weak sense: they recur to compact subsets infinitely often, but the expected return time is unbounded.

This means we cannot hope to prove exponentially decaying upper bounds on the probability that a long random walk trajectory stays in the thin part, since that would lead to finite expected return times.
This is why the random walk approach fails for $\teich(\no_g)$.

\subsubsection{Recurrence for geodesic segments}
\label{sec:recurr-geod-segm}

In this section, we reduce the problem of counting geodesic segments that travel in the thin part to counting trajectories of random walks that do the same.

\begin{proposition}
  \label{prop:counting-geodesics}
  For any $\veperr > 0$,
  there exists a constant $C^{\prime}$, and a large enough $R$, such that the following estimate holds for the counting function $M_{r_0}(R)$.
  \begin{align*}
    M_{r_0}(R) \leq C^{\prime} \exp((\hNP - 1 + \veperr)R)
  \end{align*}
  Here, $M_{r_0}(R)$ is the number of \concave lattice points in a ball of radius $R$ centered at $r_0$, where $r_0$ is a point in $\thickarg{\vept^{\prime}}(\no_g)$ at which we start our random walk, and $\hNP = \hNP(\core(\teich(\no_g)), \vept)$.
\end{proposition}

\begin{proof}
  We first check if $\tau$ we picked in the proof of Proposition \ref{prop:rw-recurrence} satisfies $\displaystyle \frac{2\vepn}{\tau} < \frac{\veperr}{2}$: if not, we pick a larger $\tau$.

  Let $\gamma$ be a mapping class such that $\gamma p$ is a concave lattice point.
  We consider now the concavity detecting path for $\gamma p$: recall that this is a path that starts at $p$, and ends at $\gamma p$, and the middle segment obtained by deleting a prefix of length $2 \cdot \diam \left( \thick(\no_g)/\mcg(\no_g) \right)$ and a suffix of length $2 \cdot \diam \left( \thick(\no_g)/\mcg(\no_g) \right)$ stays outside $\thick(\no_g)$.
  We turn this path into an $\mathfrak{s}$-concave random walk trajectory by marking off points at distance $\tau\left( 1 - \frac{2\vepn}{\tau} \right)$ on the segment, and then replacing those points with the nearest net point.
  All but the first $\mathfrak{s}$ and the last $\mathfrak{s}$ points in the trajectory lie outside $\thick(\no_g)$.
  Furthermore, the distance between the adjacent points on the trajectory are at most $\tau$.
  The number of steps in this trajectory is $\displaystyle n \coloneqq \left\lceil \frac{R}{\tau} \right\rceil$.

  Let $\mathcal{P}$ denote the collection of trajectories obtained via this construction.
  We apply Proposition \ref{prop:rw-recurrence} to count the number of such trajectories.
  \begin{align}
    \# \mathcal{P} &\leq C \exp\left(\left(\hNP - 1 + \frac{\veperr}{2}\right) n\tau\right) \\
                   &\leq C \exp\left(\left(\hNP - 1 + \frac{\veperr}{2}\right) (R + \tau)\right) \label{eq:traj-bound-4}
  \end{align}

  We now determine how many different geodesic segments can map to the same random walk trajectory.
  If two geodesic segments $[r_0, \gamma_1 r_0]$ and $[r_0, \gamma_2 r_0]$ map to the same random walk trajectory, we must have that they fellow travel for most of their length, and as a result, $d_{\vept}(r_0, \gamma_2^{-1} \gamma_1 r_0)$ is bounded above by a constant value that only depends on $\tau$.
  Combining the above fact with \eqref{eq:traj-bound-4} gives us a constant $C^{\prime}$ such that the following bound on $M_{r_0}(R)$ holds.
  \begin{align*}
    M_{r_0}(R) &\leq C^{\prime} \exp\left(\left(\hNP - 1 + \frac{\veperr}{2}\right) (R + \tau)\right) \\
    &= C^{\prime} \exp\left(\left(\hNP - 1 + \frac{\veperr}{2}\right) \left(1 + \frac{\tau}{R}\right) (R)\right)
  \end{align*}
  Picking a value of $R$ large enough gives us the result.
\end{proof}

We can now tie all of these calculations together to state our results on statistical convexity of $\core(\teich(\no_g))$.
Proposition \ref{prop:counting-geodesics} gives us an upper bound on $\hLPb(\core(\teich(\no_g)), \vept)$ (by applying the result for smaller and smaller values of $\veperr$).
\begin{align}
  \label{eq:bad-entropy-bound}
  \hLPb(\core(\teich(\no_g)), \vept) \leq \hNP(\core(\teich(\no_g)), \vept) - 1
\end{align}

To prove Theorem \ref{thm:stat-convex}, it will suffice to prove the following equality relating the lattice point entropy and net point entropy.
\begin{align}
  \label{eq:net-and-lattice-condition}
  \hNP(\core(\teich(\no_g)), \vept) - 1 < \hLP(\core(\teich(\no_g)), \vept)
\end{align}

For convenience, we also define the undistorted versions of these entropy terms, using the Teichmüller metric $d$ rather than the induced metric $d_{\vept}$.

\begin{definition}[(Undistorted) lattice point entropy for $\teich(\no_g)$]
  Let $p$ be a point in $\thick(\no_g)$, and let $N_p(R)$ be the lattice point counting function.
  \begin{align*}
    N_p(R) \coloneqq \#\left( \gamma \in \mcg(\no_g) \mid d(p, \gamma p) \leq R \right)
  \end{align*}
  The lattice point entropy $\hLP(\teich(\no_g))$ is the following quantity.
  \begin{align*}
    \hLP(\teich(\no_g)) \coloneqq \lim_{R \to \infty} \frac{\log N_p(R)}{R}
  \end{align*}
\end{definition}

\begin{definition}[(Undistorted) net point entropy]
  Let $K_p(R, \vept)$ be the counting function for net points, where $p \in \core(\teich(\no_g))$.
  \begin{align*}
    K_p(R, \vept) \coloneqq \#\left( y \in \net \mid d(p, y) \leq R \right)
  \end{align*}
  The net point entropy $\hNP(\core(\teich(\no_g)))$ is the following function defined in terms of $K_p$.
  \begin{align*}
    \hNP(\core(\teich(\no_g))) \coloneqq \lim_{R \to \infty} \frac{\log K_p(R, \vept)}{R}
  \end{align*}
\end{definition}

Note that the net point entropy does not depend on the precise value of $\vept$, even though it is counting net-points in $\core(\teich(\no_g))$, since the different values of $\vept$ change the counting function by a multiplicative term that does not depend on $R$.

Recall now Theorem \ref{thm:weak-convexity}, which for any $\vepd > 0$, provides a $\vept > 0$ such that the ratio of $d_{\vept}$ and $d$ is bounded above by $1 + \vepd$.
A consequence of this is that the distorted and the undistorted versions of the entropy terms differ by at most $\hNP(\core(\teich(\no_g))) \cdot \vepd$ and $\hLP(\teich(\no_g)) \cdot \vepd$.

In particular, if we show $\hNP(\core(\teich(\no_g))) = \hLP(\teich(\no_g))$, \eqref{eq:net-and-lattice-condition} will follow (for small enough $\vept$), and so will Theorem \ref{thm:stat-convex}.
We package up this result as a theorem, which we will use in subsequent sections.

\begin{theorem}
  \label{thm:entropy-equality-implies-scc}
  If $\hNP(\core(\teich(\no_g))) = \hLP(\teich(\no_g))$, then $\thick(\no_g)$ is statistically convex, and the action of $\mcg(\no_g)$ on $\core(\teich(\no_g))$ is statistically convex-cocompact.
\end{theorem}


%% file: entropy-equality.tex
\section{Equality of Lattice Point Entropy and Net Point Entropy}
\label{sec:equal-latt-point}


In this, and the following section, we will prove that $\hLP = \hNP$, which will let us apply Theorem \ref{thm:entropy-equality-implies-scc} to conclude that the $\mcg(\no_g)$ action on $\core(\teich(S))$ is SCC for surfaces $S$ of finite type.
\begin{theorem}[Entropy equality]
  \label{thm:entropy-equality}
  For any surface $S$ of finite type, the following relationship holds between the $\hNP$ and $\hLP$.
  \begin{align*}
    \hNP(\core(\teich(S)), \vept) = \hLP(\teich(S), \vept)
  \end{align*}
\end{theorem}

\begin{remark}
  In the case where $S$ is an orientable surface, the theorem is a corollary of \textcite[Theorem 1.2]{10.1215/00127094-1548443}.
  However, the proof of the stronger theorem in the orientable setting uses facts about the dynamics of the geodesic flow on the moduli space, which we don't have in the non-orientable setting.
  The proof of the weaker theorem only uses coarse geometric methods, and works equally well for orientable and non-orientable surfaces.
\end{remark}

\subsection{Base Case}
\label{sec:base-case}

We will prove this theorem by inducting on the Euler characteristic of the surface $S$.
The $4$ base cases we need to check are the $3$ non-orientable surfaces, and one orientable surface with Euler characteristic $-1$.
\begin{itemize}
\item $\os_{1,1,0}$: This is the torus with $1$ boundary component and $0$ crosscaps attached.
\item $\os_{1,0,1}$: This is a torus with $0$ boundary components, and $1$ crosscap attached.
\item $\os_{0,2,1}$: This is a sphere with $2$ boundary components, and $1$ crosscap attached.
\item $\os_{0,1,2}$: This is a sphere with $1$ boundary component, and $2$ crosscaps attached.
\end{itemize}


\begin{lemma}[Entropy equality: base case]
  \label{lem:entropy-equality-base-case}
  For a surface $S$ in $\left\{ \os_{1,1,0}, \os_{1,0,1}, \os_{0,2,1}, \os_{0,1,2} \right\}$, the following relationship holds between the $\hNP$ and $\hLP$.
  \begin{align*}
    \hNP(\core(\teich(S)), \vept) = \hLP(\teich(S), \vept)
  \end{align*}
\end{lemma}

\begin{proof}
  For $S = \os_{1,1,0}$, we will directly prove the lemma, and for the remaining three non-orientable surfaces, we will use a description of their Teichmüller spaces and mapping class groups from \textcite{gendulphe2017whats} to reduce to the first case, or show that the result follows trivially.
  \begin{itemize}
  \item $\os_{1,1,0}$: Since $\os_{1,1,0}$ is orientable, we have that $\mathrm{core(\teich(\os_{1,1,0}))} = \teich(\os_{1,1,0})$, so it suffices to look at the full Teichmüller space.
    The Teichmüller space of $\os_{1,1,0}$ is the upper half plane $\HH^2$, and the mapping class group is $\mathrm{SL}(2, \mathbb{Z})$.
    In this case, the number of lattice points in a ball of radius $R$ grows like $\exp(R)$.
    More precisely, we have the following inequality for some constants $c$ and $c^\prime$.
    \begin{align}
      c \leq \frac{\#\left( B_R(p) \cap p \cdot \mathrm{SL}(2, \mathbb{Z}) \right)}{\exp(R)} \leq c^{\prime} \label{eq:lattice-point-count}
    \end{align}
    Here, $p$ is a lattice point, and $B_R(p)$ is the ball of radius $R$ centered at $p$.

    To count the net points in the ball of radius, we parameterize the net points by how far from the orbit of $p$ they lie. Since we're looking for net points in a ball of radius $R$, the furthest away they can be from the orbit is $R$.
    We have the following sum decomposition (for an arbitrary choice of $\vepb > 0$) for the cardinality of the net points.
    \begin{align}
      \label{eq:net-point-decomposition-base-case}
      \#\left( B_R(p) \cap \net \right) = \#\left( B_R(p) \cap \net_{\leq \vepb R} \right) + \#\left( B_R(p) \cap \net_{> \vepb R} \right)
    \end{align}
    Here, $\net_{\leq \vepb R}$ denotes the net points that lie within distance $\vepb R$ of the orbit of $p$, and $\net_{> \vepb R}$ denotes the net points that lie more than distance $\vepb R$ of the orbit of $p$.

    We will show that the first term is at most $p(R) \exp(R(1 + \vepb))$, for some polynomial $p(R)$, and that the second term grows slower than the first term.
    Since the choice of $\vepb$ was arbitrary, this will prove the equality of the two entropy terms.

    Let $\net_{\leq \vepb R}(\gamma)$ denote the subset of $\net_{\leq \vepb R}$ whose closest lattice point is $\gamma p$.
    Observe that $d(p, \gamma p)$ is at most $R + \vepb R$, by the triangle inequality.
    We also have the following inequality for any $\gamma$, and for some polynomial $p$, by Lemma \ref{lem:fd-polynomial-growth}.
    \begin{align*}
      \#\left( B_R(p) \cap \net_{\leq \vepb R}(\gamma) \right) \leq p(R)
    \end{align*}
    Using the two facts we stated, we get the following upper bound for $\#\left( B_R(p) \cap \net_{\leq \vepb R} \right)$.
    \begin{align*}
      \#\left( B_R(p) \cap \net_{\leq \vepb R} \right) \leq p(R) \cdot \left( \exp(R(1+\vepb)) \right)
    \end{align*}
    This is precisely the bound we needed for the first term in \eqref{eq:net-point-decomposition-base-case}.

    Now we show that the second term of \eqref{eq:net-point-decomposition-base-case} grows slower than $\exp(R(1- \frac{ \vepb}{2}))$.
    For any point $x$ in $\net_{> \vepb R}$, we can replace the geodesic $[p, x]$ with two shorter segments, $[p, x_0]$ and $[x_0, x]$, where $x_0$ is the net point closest to the last point on the $[p, x]$ which stays within some bounded distance of a lattice point.
    We also have that $d(p, x_0) \leq R(1-\vepb)$, by our assumption, which lets us count the number of such points $x_0$.
    There are at most $\exp(R(1-\vepb))$ such points.
    Now we fix an $x_0$, and we need to estimate the number of possibilities for $x$, given that $[x_0, x]$ stays entirely within the thin part of $\mathrm{SL}(2, \mathbb{R})/\mathrm{SL}(2, \mathbb{Z})$.
    Note that this reduces to estimating the volume of the intersection of a ball $B_{\vepb R}(x_0)$ with a horoball $H$ which has $x_0$ in its boundary.
    Working in the upper half plane model for $\mathbb{H}$, where $x_0 = i$, and the region $H$ is the set of points whose imaginary component is greater than $1$, we get that the region of integration is contained in a rectangle, bounded by $-C \exp\left( \frac{\vepb R}{2} \right) \leq \mathrm{Re}(z) \leq C \exp\left( \frac{\vepb R}{2} \right)$ and $1 \leq  \mathrm{Im}(z) \leq \exp(\vepb R)$, where $C$ is some fixed constant that we do not explicitly write down.
    The volume of this region is given by the following integral.
    \begin{align}
      \label{eq:horocycle-np-lp-is-half}
      \mathrm{Vol}(B_R(x_0) \cap H) &\leq \int_{1}^{\exp(\vepb R)} \int_{-C \exp\left( \frac{\vepb R}{2} \right)}^{C \exp\left( \frac{\vepb R}{2} \right)} \frac{1}{y^2} dx dy \\
      &\leq C^{\prime} \exp\left( \frac{\vepb R}{2} \right)
    \end{align}

    We thus have the following upper bound on $\#\left( B_R(p) \cap \net_{> \vepb R} \right) $ for large enough values of $R$.
    \begin{align*}
      \#\left( B_R(p) \cap \net_{> \vepb R} \right)  &\leq \exp\left( \frac{\vepb R}{2} \right) \cdot \exp(R(1-\vepb)) \\
                                                        &\leq \exp\left(R\left(1 - \frac{\vepb}{2}\right)\right)
    \end{align*}
    This finishes proving the two claims we made about the terms of \eqref{eq:net-point-decomposition-base-case}, and proves the result for $\os_{1,1,0}$.

  \item $\os_{1,0,1}$: This surface is very similar to the previous case: it's obtained by gluing together the boundary component of $\os_{1,1,0}$ via the antipodal map.
    It's a theorem of \textcite{scharlemann1982complex} and also \textcite{gendulphe2017whats} that there is a unique one-sided curve $\kappa$ in $\os_{1,0,1}$ whose complement is $\os_{1,1,0}$.
    As a consequence, $\mcg(\os_{1,0,1}) \cong \mcg(\os_{1,1,0})$, and $\teich(\os_{1,1,0}) \hookrightarrow \teich(\os_{1,0,1})$, where the inclusion map is given by considering a point in $\teich(\os_{1,1,0})$, where the boundary component has length $\vept$, and gluing it via the antipodal map to get a point in $\teich(\os_{1,0,1})$.
    The inclusion map is also equivariant with respect to the action of $\mcg(\os_{1,1,0})$ and $\mcg(\os_{1,0,1})$.

    We consider now $\core(\teich(\os_{1,0,1}))$: the curve $\kappa$ cannot get shorter than the threshold specified by the core.
    We now show that $\kappa$ cannot be arbitrary long either.
    At any point $z \in \core(\teich(\os_{1,0,1}))$, let $\kappa^{\prime}$ be the shortest curve that intersects $\kappa$ exactly once (see \autoref{fig:base-case-2}).
    We have an upper bound for the length of $\kappa^{\prime}$: namely the length of the orthogeodesic arc on $\os_{1,0,1} \setminus \kappa$ that starts and ends at $\kappa$.
    It follows from hyperbolic trigonometry that if the length of $\kappa$ goes to $\infty$, then the length of the orthogeodesic, hence the length of $\kappa^{\prime}$ will approach $0$, the point in $\teich(\os_{1,0,1})$ will leave $\core(\teich(\os_{1,0,1}))$.

    If we consider the pants decomposition of the surface along $\kappa$, and any two sided curve, we see that the length coordinates of $\kappa$ in $\core(\teich(\os_{1,0,1}))$ are contained in a compact interval $[t_1, t_2]$, where $t_1 > 0$.
    This means that $\core(\teich(\os_{1,0,1}))$ is a bounded neighbourhood of the image of $\teich(\os_{1,1,0})$.
    \begin{figure}[h]
      \centering
    \def\svgscale{1}
    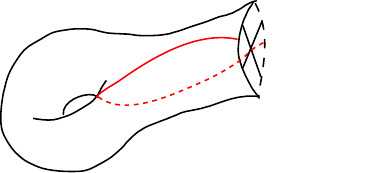

      \caption{The curves $\kappa$ and $\kappa^{\prime}$ on $\os_{1,0,1}$.}
      \label{fig:base-case-2}
    \end{figure}

    From the previous case, we already have $\hNP(\core(\teich(\os_{1,1,0})), \vept) = \hLP(\teich(\os_{1,1,0}), \vept)$, and since their mapping class groups are isomorphic, we also have $\hLP(\teich(\os_{1,1,0}), \vept) = \hLP(\teich(\os_{1,0,1}), \vept)$.
    We now need to prove that $\hNP(\core(\teich(\os_{1,1,0})), \vept) = \hNP(\core(\teich(\os_{1,0,1})), \vept)$ to prove the result for this case.
    We have that the net for $\core(\teich(\os_{1,0,1}))$ lies in a bounded neighbourhood of the net for $\core(\teich(\os_{1,1,0}))$: this implies that the cardinalities of the net points in a ball of radius $r$ differ by at most a multiplicative constant.
    \begin{align*}
      \#\left( B_R(p) \cap \net_{\core(\teich(\os_{1,0,1}))} \right) \leq c \cdot \#\left( B_R(p) \cap \net_{\core(\teich(\os_{1,1,0}))} \right)
    \end{align*}
    Since the two cardinalities differ by at most a multiplicative constant, they have the same exponential growth rate.
  \item $\os_{0,2,1}$: The mapping class group of this surface is finite: in fact, it is isomorphic to $\mathbb{Z}/2\mathbb{Z} \times \mathbb{Z}/2\mathbb{Z}$ (see \textcite{gendulphe2017whats}).
    This means $\hLP(\teich(\os_{0,2,1}), \vept) = 0$.
    This surface has exactly two simple geodesics $\kappa$ and $\kappa^{\prime}$, which intersect each other exactly once, such that deleting either one of them results in a pair of pants (see \autoref{fig:s_021}).
    \begin{figure}[h]
      \centering
    \def\svgscale{0.9}
    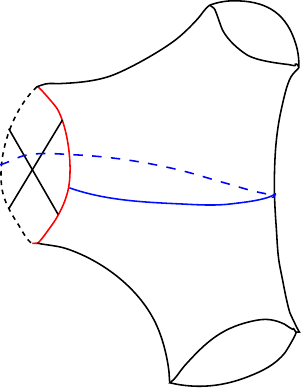

      \caption{The curves $\kappa$ and $\kappa^{\prime}$ on $\os_{0,2,1}$.}
      \label{fig:s_021}
    \end{figure}

    Picking a pants decomposition along either $\kappa$ or $\kappa^{\prime}$, we see that $\teich(\os_{0,2,1})$ is homeomorphic to $\mathbb{R}_{>0}$, where the homeomorphism is given by the length coordinate.

    If we now consider $\core(\teich(\os_{0,2,1}))$, the lengths of $\kappa$ and $\kappa^{\prime}$ are bounded below by the threshold.
    But they are also bounded above, by an argument similar to the previous case, namely is either $\kappa$ or $\kappa^{\prime}$ are very long, the other one sided curve must be very short.
    This proves that $\core(\teich(\os_{0,2,1}))$ is compact, and as a result $\hNP(\core(\os_{0,2,1}), \vept) = 0$.
    This proves the lemma for $\os_{0,2,1}$.
  \item $\os_{0,1,2}$: This surface has a unique two-sided element, which we denote by $\gamma_{\infty}$.
    The one sided curves on this surface are indexed by $\mathbb{Z}$, where $\gamma_n = D_n \gamma_0$, and $D_n$ is the Dehn twist about $\gamma_\infty$ (see \autoref{fig:s_012}).
    \begin{figure}[h]
      \centering
    \def\svgscale{0.9}
    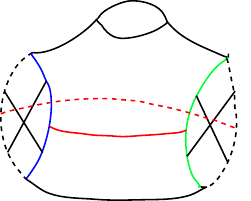

      \caption{The curves $\gamma_{\infty}$, $\gamma_0$ and $\gamma_1$ on $\os_{0,1,2}$.}
      \label{fig:s_012}
    \end{figure}

    The mapping class group of this surface is also virtually generated by $D_n$.
    If we consider the pants decomposition along $\gamma_{\infty}$, we get a Fenchel-Nielsen map from $\teich(\os_{0,1,2})$ to the upper half plane $\mathbb{H}^2$, where the $y$-coordinate is $\frac{1}{\ell(\gamma_{\infty})}$, and the $x$-coordinate is the twist around $\gamma_{\infty}$.
    Furthermore, this map is also an isometry, and with respect to these coordinates, $D_n$ is the action of $
    \begin{pmatrix}
      1 & 1 \\
      0 & 1
    \end{pmatrix}
    $ on $\mathbb{H}^2$.

    If we now consider $\core(\teich(\os_{0,1,2}))$, that consists of the points in $\mathbb{H}^2$ whose $y$-coordinate is greater than some threshold value, i.e.\ a horoball in $\mathbb{H}^2$.
    We showed in \eqref{eq:horocycle-np-lp-is-half} that for the action of $D_n$ on a horocycle, the volume growth entropy and the lattice point growth entropy are both equal to $\frac{1}{2}$.
    The former entropy is precisely $\hNP(\core(\teich(\os_{0,1,2})), \vept)$, and the latter entropy is $\hLP(\teich(\os_{0,1,2}), \vept)$.
  \end{itemize}
  This concludes the proof of the lemma for the $4$ surfaces with $\chi(S) = -1$.
\end{proof}

\subsection{Good Points and Bad Points}
\label{sec:good-points-bad}

The proof of Theorem \ref{thm:entropy-equality} will split up into counting two kinds of net points, which we will call \emph{good points} and $\emph{bad points}$.

\begin{definition}[Good points]
  A point in $B_R(p) \cap \net$ is good if it is at most distance $\vepb R$ away from a lattice point $\gamma p$. The set of good points is denoted by $\net_g(p, R, \vepb)$.
\end{definition}

\begin{definition}[Bad points]
  \label{defn:bad-points}
  A point in $B_R(p) \cap \net$ is bad if it is more than distance $\vepb R$ away from the nearest lattice point. The set of bad points is denoted by $\net_b(p, R, \vepb)$.
\end{definition}

Observe that the classification of a point as good or bad depends on the choice of $R$, $p$, and an additional parameter $\vepb > 0$.

We also further subdivide $\net_g(p, R, \vepb)$ based on what the closest lattice point is.

\begin{definition}[Good point in the domain of $\gamma$]
  For $\gamma \in \mcg(S)$, the set $\net_g(\gamma, p, R, \vepb)$ denotes the subset of $\net_g(p, R, \vepb)$ whose closest lattice point is $\gamma p$.
\end{definition}

We will now prove a lemma that provides an upper bound on the number of good points when restricted to a fundamental domain.

\begin{lemma}
  \label{lem:fd-polynomial-growth}
  There exists a polynomial function $q$, whose degree only depends on the topological type of $S$, such that for any $\gamma \in \mcg(S)$, the following inequality holds for the cardinality of points in $\net_g(\gamma, p, R, \vepb)$.
  \begin{align*}
    \#\left( \net_g(\gamma, p, R, \vepb) \cap B_R(\gamma p) \right) \leq q(R)
  \end{align*}
\end{lemma}

\begin{remark}
  \textcite[Lemma 3.2]{eskinmirzakhani} prove this lemma for Teichmüller spaces of orientable surfaces, by comparing the extremal lengths of various curves on the underlying surfaces.
  We adapt the same proof for non-orientable surfaces, replacing extremal length for hyperbolic lengths instead.
\end{remark}

Before we prove Lemma \ref{lem:fd-polynomial-growth}, we will need the following lemma on packing an $\vepn$-separated set into a ball of fixed radius in Teichmüller space (where $\vepn$ is the parameter associated to our net $\net$).

\begin{lemma}[Packing bound]
  \label{lem:packing-argument}
  For an constants $C > 0$ and $\vepn > 0$, there exists a constant $D(C, \vepn, S)$ depending on the constants $C$, $\vepn$, and the topological type of the surface $S$ such that any ball $B_C(p)$ (independent from the choice of $p$) in $\teich(S)$ cannot contain more than $D$ points that are pairwise distance at least $\vepn$ apart.
\end{lemma}

\begin{proof}
  First of all, note that the above lemma holds for $S = \os_{1,1,0}$, since $\teich(\os_{1,1,0})$ is $\mathbb{H}^2$, which is homogeneous.

  Next, note that the lemma also holds for compact metric spaces, because we can express $D$ as a upper semi-continuous function of the point $p$, which will achieve a maximum on a compact metric space.

  Next, note that if the lemma holds for metric spaces $X$ and $Y$, it also holds for $X \times Y$, where the metric on $X \times Y$ is the $\sup$ product without an additive error.
  To see this, we consider the minimal number of balls $E(C, \vepn, X \times Y)$ of radius $\frac{\vepn}{2}$ needed to cover $X \times Y$.
  The covering number and the packing number $D$ are related by the following standard inequality.
  \begin{align*}
    D(C, \vepn, X \times Y) \leq E(C, \vepn, X \times Y) \leq D(C, 2 \vepn, X \times Y)
  \end{align*}
  Furthermore, one can easily see that for a $\sup$-product $X \times Y$, we can bound the covering number of $X \times Y$ by a product of the covering number for $X$ and $Y$, by taking a product of coverings for $X$ and $Y$.
  \begin{align*}
    E(C, \vepn, X \times Y) \leq E(C, \vepn, X) \times E(C, \vepn, Y)
  \end{align*}
  Combining the two inequalities, get a bound for $D(C, \vepn, X \times Y)$ in terms of $D(C, 2\vepn, X)$ and $D(C, 2\vepn, Y)$.
  \begin{align*}
D(C, \vepn, X \times Y) \leq D(C, 2\vepn, X) \times D(C, 2\vepn, Y)
  \end{align*}

  We now show that the lemma also holds, with a worse constant, for metric spaces that are $\sup$-products with an additive error $c$, where $c$ is a constant smaller than $\vepn$.
  Let $Z$ be a metric space for which the lemma holds, and $Z^{\prime}$ be a metric space with the same underlying points, but whose metric differs from $Z$ by an additive error $c$.
  \begin{align*}
    \left| d_Z(x,y) - d_{Z^{\prime}}(x,y) \right| \leq c
  \end{align*}
  Let $\mathcal{Z}$ be a set of points in a ball of radius $C$ in $Z^{\prime}$ that are pairwise distance at least $\vepn$ apart.
  Consider $\mathcal{Z}$ as a subset of $Z$ instead, we have that they are contained in a ball of radius at most $C + c$, and are pairwise distance at least $\vepn - c$ apart.
  We thus get the inequality relating the packing numbers for $Z$ and $Z^{\prime}$.
  \begin{align*}
    D(C, \vepn, Z^{\prime}) \leq D(C + c, \vepn - c, Z )
  \end{align*}

  Suppose now that we have the lemma for the Teichmüller spaces of all surfaces with Euler characteristic at least $-n$.
  To show the lemma for a Teichmüller space of a surface $S$ with Euler characteristic $-n-1$, we break up the Teichmüller space into the thick part, where all curves are at least $\varepsilon$ or longer, and the thin part.
  We pick $\varepsilon$ such that the metric thin part on the thin part is equal to the $\sup$-product metric, up to an additive error $c$, where $c < \vepn$, by Minsky's product region theorem \cite[Theorem 6.1]{1077244446} (or Theorem \ref{thm:prno}).
   The mapping class group acts co-compactly on the thick, so the first reduction applies, and on the thin part, the metric is the $\sup$ product up to an additive error of $c$, so the second reduction applies.
\end{proof}

\begin{proof}[Proof of Lemma \ref{lem:fd-polynomial-growth}]
  We begin by making three simplifying reductions.
  First, it will suffice to prove the following stronger claim instead.
  \begin{align}
    \label{eq:reduction-one}
    \#\left( \net(\gamma) \cap B_R(\gamma p) \right) \leq q(R)
  \end{align}
  Here, $\net(\gamma)$ denotes the set of net points whose closest lattice point is $\gamma p$: $\net(\gamma)$ is therefore a superset of $\net_g(\gamma, p, R, \vepb)$.

  Next, note that it suffices to prove \eqref{eq:reduction-one} for $\gamma = 1$, since our choice of basepoint $p$ was arbitrary.

  And finally, it will suffice to prove the following claim.
  \begin{claim*}
  There exists a set $\cZ \subset \teich(S)$ such that $\# \cZ \leq R^{f(S)}$, and for $y \in B_R(p)$, there exists a $z \in \cZ$ and $\kappa \in \mcg(S)$ such that $d(y, \kappa z) \leq C$, for some value $f(S)$ that only depends on the topological type of $S$, and some fixed constant $C$.
  \end{claim*}
  To see why this suffices, suppose we have such a $\cZ$.
  Without loss of generality, we can assume that for all points $z \in \cZ$, the closest lattice point is $p$: otherwise we could replace such a point $z$ by $\kappa z$ for an appropriate choice of $\kappa$.
We then have that for any $n \in \net(1) \cap B_p(R)$, there exists some $z \in \cZ$, such that $d(z, n) \leq 2C$.
Since $\# \cZ \leq R^{f(\chi(S))}$, we have that $\#\left( \net(1) \cap B_p(R) \right) \leq C^{\prime} R^{f(\chi(S))}$, for some other constant $C^{\prime}$, by Lemma \ref{lem:packing-argument}.

\emph{Proof of claim:}
We consider short markings on the point $p \in teich(S)$.
Given a pants decomposition $\left\{ \alpha_1, \ldots, \alpha_{k} \right\}$ of a surface, a set of short transverse curves is a collection of curves $\left\{ \beta_1, \ldots , \beta_k \right\}$, such that $\beta_i$ only intersects $\alpha_i$, and is the shortest such curve amongst all the curves intersecting only $\alpha_i$.
The set $\left\{ \alpha_1, \ldots, \alpha_k, \beta_1, \ldots, \beta_k \right\}$ is called a marking.
A marking is said to be short if the total length of the pants curves $\alpha_i$ is minimized amongst the mapping class group orbit of $\left\{ \alpha_1, \ldots, \alpha_k \right\}$.

Note that there are only many short markings at point $p$.
  We  know that for each of these short markings, the lengths of the pants curves are bounded above by some constant $T$.
  Each of these pants multicurves have $N = -3 \chi(S) - b$ pants curves on them, where $b$ is the number of boundary components of $S$.
  Let $\left\{ M_1, \ldots, M_j \right\}$ denote the set of short markings.

  We construct the points $z \in \cZ$ by just varying the lengths of these pants curves: the set of lengths we will allow are the following.
  \begin{align*}
    \text{Acceptable lengths} = \left\{ T, T \exp(-1), T\exp(-2), \ldots, T \left( \exp(-\lceil R \rceil) - \log (s) \right)  \right\}
  \end{align*}
  Here $s$ is the length of the shortest curve on the point $p$ in $\teich(S)$.
  We define the point $z_{j, i_1, i_2, \ldots , i_N}$ to be the point in $\teich(S)$ obtained by considering the marking $M_j$ at $p$, and setting the length of the $k$\textsuperscript{th} pants curve to be $i_k$, where the $i_k$ is one of the acceptable lengths.
  It's clear that the cardinality of $\cZ$ is at most $J \cdot R^N$, which is a polynomial only depending on the topological type of the surface $S$ and the basepoint $p$.

  Suppose now that $y$ is some other point in $B_R(p)$.
  We pick a $\kappa \in \mcg(S)$ such that the shortest marking on $\kappa y$ is one of the markings $M_j$ for $1 \leq j \leq J$.
  We now need to show that one of the $z \in \cZ$ is close to $\kappa y$.
  Pick the $z$ such that the corresponding lengths of the pants curves are closest to the lengths of the pants curves on $\kappa y$.
  We can now invoke the combinatorial distance formula for Teichmüller metric (proved by \textcite{rafi2007combinatorial} for the orientable setting, and Theorem \ref{thm:distance-formula} for the non-orientable case).
  \begin{align*}
    d(z, \kappa y) \emul \sum_Y \left[ d_Y(z, \kappa y) \right]_k + \sum_{\alpha \not \in \Gamma} \log \left[ d_{\alpha} (z, \kappa) \right]_k + \max_{\alpha \in \Gamma} d_{\mathbb{H}_{\alpha}} (z, \kappa y)
  \end{align*}
  In the above formula, the first term is the distance between the short markings when projected to non-annular subsurfaces, the second term is the distance between the short markings when projected to annular subsurfaces whose core curves are not the pants curves in the marking, and the third term corresponds to the length and twist parameters of the short curves.

  Since both $z$ and $\kappa y$ have the same short markings, the first two terms in the above sum become $0$.
  Also, since we picked $z$ to be the element of $\cZ$ such that the lengths were closest to those on $\kappa y$, the third term is bounded by some constant, which proves the result.
\end{proof}

\subsection{Using Complexity Length to Count Bad Points}
\label{sec:using-compl-length}

In this section we introduce an alternative to the Teichmüller metric, called the \emph{complexity length} (see Definition \ref{defn:complexity-length}).
Complexity length (denoted by $\mathfrak{L}$) was constructed by \textcite{dowdall2023lattice}, in order to get better estimates on net points contained in the thin part of Teichmüller space (for orientable surfaces).
We adapt the construction of complexity to the Teichmüller space of non-orientable surfaces in Section \ref{sec:line-gap-compl}.
In this section, we state the main results about complexity length we need in order to prove Theorem \ref{thm:entropy-equality}.

For this section, we will state the results with a rescaled version of complexity length, in order to compare it with Teichmüller length.

\begin{definition}[Rescaled complexity length]
  Let $S$ be a surface of finite type.
  The rescaled complexity length $\dcomp$ on $\core(\teich(S))$ is given by the following formula.
  \begin{align*}
    \dcomp(x, y) = \frac{\mathfrak{L}(x,y)}{\hNP(\core(\teich(S)), \vept)}
  \end{align*}
  Here, $\mathfrak{L}(x,y)$ is the complexity length between points $x$ and $y$.
\end{definition}

The first result we will need is a count of the net points with respect to the rescaled complexity length.

\begin{theorem}[Theorem 12.1 of \cite{dowdall2023lattice}, Theorem \ref{thm:counting-with-complexity}]
  \label{thm:counting-with-complexity-rescaled}
  There exists a polynomial function $p(R)$ that depends on the net $\net$, and a parameter $\veperr > 0$, such that the following inequality holds for any $\veperr > 0$.
  \begin{align*}
    \#\left( y \in \net \mid \dcomp(p, x) \leq R \right) \leq p(R) \exp((\hNP(\core(\teich(S)), \vept) + \veperr) \cdot R)
  \end{align*}
\end{theorem}

The next result, which is the main theorem of Section \ref{sec:line-gap-compl}, is that if $y$ is a bad point that is Teichmüller distance $R$ away from $p$, then its rescaled complexity distance to $p$ is smaller than $R$ by a definite amount.
We state this theorem with an additional hypothesis on the net point entropy of subsurfaces.
We will establish that this hypothesis holds inductively in Section \ref{sec:proof-theorem}.

\begin{theorem}[Linear gap in complexity length, Theorem \ref{thm:linear-gap}]
  \label{thm:linear-gap-rescaled}
  Suppose that for all proper subsurfaces $V$ of $\os$, the following inequality holds.
  \begin{align*}
    \hNP(\core(\teich(V)), \vept) < \hNP(\coret{\os})
  \end{align*}
  Then for any $\vepb >0$, there exists $c > 0$, such that for all $R > 0$, and for any bad point $y$, i.e.\ a point in $\net_b(p, R, \vepb)$, the following upper bound on the complexity distance between $p$ and $y$ holds.
  \begin{align*}
    \dcomp(p, y) \leq R(1 - c)
  \end{align*}
\end{theorem}
\begin{remark}
  We give a brief outline of why the above result should hold.
  Complexity length can be thought of as a weighted version of Teichmüller length, where a specific segment is assigned a weight based on whether it's traveling in the thick part of Teichmüller space, or the thin part.
  In the latter case, it is assigned a smaller weight that is proportional to the net point entropy of the product region it is traveling in.
  The Teichmüller geodesics associated to bad points spend a significant fraction traveling in the thin part, by the very definition of bad points.
  It stands to reason then that the complexity length assigned to them is smaller by a definite amount, due to the time they spend in the thin part.
\end{remark}

Combining Theorems \ref{thm:counting-with-complexity-rescaled} and \ref{thm:linear-gap-rescaled}, it follows that as $R$ goes to $\infty$, the proportion of bad points goes to $0$, which is what we need for Theorem \ref{thm:entropy-equality}.

\subsection{Entropy Gap}
\label{sec:entr-gap-cons}

In the previous section, we saw that the key hypothesis we need for the complexity length estimate was that for any proper subsurface $V$ of $\os$, the following strict inequality held.
\begin{align}
  \label{eq:net-point-entropy-gap}
  \hNP(\core(\teich(V)), \vept) < \hNP(\core(\teich(\os)), \vept)
\end{align}

In this section, we will prove that \eqref{eq:net-point-entropy-gap} holds by proving a similar inequality for the lattice point entropy, and using the fact that Theorem \ref{thm:entropy-equality} holds for all proper subsurfaces $V$, by the inductive hypothesis.

\begin{lemma}[Lattice point entropy gap]
  \label{lem:entropy-inequality}
  Let $\os$ be a surface, and $\chi(\os) \leq -2$. If $V$ is a proper subsurface and Theorem \ref{thm:entropy-equality} holds for $V$, then we have the following strict inequality between their lattice point entropy.
  \begin{align*}
    \hLP(\teich(V), \vept) < \hLP(\teich(\os), \vept)
  \end{align*}
\end{lemma}

\begin{remark}
  We do actually need the hypothesis $\chi(\os) \leq -2$ in the statement of the lemma for two reasons.
  The first reason is that the lemma is actually false for $\os_{1,0,1}$.
  Recall that this surface has the torus with one boundary component as a subsurface, but their mapping class groups are isomorphic, and have the same lattice point growth entropy.
  Another reason why we need the hypothesis is that the proof of the lemma proceeds via a construction of pseudo-Anosov elements on $\os$, and $\os_{1,0,1}$ does not admit any pseudo-Anosov mapping classes.
\end{remark}

\begin{proof}[Proof of Lemma \ref{lem:entropy-inequality}]
  Observe that $\mcg(V)$ is a subgroup of $\mcg(\os)$.
  We will first construct an intermediate subgroup $H = \mathbb{Z} * \mcg(V)$, which is the free product of a pseudo-Anosov element in $\mcg(\os)$ with $\mcg(V)$, and show that $\hLP(H, \vept) > \hLP(\mcg(V), \vept)$.
  This is enough to prove the result, since $H$ is a subgroup of $\mcg(\os)$, we have that $\hLP(\mcg(\os), \vept) \geq \hLP(H, \vept)$.

  We now need to show that $\mcg(\os)$ contains a pseudo-Anosov element.
  We can invoke Penner's construction of pseudo-Anosov mapping classes (\cite[Theorem 4.1]{penner1988construction}), as long as we can construct a filling collection of \emph{two-sided} curves in $\os$.
  This may not be always possible for $\os$ where $\chi(\os) = -1$, but for $\os$ with $\chi(\os) \leq -2$, this is always possible (see \cite{Liechti2018MinimalPS} and \cite{khan2023pseudo} for explicit constructions).
  Let $\kappa$ denote the pseudo-Anosov mapping class we construct.

  By \cite[Proposition 6.6]{10.1093/imrn/rny001}, there exists a large enough $n$ such that $\kappa^n$ and $\mcg(V)$ generate their free product in $\mcg(S)$: call this subgroup $H$.


  We now need to show that the lattice point entropy for $H$ is strictly larger than the $\mcg(V)$.
  To see this, we recall an equivalent definition of the lattice point entropy.
  The lattice point entropy is the infimum of the set of exponents $h$ such that the following Poincaré series transitions converges for any $x \in \teich(\os)$.
  \begin{align}
    \label{eq:dirichlet}
    \sum_{\gamma \in H} \exp\left( -h \cdot d_{\vept}(x, \gamma x) \right)
  \end{align}
  Since $H = \mathbb{Z} * \mcg(V)$, we can represent $\gamma \in H$ as $a_1 \cdot b_1 \cdot a_2 \cdots a_k \cdot b_k$, where $a_i$ belong in $\mathbb{Z}$ and $b_i$ belong in $\mcg(V)$.
  We use this along with the triangle inequality to get an upper bound for $d(x, \gamma x)$.
  \begin{align}
    \label{eq:triangle-inequality-1}
    d_{\vept}(x, \gamma x) \leq \sum_{i=1}^k d_{\vept}(x, a_i x) + d_{\vept}(x, b_i x)
  \end{align}
  We plug inequality \eqref{eq:triangle-inequality-1} into \eqref{eq:dirichlet} to get a lower bound.
  \begin{align}
    \label{eq:dirichlet-lower-bound}
    \sum_{\gamma \in H} \exp\left( -h \cdot d_{\vept}(x, \gamma x) \right) &= \sum_{k=1}^\infty \left(  \sum_{a_1}\cdots \sum_{a_k} \sum_{b_1}\cdots \sum_{b_k}  \exp(-h \cdot d_{\vept}(x, a_1 \cdot b_1 \cdots a_k \cdot b_k x)) \right) \\
    &\geq \sum_{k=1}^{\infty} \left( \sum_{a \in \mathbb{Z}} \exp(-h \cdot d_{\vept}(x, ax)) \right)^k \left( \sum_{b \in \mcg(V)} \exp(-h \cdot d_{\vept}(x, bx)) \right)^k
    \label{eq:dirichlet-lower-bound-2}
  \end{align}

  We have that Theorem \ref{thm:entropy-equality} holds for $V$, which means that $\core(\teich(V))$ is SCC.
  Corollary 5.4 of \cite{10.1093/imrn/rny001} states that group actions that are SCC have Poincaré series that diverge at the critical exponent.
  This means there's small enough $\varepsilon > 0$ such that for $h = \hLP(\teich(V), \vept) + \varepsilon$, the series converges to a value greater than $1$.
  But that means the Poincaré series for $H$ diverges at $\hLP(\teich(V), \vept) + \varepsilon$, since we have a lower bound by a geometric series whose ratio is greater than $1$.
  This proves that the critical exponent for $H$ is strictly greater than the critical exponent for $\mcg(V)$.
\end{proof}

We can now prove the entropy gap result for $\hNP$.

\begin{lemma}[Net point entropy gap]
  \label{lem:net-point-entropy-inequality}
  Let $\os$ be a surface, and $\chi(\os) \leq -2$. If $V$ is a proper subsurface and Theorem \ref{thm:entropy-equality} holds for $V$, then we have the following strict inequality between their net point entropy.
  \begin{align*}
    \hNP(\core(\teich(V)), \vept) < \hNP(\core(\teich(\os)), \vept)
  \end{align*}
\end{lemma}

\begin{proof}
  We have the following inequality, which follows trivially from the definition of $\hLP$ and $\hNP$.
  \begin{align}
    \label{eq:ineq1}
    \hLP(\teich(\os), \vept) \leq \hNP(\core(\teich(\os)), \vept)
  \end{align}
  From Lemma \ref{lem:entropy-inequality}, we get the following inequality.
  \begin{align}
    \label{eq:ineq2}
    \hLP(\teich(V), \vept) < \hLP(\teich(\os), \vept)
  \end{align}
  Finally, since we have Theorem \ref{thm:entropy-equality} for $V$, we have the following equality.
  \begin{align}
    \label{eq:eq3}
    \hLP(\teich(V), \vept) = \hNP(\core(\teich(V)), \vept)
  \end{align}
  Chaining together \eqref{eq:ineq1}, \eqref{eq:ineq2}, and \eqref{eq:eq3} gives us the result.
\end{proof}

\subsection{Proof of Theorem \ref{thm:entropy-equality}}
\label{sec:proof-theorem}

We now have all the lemmas we need in order to prove Theorem \ref{thm:entropy-equality}.

\begin{proof}[Proof of Theorem \ref{thm:entropy-equality}]
  We will prove this lemma by inducting on the complexity of the surface $\os$.
  Lemma \ref{lem:entropy-equality-base-case} proves the result for surfaces with Euler characteristic equal to $-1$, which serves as the base case of the theorem.

  We now assume that Theorem \ref{thm:entropy-equality} already holds for all proper subsurfaces $V$ of $\os$: it will suffice to show that the result holds for $\os$.

  We will establish that for any $\vepb > 0$, there exists a polynomial $q(R)$, and $R$ large enough, such that the following bound holds.
  \begin{align*}
    \#\left( B_R(p) \cap \net \right) \leq q(R) \cdot \exp(\hLP(\teich(\os), \vept) \cdot R \cdot (1 + 2 \vepb))
  \end{align*}

  We first count the good points in $B_R(p)$, by partitioning them according to the nearest lattice point.
  \begin{align}
    \label{eq:good-point-partition}
    \net_g(p, R, \vepb) = \bigsqcup_{\gamma \in \mcg(\os)} \net_g(\gamma, p, R, \vepb)
  \end{align}
  Observe that if $y \in \net_g(\gamma, p, R, \vepb)$, then $d(p, \gamma p) \leq R(1 + \vepb)$, since $d(p, y) \leq R$ and $d(y, \gamma p) \leq \vepb R$.
  This observation leads to the following upper bound on $\#\left( \net_g(p, R, \vepb) \right)$.
  \begin{align}
    \#\left( \net_g(p, R, \vepb) \right) \leq \sum_{\substack{\gamma \in \mcg(\os) \\ d(p, \gamma p) \leq R(1 + \vepb)}} \#\left( B_{\vepb R}(\gamma p) \cap \net_g(\gamma, p, R, \vepb)  \right) \label{eq:good-estimate-1}
  \end{align}
  By Lemma \ref{lem:fd-polynomial-growth}, there exists a polynomial $q(R)$ such that each term in the above sum is at most $q(R)$.
  \begin{align}
    \#\left( \net_g(p, R, \vepb) \right) &\leq \sum_{\substack{\gamma \in \mcg(\os) \\ d(p, \gamma p) \leq R(1 + \vepb)}} q(R) \\
     &\leq q(R) \cdot \exp(\hLP(\teich(\os), \vept) \cdot R \cdot (1 + 2 \vepb))
    \label{eq:good-estimate-2}
  \end{align}
  Here, we estimated the cardinality of $\gamma$ such that $d(p, \gamma p) \leq R(1+\vepb)$ as at most $\exp(\hLP(\teich(\os), \vept) \cdot R \cdot (1 + 2 \vepb))$, for large enough $R$.
  We have the desired upper bound on the cardinality for the good points.
  Now we show that the number of bad points is much smaller than the total number of points in the ball, which will then prove the result.

  From the inductive hypothesis, we have that Theorem \ref{thm:entropy-equality} holds for all proper subsurfaces $V$.
  By Lemma \ref{lem:net-point-entropy-inequality}, we have that $\hNP(\core(\teich(V)), \vept) < \hNP(\core(\teich(\os)), \vept)$: this is precisely the hypothesis we need to apply Theorem \ref{thm:linear-gap-rescaled}.
  Applying the theorem, we see that if $y$ is a bad point, $\dcomp(p, y) \leq R(1 - c)$.
  We then apply Theorem \ref{thm:counting-with-complexity-rescaled} to get an upper bound on the number of bad points.
  \begin{align*}
    \#\left( \net_b(p, R, \vepb) \right) \leq kR^k \cdot \exp((\hNP(\core(\teich(\os)), \vept) + \veperr) \cdot R \cdot (1 - c))
  \end{align*}
  We pick $\veperr$ small enough such that the above term satisfies the following inequality for large enough $R$.
  \begin{align*}
    \exp((\hNP(\core(\teich(\os)), \vept) + \veperr) \cdot R \cdot (1 - c)) < \exp\left(\hNP(\core(\teich(\os)), \vept) \cdot R \cdot \left(1 - \frac{2c}{3}\right)\right)
  \end{align*}
  On the other hand, we have that for large enough $R$, the total number of net points is at least $\exp\left(\hNP(\core(\teich(\os)),\vept) \cdot R \cdot \left(1 - \frac{c}{2}\right)\right)$.
  Combining these two facts, we see that the proportion of bad points goes to $0$ as $R$ goes to $\infty$, which proves the result.
\end{proof}

%% file: images/s_101.pdf_tex
\begingroup%
  \makeatletter%
  \providecommand\color[2][]{%
    \errmessage{(Inkscape) Color is used for the text in Inkscape, but the package 'color.sty' is not loaded}%
    \renewcommand\color[2][]{}%
  }%
  \providecommand\transparent[1]{%
    \errmessage{(Inkscape) Transparency is used (non-zero) for the text in Inkscape, but the package 'transparent.sty' is not loaded}%
    \renewcommand\transparent[1]{}%
  }%
  \providecommand\rotatebox[2]{#2}%
  \newcommand*\fsize{\dimexpr\f@size pt\relax}%
  \newcommand*\lineheight[1]{\fontsize{\fsize}{#1\fsize}\selectfont}%
  \ifx\svgwidth\undefined%
    \setlength{\unitlength}{183.69465024bp}%
    \ifx\svgscale\undefined%
      \relax%
    \else%
      \setlength{\unitlength}{\unitlength * \real{\svgscale}}%
    \fi%
  \else%
    \setlength{\unitlength}{\svgwidth}%
  \fi%
  \global\let\svgwidth\undefined%
  \global\let\svgscale\undefined%
  \makeatother%
  \begin{picture}(1,0.450167)%
    \lineheight{1}%
    \setlength\tabcolsep{0pt}%
    \put(0,0){\includegraphics[width=\unitlength,page=1]{s_101.pdf}}%
    \put(0.24368216,0.26165627){\makebox(0,0)[lt]{\lineheight{1.25}\smash{\begin{tabular}[t]{l}$\kappa^{\prime}$\end{tabular}}}}%
    \put(0.71142293,0.34510205){\makebox(0,0)[lt]{\lineheight{1.25}\smash{\begin{tabular}[t]{l}$\kappa$\end{tabular}}}}%
  \end{picture}%
\endgroup%

%% file: images/s_021.pdf_tex
\begingroup%
  \makeatletter%
  \providecommand\color[2][]{%
    \errmessage{(Inkscape) Color is used for the text in Inkscape, but the package 'color.sty' is not loaded}%
    \renewcommand\color[2][]{}%
  }%
  \providecommand\transparent[1]{%
    \errmessage{(Inkscape) Transparency is used (non-zero) for the text in Inkscape, but the package 'transparent.sty' is not loaded}%
    \renewcommand\transparent[1]{}%
  }%
  \providecommand\rotatebox[2]{#2}%
  \newcommand*\fsize{\dimexpr\f@size pt\relax}%
  \newcommand*\lineheight[1]{\fontsize{\fsize}{#1\fsize}\selectfont}%
  \ifx\svgwidth\undefined%
    \setlength{\unitlength}{144.0011008bp}%
    \ifx\svgscale\undefined%
      \relax%
    \else%
      \setlength{\unitlength}{\unitlength * \real{\svgscale}}%
    \fi%
  \else%
    \setlength{\unitlength}{\svgwidth}%
  \fi%
  \global\let\svgwidth\undefined%
  \global\let\svgscale\undefined%
  \makeatother%
  \begin{picture}(1,1.29014523)%
    \lineheight{1}%
    \setlength\tabcolsep{0pt}%
    \put(0,0){\includegraphics[width=\unitlength,page=1]{s_021.pdf}}%
    \put(0.48999986,0.7919788){\makebox(0,0)[lt]{\lineheight{1.25}\smash{\begin{tabular}[t]{l}$\kappa$\end{tabular}}}}%
    \put(0.24858646,0.85376338){\makebox(0,0)[lt]{\lineheight{1.25}\smash{\begin{tabular}[t]{l}$\kappa^{\prime}$\end{tabular}}}}%
  \end{picture}%
\endgroup%

%% file: images/s_012.pdf_tex
\begingroup%
  \makeatletter%
  \providecommand\color[2][]{%
    \errmessage{(Inkscape) Color is used for the text in Inkscape, but the package 'color.sty' is not loaded}%
    \renewcommand\color[2][]{}%
  }%
  \providecommand\transparent[1]{%
    \errmessage{(Inkscape) Transparency is used (non-zero) for the text in Inkscape, but the package 'transparent.sty' is not loaded}%
    \renewcommand\transparent[1]{}%
  }%
  \providecommand\rotatebox[2]{#2}%
  \newcommand*\fsize{\dimexpr\f@size pt\relax}%
  \newcommand*\lineheight[1]{\fontsize{\fsize}{#1\fsize}\selectfont}%
  \ifx\svgwidth\undefined%
    \setlength{\unitlength}{120.44186089bp}%
    \ifx\svgscale\undefined%
      \relax%
    \else%
      \setlength{\unitlength}{\unitlength * \real{\svgscale}}%
    \fi%
  \else%
    \setlength{\unitlength}{\svgwidth}%
  \fi%
  \global\let\svgwidth\undefined%
  \global\let\svgscale\undefined%
  \makeatother%
  \begin{picture}(1,0.80071577)%
    \lineheight{1}%
    \setlength\tabcolsep{0pt}%
    \put(0,0){\includegraphics[width=\unitlength,page=1]{s_012.pdf}}%
    \put(0.35285514,0.19400406){\makebox(0,0)[lt]{\lineheight{1.25}\smash{\begin{tabular}[t]{l}$\gamma_{\infty}$\end{tabular}}}}%
    \put(0.19110861,0.50847598){\makebox(0,0)[lt]{\lineheight{1.25}\smash{\begin{tabular}[t]{l}$\gamma_0$\end{tabular}}}}%
    \put(0.69696526,0.54798904){\makebox(0,0)[lt]{\lineheight{1.25}\smash{\begin{tabular}[t]{l}$\gamma_1$\end{tabular}}}}%
  \end{picture}%
\endgroup%

%% file: linear-gap.tex
\section{Linear Gap in Complexity Length}
\label{sec:line-gap-compl}

\subsection{An Example of Counting in Product Regions}
\label{sec:an-example-counting}

Before we define complexity length, we will look at an example that illustrates why we need complexity length.
Theorem 1.3 of \textcite{10.1215/00127094-1548443} proves an estimate on the volume of balls in Teichmüller space.
From this volume estimate, we can obtain an estimate on the cardinality of net points of an $(\vepn, 2 \vepn)$-net $\net$.
\begin{theorem}[Theorem 1.3 of \cite{10.1215/00127094-1548443}]
  \label{thm:abem}
  For a point $p$ in $\teich(S)$ (where $S$ is a genus $g$ surface with $b$ boundary components), the number of net points in a ball of radius $R$ centered at the origin satisfies the following asymptotic as $R$ goes to $\infty$.
  \begin{align*}
    \#\left( \net \cap B_R(p) \right) \emul \exp((6g-6 + 2b)R)
  \end{align*}
  Here, the multiplicative and additive constants showing up in $\emul$ only depend on $p$ and $\vepn$.
\end{theorem}

Suppose now that we want to use the above theorem to count net points in a product region.
More concretely, let $p$ be a point in $\teich(S)$ such that a non-separating curve $\gamma$ is very short: $\ell_{\gamma}(p) \leq \delta \cdot \exp(-R_0)$, for some $\delta > 0$, and some large $R_0$, and we want to estimate the cardinality of $\net \cap B_R(p)$ for $R < R_0$.
Note that since $R < R_0$, the ball $B_R(p)$ is still contained in the product region of Teichmüller space where $\ell_\gamma \leq \delta$.

Since the entire ball $B_R(p)$ is in a product region, we have by Minsky's product region theorem (see Theorem \ref{thm:prno} for a precise statement) that the ball decomposes (up to an additive error) as the
product of a ball in $\teich(S \setminus \gamma)$ and ball in $\mathbb{H}$ (which corresponds to the length and twist around $\gamma$).
This gives us an alternative estimate for $\#\left( \net \cap B_R(p)\right)$.
\begin{align}
  \label{eq:prod-decomp}
  \#\left( \net \cap B_R(p) \right) &\leq C \cdot \left( \net_1 \cap B_R(p, S \setminus \gamma) \right) \cdot \left( \net_2 \cap B_R(p, \mathbb{H})\right)
\end{align}
Here $\net_1$ and $\net_2$ are $(\vepn, 2\vepn)$ nets for $\teich(S \setminus \gamma)$ and $\mathbb{H}$, and $B_R(p, S \setminus \gamma)$ and $B_R(p, \mathbb{H})$ are projections on the ball $B_R(p)$ to the two components.
Applying Theorem \ref{thm:abem} to the right hand side of \eqref{eq:prod-decomp}, we get a better estimate than we would have gotten with a direct application of Theorem \ref{thm:abem}.
\begin{align*}
  \#\left( \net \cap B_R(p) \right) &\leq C \cdot \left( \net_1 \cap B_R(p, S \setminus \gamma) \right) \cdot \left( \net_2 \cap B_R(p, \mathbb{H})\right) \\
                                           &\emul \exp((6(g-1)-6 + 2(b+2))R) \cdot \exp(R) \\
  &= \exp((6g-6+2b-1)R)
\end{align*}

This example illustrates that in order to count net points accurately, it's not sufficient to just estimate the distance the between the base point $p$ and the net point $n$: if the geodesic segment $[p, n]$ travels in a product region, the count will be lower than what Theorem \ref{thm:abem} predicts.
In fact, the count will also depend on the type of the product region.
In the above example, the product region had just one curve $\gamma$ becoming short, but in general, a product region can have multiple curves getting short, in which case, the net point count will be even smaller.

We now consider a geodesic $[x,y]$ (where $x$ and $y$ are net points) that travels through several product regions $\pi_i$, and possibly the thick part, which we will also consider a product region, albeit a trivial one.
Let $h_i$ be the exponent associated to the product region $\pi_i$: this is the exponent that will appear when we invoke Theorem \ref{thm:abem} to count net points in the product region $\pi_i$.
The order in which $[x, y]$ travels through the product region is specified in \autoref{fig:schematic}.
\begin{figure}[h]
  \centering
    \def\svgscale{0.8}
    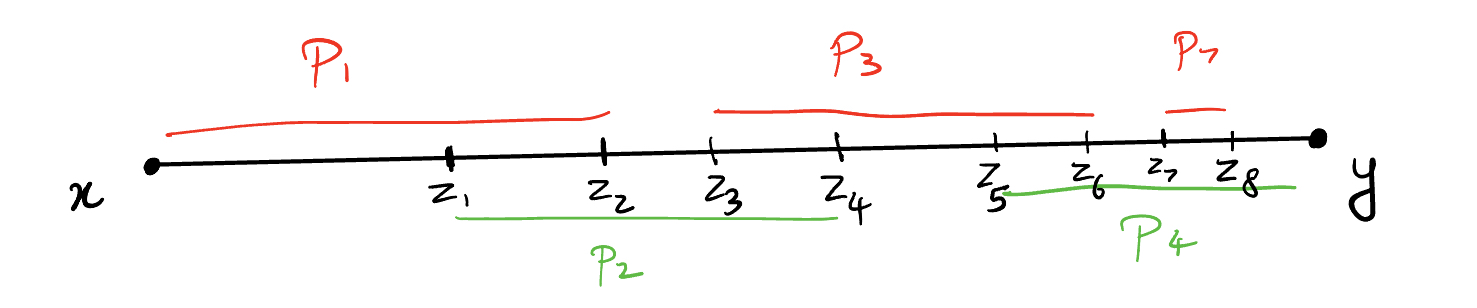

  \caption{A schematic of the geodesic $[x,y]$ traveling through several product regions.}
  \label{fig:schematic}
\end{figure}

Let $z_i$ denote the points on the geodesic segment that correspond to the times when the geodesic enters or exits a product region: in \autoref{fig:schematic}, we have labeled $z_i$ for $1 \leq i \leq 8$.
Let $\mathcal{J}_i$ denote the interval $[x, z_1]$ for $i=0$, the interval $[z_i, z_{i+1}]$ for $1 \leq i \leq 7$, and $[z_8, y]$ for $i=8$.
Let $\ell_i$ be the length of $\mathcal{J}_i$, and $e_i$ be the sum of the $h_i$ for each of the product regions that the interval $\mathcal{J}_i$ is in.
Also, for each $z_i$, let $z_i^{\prime}$ denote the nearest net point.

Keeping $x$ fixed, we can try to count the number of net points $y$ that satisfy the configuration we have described.
We have about $O(\exp(e_0 \ell_0))$ possibilities for $z_1^{\prime}$, and then keeping a $z_1^{\prime}$ fixed, $O(\exp(e_i \ell_1))$ possibilities for $z_2^{\prime}$ and so on.
Multiplying all these estimates, we have the following upper bound for cardinality of $y$.
\begin{align*}
  \#\left( y \right) \leq \exp\left( \sum_{i=0}^8 e_i \ell_i \right)
\end{align*}

The quantity $\sum e_i \ell_i$ serves as a re-weighted version of length of the geodesic in manner that works well with the counting function.
This is a primitive version of the \emph{complexity length} of $[x,y]$, and motivates the actual definition.

Before we define complexity length, we note two ways in which the above estimate overestimates the actual number of net points: that happens when a point in the geodesic is simultaneously in two or more product regions, which can happen in two ways.
\begin{enumerate}[(i)]
\item The active subsurfaces associated to product regions are disjoint: In this case, we are accounting the length of the geodesic segment multiple times: once for each product region we are in. However, this overcounting is still better than directly invoking Theorem \ref{thm:abem}, since the sums of the exponents $h_i$ associated to each of the disjoint product regions are smaller than the exponent associated to the entire surface.
\item The active subsurfaces associated toproduct regions are nested: In this case as well, we are accounting for the length of a geodesic multiple times, once for each product region we are in.
  Unlike in the previous case, in this case, the exponents associated to each product region can add up to a quantity larger than the exponent associated to the entire surface, which means the presence of nested product regions can give a worse estimate than Theorem \ref{thm:abem}.
  We will get around this problem by looking only at product regions associated to special subsurfaces which are called \emph{witnesses}.
\end{enumerate}




\subsection{An Overview of Complexity Length}
\label{sec:an-overv-compl}

Now that we have motivated the need for complexity length, as well as considering special subsurfaces called witnesses, we formally define them in this section.
This section is a summary for Sections 7 through 12 of \textcite{dowdall2023lattice}, so we refer the reader to those sections for details we elide.
One difference in our presentation is that we care about these constructions for both orientable and non-orientable surfaces, while the original authors only work with orientable surfaces.
However, their constructions and proofs go through for non-orientable surfaces, as long as we provide a proof of the non-orientable versions of some of the foundational results they use.
We list those theorems here, and link to the proof of the non-orientable version that appears in Section \ref{sec:geom-of-teich}.
\begin{enumerate}[(i)]
\item Minsky's product region theorem (see Theorem \ref{thm:prno}).
\item Distance formula for Teichmüller space (see Theorem \ref{thm:distance-formula}).
\item Active intervals for subsurfaces (see Proposition \ref{thm:active-intervals}).
\item Consistency and realization (see Theorem \ref{thm:consistency-realization}).
\end{enumerate}

Let $\os$ be a surface (not necessarily orientable), and $\mathbf{C}$ some large arbitrary constant, and $\vept > 0$ a small constant we pick later.
We also pick constants $N_V$, for each $V \sqsubset \os$, such that $N_V$ only depends on the topological type of $V$.
The precise values of the $N_V$'s is specified via Proposition 10.13 of \cite{dowdall2023lattice}.
We will also abuse notation slightly and use $\hNP(V)$ to refer to $\hNP(\core(\teich(V)))$ whenever $V$ is a non-orientable surface: when $V$ is orientable, $\hNP(V)$ will refer to $\hNP(\teich(V))$.

Let $[x,y]$ be a geodesic segment in $\teich(\os)$: we describe the set $\Upsilon(x, y)$ of subsurfaces along which $[x, y]$ has large projections.

\begin{definition}[Active subsurfaces]
  A subsurface $V \sqsubset \os$ is an active subsurface, i.e.\ in $\Upsilon(x,y)$, if one of the following two conditions hold.
  \begin{enumerate}[(i)]
  \item The projection to $\cC(V)$ has diameter at least $N_V$.
  \item If $V$ is annular with core curve $\gamma$, then
    \begin{align*}
      \min\left( \ell_\gamma(x), \ell_{\gamma}(y) \right) < \vept
    \end{align*}
  \end{enumerate}
\end{definition}

Associated to each active subsurface $V$, there is a non-empty connected sub-interval of $[x,y]$, which we call an active interval, and denote $\mathcal{I}_V^{\vept}$, which we obtain via an application of Proposition \ref{thm:active-intervals}.
The active intervals associated to active subsurfaces enjoy the following properties.
\begin{enumerate}[(i)]
\item $\ell_\alpha(z) < \vept$ for $z \in \mathcal{I}_V^{\vept}$ and $\alpha \in \partial V$.
\item For $z \not \in \mathcal{I}_V^{\vept}$, $\ell_{\alpha}(z) > {\vept}^{\prime}$ for some $z \in \partial V$, and some ${\vept}^{\prime} < {\vept}$ that only depends on ${\vept}$.
\item For $[w,z] \subset [x,y]$ with $[w,z] \cap \mathcal{I}_V^{\vept} = \varnothing$, $d_V(w,z) \leq M_{\vept}$ for some $M_{\vept}$ that only depends on ${\vept}$.
\item For $U \pitchfork V$, $\mathcal{I}_U^{\vept} \cap \mathcal{I}_V^{\vept} = \varnothing$.
\end{enumerate}

For pairs of transverse subsurfaces $U \pitchfork V$, since $\mathcal{I}_U^{\vept} \cap \mathcal{I}_V^{\vept} = \varnothing$ we can also determine which of the subsurfaces are active first.
\begin{definition}[Behrstock partial order]
  If $U$ and $V$ are a pair of transverse subsurfaces in $\Upsilon(x,y)$, we say $U \lessdot V$ if $\mathcal{I}_U^{\vept}$ appears to the left of $\mathcal{I}_V^{\vept}$ in $[x,y]$.
\end{definition}

Observe that when restricted to $\mathcal{I}_V^{\vept}$, the geodesic is traveling in a product region, one of whose components is $\teich(V)$, but trying to apply the technique from the previous subsection leads to the problem of overcounting, namely overcounting arising from subsurfaces either nested in $V$, or subsurfaces $V$ is nested in.

To deal with this issue, we will consider a subset of $\Upsilon(x,y)$, called a \emph{witness family}.
However, to avoid overcounting, some additional properties are required of the witness families.
Rather than defining all of those properties without context, we introduce them one at a time, after motivating the need for the property.

\begin{definition}[Witness family]
  A witness family $\Omega(x,y)$ associated to the geodesic $[x,y]$ is a subset of $\Upsilon(x, y)$ satisfying the following properties.
  \begin{enumerate}[(i)]
  \item For any $Z \in \Upsilon(x,y)$, $Z \sqsubset W$ for some $W \in \Omega(x,y)$.
  \item If $Z \sqsubset W$, and $Z \in \Omega(x,y)$ and $W \in \Upsilon(x,y)$, then $W$ must also either be a witness, or must be transverse to a witness $V \in \Omega(x,y)$ such that $Z \sqsubset V$.
  \end{enumerate}
\end{definition}
The first condition of the definition ensures that when we restrict our attention from all active subsurfaces to witnesses, we do not lose information, i.e.\ every active subsurface contributes to whichever witness it is contained in.
The second condition is a more technical requirement that is required to ensure that the other properties we define later work nicely.

We now make the notion of an active subsurface \emph{contributing to a witness} more precise.
\begin{definition}[Complete witness family]
  For an active subsurface $V$, a witness $W$ is said to be the $\Omega$-completion of $V$, denoted $\colsup{V}{\Omega}$ if $W$ is the minimal (by inclusion) witness containing $V$.
  If $W = \colsup{V}{\Omega}$, we say $V$ \emph{contributes} to $W$.
  Furthermore, a witness family is \emph{complete} if every active subsurface has a unique $\Omega$-completion.
\end{definition}

By partitioning off the collection active subsurfaces into classes, where each class is represented by a witness, and only considering the product regions associated to the witnesses, rather than all the active subsurfaces, we can cut down on the overcount we obtain by considering all active subsurfaces.

We now look at an extreme example of a complete witness families to motivate further properties that we will need from the witness families in order to count well.

\begin{example}[Trivial witness family]
  \label{ex:uninsulated}
  Let $\gamma$ be a pseudo-Anosov mapping class on $\os$, such that $\gamma$ has large translation distance on $\cC(\os)$, and $\delta$ a reducible mapping class, acting on a subsurface $V$ such that the action of $\delta$ on $\cC(V)$ has large translation distance as well.

  Let $x$ be a point in $\teich(\os)$, and $y = \delta \gamma \delta^{-1} x$.
  The active subsurfaces for $[x, y]$ contain the surfaces $\os$, $V$, and $\gamma V$: however, we can pick $\Omega(x,y) = \left\{ \os \right\}$, and check that this is a complete witness family.
\end{example}

  In the above example, since we only have one subsurface in our witness family, we certainly do not overcount via overlapping product regions, but we do end up ignoring the fact that the geodesic $[x,y]$ travels in a smaller product region near the beginning of the segment, as well as the end.
  For the initial and the final segment of the geodesic, the witness $\os$ is too big for the subsurface the geodesic is actually traveling in.
  This suggest that a better choice of a witness family would be to include both $V$ and $\gamma V$ as witnesses too.
  We can take this approach further, and include every subsurface in $\Upsilon(x,y)$ as a witness: this will still form a complete witness family.
  However, this approach also leads to multiple witnesses nested within one another, which is something we want to avoid as much as possible.

  The drawback of Example \ref{ex:uninsulated} motivates the next property we will require from witness families, which is the notion of being \emph{insulated}.
  Informally, a witness family $\Omega(x,y)$ is insulated if all the maximal active subsurfaces that are active near the beginning or end of $[x,y]$ are also witnesses.
  \begin{definition}[Insulated witness family]
    \label{def:insulated}
    A witness family $\Omega(x,y)$ is insulated if for every $E \in \Omega(x,y)$,
    all subsurfaces $V \sqsubset E$ satisfying the following properties are also witnesses.
    \begin{enumerate}[(i)]
    \item $V \in \Upsilon(x,y)$.
    \item $d_E(\cC(V), x) \leq 9\mathbf{C}$, or $d_E(\cC(V), y) \leq 9 \mathbf{C}$, where we consider $\cC(V)$ to be a subset of $\cC(E)$.
    \item $V$ is topologically maximal among the subsurfaces that satisfy (i) and (ii).
    \end{enumerate}
  \end{definition}

  Once we have an insulated witness family, we can order a nested pair of witnesses $W \sqsubset V$ based on whether $W$ is active near the beginning or the end of the geodesic $[x,y]$ projected to $\cC(V)$.

  \begin{definition}[Subordering]
    Let $[x,y]$ be a geodesic in $\teich(\os)$ and $\Omega(x,y)$ a complete insulated witness family associated to $[x,y]$.
    Then for each nested pair of witnesses $W \sqsubset V$, a subordering is an assignment of exactly one of the following two possibilities:
    \begin{enumerate}[(i)]
    \item $W \swarrow V$
    \item $V \searrow W$
    \end{enumerate}
    The orderings $\swarrow$ and $\searrow$ satisfy the following properties.
    \begin{enumerate}[(i)]
    \item If $Z$, $V$, and $W$ are witnesses such that $Z \sqsubset V \sqsubset W$, then $Z \swarrow W$ iff $V \swarrow W$ (equivalently, $W \searrow Z$ iff $W \searrow V$).
    \item If $Z$, $V$ and $W$ are witnesses such that $Z \swarrow V \searrow W$, then $Z \pitchfork_V W$ and $Z \lessdot W$.
    \item If $Z$ and $V$ are witnesses, and $W$ an active subsurface such that $Z \swarrow V \lessdot W$, or $W \lessdot V \searrow Z$, then $Z \pitchfork_V W$.
    \item If $Z$ and $V$ are witnesses, such that $Z \swarrow V$ (or $Z \searrow V$), then there does not exist any active subsurface $W$ such that the $\Omega$-closure of $W$ is $V$ and $W \lessdot Z$ (or $Z \lessdot W$).
    \end{enumerate}
  \end{definition}

  Here $Z \pitchfork_V W$ refers to notion of two subsurfaces cutting each other relative to $V$.
  \begin{definition}[Relative cutting]
    Given a subsurface $V$ of $S$, we say two subsurfaces $Z$ and $W$ of $S$ \emph{cut relative} to $V$ if for any subsurfaces $Z^{\prime} \sqsubset Z$ and $W^{\prime} \sqsubset W$ that intersect $V$, $Z^{\prime} \pitchfork W^{\prime}$.
  \end{definition}

  We now provide some motivation for the various conditions that appear in the above definition.
  First of all, when we see $Z \swarrow W$, we are to read that as \emph{the geodesic $[x,y]$ makes progress in the nested subsurface $Z$, before making progress in the supersurface $W$}.
  Similarly, when we see $W \searrow Z$, we are to read that as \emph{the geodesic $[x,y]$ make progress in the supersurface $W$ before making progress in the nested subsurface $Z$}.
  With this description of the subordering, conditions (i) and (iv) of the definition are easy to understand. The conditions (ii) and (iii) let us upgrade $\swarrow$ and $\searrow$ to transversality and time-ordering.
  A more intuitive reading of condition (ii) for instance would be, if $Z \swarrow V \searrow W$, that means the geodesic makes progress in $Z$ before $V$, and then makes progress in $W$.
  That means if we just look at $V$ and $W$, it makes progress in $V$ and then $W$. And since neither of them are nested in the other, the only way they can be time-ordered is by cutting relative to $V$.

  We can also see how the subordering on a witness family interacts with the witness family being insulated: recall the pair of witnesses $V \sqsubset E$ from Definition \ref{def:insulated}.
  \begin{itemize}
  \item If $d_E(\cC(V), x) \leq 9 \mathbf{C}$, then $V \swarrow E$, since the geodesic makes progress in $V$ before $E$.
  \item If $d_E(\cC(V), Y) \leq 9 \mathbf{C}$, then $E \searrow V$, since the geodesic makes progress in $E$ before $V$.
  \end{itemize}
  However, the above example does not capture all the ways in which we can have $V \swarrow E$ or $E \searrow V$.
  Consider a decomposition of a Teichmüller geodesic by the active intervals corresponding to witnesses illustrated in \autoref{fig:subordering-examples}.
  \begin{figure}[h]
    \centering
    \def\svgscale{1}
    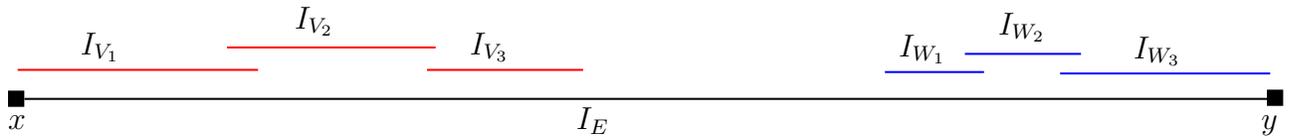

    \caption{Examples of $V_i \swarrow E$ and $E \searrow W_i$.}
    \label{fig:subordering-examples}
  \end{figure}

  In this example, all of the witnesses $V_i$ satisfy $V_i \swarrow E$, and all of the witnesses $W_i$ satisfy $E \searrow W_i$, but $d_E(\cC(V_i), x)$ need not be less than $9\mathbf{C}$ for $i=2$ or $i=3$, and similarly, $d_E(\cC(W_i), y)$ need not be less than $9 \mathbf{C}$ for $i=2$ or $i=3$.

  In fact, the above example illustrates that for witnesses $V$ that are not within distance $9 \mathbf{C}$ from one of the endpoints, the choice between assigning $V \swarrow E$ and $E \searrow V$ is ambiguous, which is what the property of being \emph{wide} tries to fix.
  The property of being \emph{wide} also tells when an active subsurface $V$ nested in a witness $E$ contributes to $E$: this happens when $V$ appears in the ``middle'' of the segment $[x,y]$.


  \begin{definition}[Wide witness families]
    An insulated complete subordered witness family is wide if for each $V$ in the witness family, both of the following quantities are at most $\frac{N_V}{3}$.
    \begin{itemize}
    \item For $W$ a witness such that $W \swarrow V$, the quantity $\mathrm{diam}_V(x, \cC(W))$.
    \item For $W$ a witness such that $V \searrow W$, the quantity $\mathrm{diam}_V(y, \cC(W))$.
    \end{itemize}
  \end{definition}

  The idea behind a wide witness family is to create a buffer zone of length at least $\frac{N_V}{3}$ in the middle of the projection of the geodesic to $\cC(V)$ for any witness $V$ such that:
  \begin{itemize}
  \item If any subsurface $W$ is active to the left of the buffer zone, it contributes to a subsurface $Z$ such that $Z \swarrow V$.
  \item If any subsurface $W$ is active to the right of the buffer zone, it contributes to a subsurface $Z$ such that $V \searrow Z$.
  \item If a subsurface $W$ is active within the buffer zone, it contributes to $V$.
  \end{itemize}

  The upshot of defining wide, insulated, subordered, and complete witness families (which will be abbreviated to WISC witness families) is that it gives us a better idea the order in which progress is made in various active subsurfaces.
  If we were working with just the collection of active subsurfaces, the only time we can tell if a geodesic makes progress in a subsurface $V$ followed by the subsurface $W$ is when $V \pitchfork W$.
  When working with a WISC witness family, we can do that, but we can also make similar statements about pairs of nested witnesses $V \sqsubset W$, namely we can have either $V \swarrow W$ or $W \searrow V$.

  The following lemma asserts that WISC witness families exist, and their cardinality can be uniformly bounded.
  \begin{lemma}[Lemmas 7.29 and 7.30 from \cite{dowdall2023lattice}]
    Let $\os$ be surface, and $[x,y]$ a geodesic segment in $\teich(\os)$.
    Then there exists a WISC witness family $\Omega(x,y)$ for $[x,y]$.
    Furthermore, the cardinality of $\Omega(x,y)$ depends only on $\os$, and not the points $x$ and $y$.
  \end{lemma}
  \begin{remark}
    While the statements of Lemmas 7.29 and 7.30 in \textcite{dowdall2023lattice} are for orientable surfaces, they go through without any changes for non-orientable surfaces as well.
  \end{remark}

  We now get to the raison d'\^etre of witness families: turning points on the geodesic segment $[x,y]$ in $\teich(\os)$ into points in $\teich(V)$, where $V$ is a witness in $\Omega(x,y)$.
  We will do so by assigning to each point $w$ in a neighbourhood of $[x,y]$ a point $\wt{w}_Z$ in $\cC(Z)$ for all subsurfaces $Z$ contained in $V$, and then showing this assignment is consistent.
  Then the realization theorem (Theorem \ref{thm:consistency-realization}) will give us a point $\widehat{w}_V^{\Omega}$ in $\teich(V)$ which has the same projections in $\cC(Z)$ as the original point $W$.
  \begin{definition}[Projection tuple]
    \label{def:projection-tuple}
    Let $\Omega(x,y)$ be a WISC witness family for a Teichmüller geodesic $[x,y]$ in $\teich(\os)$.
    Let $w$ be a point in $\teich(\os)$ satisfying the following bound for every subsurface $V$.
    \begin{align*}
      d_{V}(x, w) + d_{V}(w, y) \leq d_{V}(x, y) + 9 \mathbf{C}
    \end{align*}
    Then for any $U \in \Omega(x,y)$, the projection tuple $\wt{w}$ of $w$ is the point in $\prod_{Z \sqsubseteq U} \cC(Z)$ given by the following formula (where $\pi_Z$ is the usual projection map from $\teich(\os)$ to $\cC(Z)$).
    \begin{align*}
      \wt{w}_Z =
      \begin{cases}
        \pi_Z(y), & \text{if $Z \in \Upsilon(x,y)$ and $\colsup{Z}{\Omega} \swarrow U$} \\
        \pi_Z(x), & \text{if $Z \in \Upsilon(x,y)$ and $U \searrow \colsup{Z}{\Omega}$} \\
        \pi_Z(w), & \text{otherwise}
      \end{cases}
    \end{align*}
  \end{definition}

  Observe that this is different from the usual projection map from $\teich(\os)$ to $\cC(Z)$: for subsurfaces $Z$ that contribute to a witness nested in $U$, and consequently, $\colsup{Z}{\Omega} \swarrow U$ or $U \searrow \colsup{Z}{\Omega}$, we change the projection from $\pi_Z(w)$ to $\pi_Z(y)$ or $\pi_Z(x)$ respectively.

  This new projection map, despite being a modification of the usual projection map, is still consistent.
  \begin{proposition}[Proposition 8.4 of \cite{dowdall2023lattice}]
    The projection tuple $\wt{w}_Z$ is $k$-consistent for some $k$ depending only on $\mathbf{C}$.
  \end{proposition}

  Using the above proposition, and the realization theorem for non-orientable surfaces (Theorem \ref{thm:consistency-realization}), we can turn a projection tuple into a point in $\teich(U)$, which \textcite{dowdall2023lattice} refer to as \emph{resolving a point $w$ in $\teich(U)$}.
  \begin{definition}[Resolution point]
    Let $[x,y]$ be a geodesic segment in $\teich(\os)$, and $\Omega(x,y)$ an associated WISC witness family.
    For $w \in \left\{ x, y \right\}$, and $U \in \Omega(x,y)$, we define $\widehat{w}_U^{\Omega}$ as follows.
    \begin{itemize}
    \item If $U$ is non-annular, then $\widehat{w}_U^{\Omega} \in \teich(U)$ is the thick point whose projections to $\cC(V)$ for $V \sqsubseteq U$ are coarsely equal to the projection tuple $\widehat{w}_U$ (which exists due to the Realization theorem (Theorem \ref{thm:consistency-realization})).
    \item If $U$ is annular, then $\wt{w}_U$ is an element of $\mathbb{Z}$, and we set $\widehat{w}_U^{\Omega}$ to be the point in $\mathbb{H}$ whose twist coordinate is $\wt{w}_U$, and whose length coordinate is $\displaystyle \frac{1}{\min\left( {\vept}, \ell_{\partial U}(w) \right)}$.
    \end{itemize}
  \end{definition}

  We can now define the complexity length associated to a witness family $\Omega$.

  \begin{definition}[Complexity of witness family]
    Let $[x, y]$ be a geodesic segment in $\teich(\os)$, and $\Omega$ an associated WISC witness family.
    The complexity $\mathfrak{L}_{\Omega}(x,y)$ of $\Omega$ is the following quantity.
    \begin{align*}
      \mathfrak{L}_{\Omega}(x,y) \coloneqq \sum_{U \in \Omega} \hNP^{\ast}(U) \cdot d_{\teich(U)}(\widehat{x}_{U}^{\Omega}, \widehat{y}_{U}^{\Omega})
    \end{align*}
    Here, $\hNP^{\ast}(U)$ is the net point growth entropy for $\teich(U)$ when $U$ is non-annular, and when $U$ is annular, $\hNP^{\ast}(U)$ is $1$ when both $\widehat{x}_{U}^{\Omega}$ and $\widehat{y}_{U}^{\Omega}$ are ${\vept}$-thick, and $2$ if not.
  \end{definition}

We now address why we used a modified version of the projection map in Definition \ref{def:projection-tuple} instead of the usual projection map to curve complexes, by revisiting Example \ref{ex:uninsulated}.
\begin{example}
  \label{ex:need-for-modified-projection}
  Let $\gamma$ be a pseudo-Anosov mapping class on $\os$, with large translation distance on $\cC(\os)$ and small projections elsewhere.
  Let $\delta$ be a reducible mapping class, which is psuedo-Anosov on a subsurface $V$, with large translation distance on $\cC(V)$ and small translation distance everywhere.
  Let $x$ be a point in $\teich(\os)$ such that $\partial V$ is a component of the short marking on $x$, and $y = \delta \gamma \delta^{-1} x$.
  We first verify that $\Upsilon(x,y) = \left\{ \os, V, \delta \gamma V \right\}$.
  To see this, we consider the following tuple of points $(x, \delta x, \delta \gamma x, \delta \gamma \delta^{-1}x)$.
  We claim that each of the points in the tuple lies coarsely on the geodesic $[x,y]$.
  To see this, we compute the projections of adjacent pairs of points in tuple to various curve complexes.
  \begin{itemize}
  \item[-] $[x, \delta x]$ has large projections on $\cC(V)$ and small projections on other curve complexes.
  \item[-] $[\delta x, \delta \gamma x]$ has large projections on the $\cC(\delta \os)$, which is the same as $\cC(\os)$, and small projections elsewhere.
  \item[-] $[\delta \gamma x, \delta \gamma \delta^{-1} x]$ has large projections on $\cC(\delta \gamma V)$, and small projections elsewhere.
  \end{itemize}
  Consider the path $\kappa$ obtained by concatenating Teichmüller geodesics between $x$ and $\delta x$, $\delta x$ and $\delta \gamma x$, and $\delta \gamma x$ and $\delta \gamma \delta^{-1} x$.
  The curve complex calculations above show that the projection of this path to any curve complex is a quasi-geodesic.
  That indicates this path is a \emph{hierarchy path}, in the language of hierarchically hyperbolic spaces.
  Furthermore, $\kappa$ does not have large projections on disjoint subsurfaces: this indicates the hierarchy path $\kappa$ fellow travels the geodesic $[x, y]$.

  Let $\Omega(x, y) = \Upsilon(x,y) = \left\{ \os, V, \delta \gamma V \right\}$.
  One can verify that this is a WISC witness family for $[x,y]$.
  Furthermore, we have that $V \swarrow \os$ and $\os \searrow \gamma V$.

  We now compute the resolution of the points $x$ and $y$ in $\teich(V)$, $\teich(\gamma V)$, and $\teich(\os)$.
  Observe that the geodesic $[x,y]$ almost immediately moves into a product region associated to $V$ at the beginning, leaves that product region at some point $w$ along the geodesic, and then enters the product region associated to $\gamma V$ at some point $z$, and then stays in that product region almost all the way up to the end.
  We will abuse notation slightly, and refer to $x$ and $w$ as points in $\teich(V)$, when we mean their projection via the product region map, and $z$ and $y$ will refer to points in $\teich(\gamma V)$
  Resolving points in $\teich(V)$ and $\teich(\gamma V)$ is easy, since there's no other witnesses nested in them, which means the projection tuple for those subsurfaces is the usual projection map.
  \begin{align*}
    \widehat{x}_V^{\Omega} &= x \\
    \widehat{y}_V^{\Omega} &= w \\
    \widehat{x}_{\gamma V}^{\Omega} &= z \\
    \widehat{y}_{\gamma V}^{\Omega} &= y \\
  \end{align*}
  To resolve points in $\teich(\os)$, we have to use our modified projection map, instead of the usual one.
  Doing so, the points $x$ and $y$ resolve in the following manner.
  \begin{align*}
    \widehat{x}_{\os}^{\Omega} = w \\
    \widehat{y}_{\os}^{\Omega} = z
  \end{align*}

  With these resolutions, we get the following estimate for complexity in terms of Teichmüller distance.
  \begin{align*}
    \mathfrak{L}_{\Omega}(x,y) &= \hNP^{\ast}(V) \cdot d_{\teich(V)}(x, w) + \hNP^{\ast}(\os) \cdot d_{\teich(\os)}(w, z) + \hNP^{\ast}(\gamma V) \cdot d_{\teich(\gamma V)}(z, y) \\
    &< \hNP^{\ast}(\os) \cdot d_{\teich(\os)}(x, y)
  \end{align*}
  Compare this to the complexity $\mathfrak{L}_{\Omega}^{\prime}$ we would have gotten if we used the usual projection map instead of the modified projection map.
  \begin{align*}
    \mathfrak{L}_{\Omega}^{\prime}(x,y) &= \hNP^{\ast}(V) \cdot d_{\teich(V)}(x, w) + \hNP^{\ast}(\os) \cdot d_{\teich(\os)}(x, y) + \hNP^{\ast}(\gamma V) \cdot d_{\teich(\gamma V)}(z, y) \\
                                        &> \hNP^{\ast}(\os) \cdot d_{\teich(\os)}(x, y)
  \end{align*}
\end{example}

If we use Rafi's distance formula to estimate the distance between the resolution points $\widehat{x}_V^{\Omega}$ and $\widehat{x}_V^{\Omega}$ we do not get very good bounds for $d_{\teich(V)}(\widehat{x}_V^{\Omega}, \widehat{x}_V^{\Omega})$: at best, we accrue multiplicative and additive errors.
For our applications however, the most we can tolerate is additive error.
To do this, we will need to refine to notion of active interval for a subsurface to something more useful for the estimate: the \emph{contribution set} $\mathcal{A}_V^{\Omega}$ of a witness $V$.
We first define two intermediate collections of subintervals of a geodesic segment $[x,y]$.
\begin{align*}
  M(V) &\coloneqq \bigcup \left\{ \mathcal{I}_W^{{\vept}} \mid \text{$W \in \Omega$ with $W \sqsubset V$ }  \right\} \\
  C(V) &\coloneqq \bigcup \left\{ \mathcal{I}_Z^{{\vept}} \mid \text{$Z$ contributes to $V$} \right\}
\end{align*}

\begin{definition}[Contribution set]
  For a witness $V \in \Omega$, the contribution set $\mathcal{A}_V^{\Omega}$ is a subset of the geodesic segment $[x,y]$ defined in the following manner.
  \begin{align*}
    \mathcal{A}_V^{\Omega} \coloneqq \left( \mathcal{I}_V^{\vept} \setminus M(V) \right) \cup C(V)
  \end{align*}
\end{definition}
\begin{remark}
  The reason we remove $M(V)$ and then later add $C(V)$ again is because it is possible for several different subsurfaces to be active at the same time: orthogonal subsurfaces for instance.
  One can have $V$ and $W$ as witnesses, with $W \sqsubset V$, and $Z \sqsubset V$ an active subsurface but not a witness, such that $W \perp Z$ with $\mathcal{I}_W^{\vept}$ and $\mathcal{I}_Z^{\vept}$ overlapping.
  In that case, removing $M(V)$ would also remove part of $\mathcal{I}_Z^{\vept}$, and adding back $C(V)$ would add back the deleted portion.
\end{remark}

The following theorem estimates $d_{\teich(V)}(\widehat{x}_V^{\Omega}, \widehat{y}_V^{\Omega})$ using $\mathcal{A}_V^{\Omega}$.
\begin{theorem}[Theorem 9.4 of \cite{dowdall2023lattice}]
  \label{thm:contribution-set-resolution-distance}
  There exists a uniform constant $C$ such that the following bound holds for any $x$, $y$, and witness $V$.
  \begin{align*}
d_{\teich(V)}(\widehat{x}_V^{\Omega}, \widehat{y}_V^{\Omega}) \leq \int_x^y \mathbbm{1}_{\mathcal{A}_V^{\Omega}} + C
  \end{align*}
\end{theorem}

Contribution sets help us make precise the notion of ``overcounting'' when multiple product regions are active at the same time.
More precisely, when a segment of $[x,y]$ is a part of two or more contribution sets, that segment shows up multiple times when computing $\mathfrak{L}_{\Omega}(x,y)$, thanks to Theorem \ref{thm:contribution-set-resolution-distance}.
If the overlapping segment is sufficiently long, one could even end up having $\mathfrak{L}_{\Omega}(x,y) > \hNP^\ast(\os) \cdot d_{\teich(\os)}(x, y)$, which as we will see, leads to a worse count for net points than the usual methods.
This phenomenon of contribution sets overlapping is called \emph{badness}, and while we will not be able to eliminate it entirely, we will be able to minimize it.

\begin{definition}[Bad set]
  We say a point $p$ in $\mathcal{A}_V^{\Omega}$ is bad if there exists some other witness $W$ such that $p$ also belongs in $\mathcal{A}_W^{\Omega}$.
  The bad set $\mathcal{B}_V^{\Omega}$ denotes the set of all bad points in $\mathcal{A}_V^{\Omega}$, and $\left| \mathcal{B}_V^{\Omega} \right|$ denotes the total length of this set, when we think of $\mathcal{B}_V^{\Omega}$ as a subset of the geodesic segment $[x,y]$.
\end{definition}

For our applications, we won't need to eliminate badness entirely, or even bound the length of the bad set uniformly: it will suffice to show that the length of the bad set is a very small multiple of $d_{\teich(\os)}(x,y)$.

\begin{definition}[Admissible and limited]
  A witness family $\Omega$ associated to a geodesic segment $[x,y]$ is said to be:
  \begin{itemize}
  \item \emph{admissible} if $\displaystyle \left| \mathcal{B}_V^{\Omega} \right| \leq \frac{d_{\teich(\os)}(x,y)}{K_V \mathbf{C}}$, for all $V \in \Omega$, and some constants $K_V$ that only depend on the topological type of $V$.
  \item \emph{limited} if $\left| \Omega \right|$ is uniformly bounded, independent of $x$ and $y$.
  \end{itemize}
\end{definition}
\begin{remark}
  Our definition of limited is a weaker version of Definition 10.7 from \textcite{dowdall2023lattice}, but since we don't need the stronger version, we present this version instead.
\end{remark}

\textcite{dowdall2023lattice} prove that WISC witness families that are admissible and limited exist. They call these witness families WISCAL witness families.

\begin{proposition}[Section 10.3 of \cite{dowdall2023lattice}]
  \label{prop:existence-of-wiscal}
  For all $[x,y]$, there exists an associated WISC witness family that is also admissible and limited.
\end{proposition}

For WISCAL witness families, the following result relating complexity and Teichmüller distance follows easily from Theorem \ref{thm:contribution-set-resolution-distance} and the definition of admissible.

\begin{proposition}
  \label{prop:complexity-length-inequality}
  If $\Omega$ is a WISCAL witness family associated to $[x,y]$, then the following inequality holds.
  \begin{align*}
    \mathfrak{L}_{\Omega}(x,y) \leq \left( \hNP(\os) + \frac{K}{\mathbf{C}} \right) d_{\teich(\os)}(x, y) + K \mathbf{C}
  \end{align*}
  Here, $K$ is some uniform constant depending only on $\os$.
\end{proposition}




We now define complexity length, which follows from the definition of the complexity of a witness family.

\begin{definition}[Complexity length]
  \label{defn:complexity-length}
  For a pair of points $x$ and $y$ in $\teich(\os)$, the complexity length $\mathfrak{L}(x,y)$ is defined to be the following.
  \begin{align*}
    \mathfrak{L}(x,y) \coloneqq \inf_{\Omega} \mathfrak{L}_{\Omega}(x,y)
  \end{align*}
  Here, we take the infimum over all WISCAL witness families for $[x,y]$.
\end{definition}

With the machinery of complexity length set up, it is now possible to count net points with respect to complexity length.

\begin{theorem}[Theorem 12.1 of \cite{dowdall2023lattice}]
  \label{thm:counting-with-complexity}
  For any large enough $\mathbf{C} > 0$, and any $\veperr > 0$, there exists an polynomial function $p(r)$, and $r > 0$ large enough such that the following bound holds for net points in $\teich(\os)$.
  \begin{align*}
    \# \left( y \in \net \mid \mathfrak{L}(x,y) \leq r \right) \leq p(r) \cdot \exp\left( (1+\veperr)r \right)
  \end{align*}
\end{theorem}

\begin{remark}
  The above theorem is a weaker version of the theorem that appears in \textcite{dowdall2023lattice}: their version does not have the $\veperr$.
  The reason we have the weaker version is that in the proof of their theorem, they count the number of net points in $\teich(V)$ in a ball of radius $R$, where $V$ is a witness, using Theorem \ref{thm:abem}, which gives them that the number of net points is equal, up to multiplicative error, to $\exp(\hNP(U) R)$.
  Since Theorem \ref{thm:abem} only holds orientable surfaces, and we want to state our results for non-orientable surfaces as well, we will need to use a weaker counting result to count net points in $\teich(V)$, namely the following bound, which holds for any $\vepent > 0$ and large enough $R$.
  \begin{align*}
    \#\left( y \in \net \mid d_{\teich(V)}(x, y) \leq R \right) \leq \exp\left( (\hNP(V) + \vepent ) R \right)
  \end{align*}
\end{remark}

We sketch out a proof of Theorem \ref{thm:counting-with-complexity} below: the proof proceeds identically to the proof in \textcite{dowdall2023lattice}, except at one point, where we plug in our weaker bound for net points in $\teich(V)$ for witnesses $V$.

\begin{proof}[Sketch of proof for Theorem \ref{thm:counting-with-complexity}]
  For each $y \in \net$ such that $\mathfrak{L}(x,y) \leq r$, we have a WISCAL witness family $\Omega$ such that $\mathfrak{L}_{\Omega}(x,y) \leq r$.
  We can turn that witness family into a graph in the following manner.
  \begin{itemize}
  \item Add a vertex for every witness $V \in \Omega$.
  \item Label the vertex associated with $V$ with the tuple $(\hNP^{\ast}(V), \lfloor d_{\teich(V)}(\widehat{x}_{V}^{\Omega}, \widehat{y}_{V}^{\Omega}) \rfloor)$.
  \item If we have a pair of witnesses $V \swarrow W$, we join the vertices associated to them with a directed edge labeled ``SW'': $V \xrightarrow{SW} W$.
  \item If we have a pair of witnesses $W \searrow V$, we join the vertices associated to them with a directed edge labeled ``SE'': $W \xrightarrow{SE} V$.
  \item If we have a pair of witnesses $W \pitchfork V$, with $W \lessdot V$, we join the vertices associated to them with a directed edge labelled ``P'': $W \xrightarrow{P} V$.
  \end{itemize}

  We first count how many distinct possibilities are there for such labeled graphs that correspond to $y$ for which $\mathfrak{L}(x,y) \leq r$.
  Since the cardinality of a WISCAL family is uniformly bounded, there are at most $k$ many vertices, for some constant $k$.
  As for the labels on the vertices, there are at most $\frac{r}{\hNP^{\ast}(V)}$ possibilities for a label on vertex which corresponds to a subsurface which is homeomorphic to $V$.
  From this, we conclude that there are at most $p(r)$ possibilities for the combinatorial type of the graph, where $p(r)$ is a polynomial in $r$.

  It will suffice to compute how many distinct net points give rise to witness families whose graph is of a given type.
  To do so, we consider \emph{initial subsets} of the graph, i.e.\ a subset $\mathcal{W}$ of the vertices $\mathcal{V}$ of the graph such that there is no directed edge from $\mathcal{V} \setminus \mathcal{W}$ to $\mathcal{W}$.

  Given an initial subset $\mathcal{W}$ of the graph, we construct points $y$ such that the witness family associated to $[x,y]$ has the combinatorial type $\mathcal{W}$.
  We then consider an enlargement of $\mathcal{W}$ by one-additional vertex $v$, such that the enlargement is still an initial subset.
  \begin{claim*}
    The entire graph $\mathcal{V}$ can be built up from such one-step enlargements.
  \end{claim*}

  We then count the number of net points whose associated witness families have the combinatorial type $\mathcal{W} \cup \left\{ v \right\}$, after we fix one witness family associated to $\mathcal{W}$.
  More concretely, let $w$ be a point such that the combinatorial type of the witness family associated to $[x,w]$ is $\mathcal{W}$.
  Suppose now that we add a vertex $(h, r_0)$ to the graph $\mathcal{W}$.
  To extend $\Omega(x,w)$ so that its combinatorial type is $\mathcal{W} \cup \left\{ (h, r_0) \right\}$, we need to add a witness $U$ whose net point entropy is $h$, and a point $y \in \teich(U)$ such that the following holds.
  \begin{align*}
    d_{\teich(U)}(\widehat{x}_U^{\Omega}, y) \leq r_0
  \end{align*}

  There are only finitely many choices for such subsurfaces $U$ (because their boundary curves must get short near $w$), and once we've made a choice of $U$, we have a choice $\exp\left( (\hNP(V) + \vepent) r_0 \right)$ points for $y$.

  Multiplying out the counts for each vertex added, we get the following estimate for the cardinality associated to each combinatorial type.
  \begin{align*}
    \#\left( y \mid \text{$\Omega(x,y)$ has combinatorial type $\mathcal{V}$} \right) &= \sum_{(h, s) \in \mathcal{V}} \exp\left( (h + \vepent)s \right) \\
    &\leq \exp\left( (1 + \veperr)r \right)
  \end{align*}
  We get the second inequality by picking $\vepent$ small enough, and observing that $\sum hs \leq r$.
\end{proof}

\subsection{Linear Gap for Bad Points}
\label{sec:linear-gap-bad}

In this subsection, we will prove our main result involving complexity length: on the complexity length of bad points.

\begin{theorem}
  \label{thm:linear-gap}
  Suppose that for all proper subsurfaces $V$ of $\os$, the following inequality holds.
  \begin{align*}
    \hNP(V) < \hNP(\os)
  \end{align*}
  Then for any $\vepb > 0$, there exists $c > 0$, and $R$ large enough, such that for any bad point $y$, i.e.\ a point in $\net_b(p, R, \vepb)$, the following upper bound holds for the complexity length between $p$ and $y$.
  \begin{align*}
    \mathfrak{L}(p, y) \leq \hNP(\os) (1 - c) R
  \end{align*}
\end{theorem}

\begin{proof}

Let $\Omega$ be a WISCAL witness family for $[p, y]$: the proof of Theorem \ref{thm:linear-gap} splits into two cases depending on whether the surface $\os$ is a witness in $\Omega$ or not.

The case where $\os$ is a witness is harder, so we deal with that first.


We consider the triple of points $(p, \widehat{y}_{\os}^{\Omega}, y)$, and first estimate $\mathfrak{L}(p, \widehat{y}_{\os}^{\Omega})$.

By applying Proposition \ref{prop:complexity-length-inequality}, we get a bound for $\mathfrak{L}(p, \widehat{y}_{\os}^{\Omega})$.
\begin{align}
  \label{eq:eq2}
  \mathfrak{L}(p, \widehat{y}_{\os}^{\Omega}) \leq \left( \hNP(\os) + \frac{K}{\mathbf{C}} \right) (d_{\teich(\os)}(p, \widehat{y}_{\os}^{\Omega})) + K \mathbf{C}
\end{align}

We next estimate $\mathfrak{L}(\widehat{y}_{\os}^{\Omega}, y)$: we claim that there exists a WISCAL witness family $\Omega^{\prime}$ for $[\widehat{y}_{\os}^{\Omega}, y]$ that does not have $\os$ as a witness.
The first thing we need in order to get such a witness family is verify that on the Teichmüller geodesic $[\widehat{y}_{\os}^{\Omega}, y]$, $\os$ is not an active subsurface.
This follows from the fact that the $\os$-coordinate in the projection tuple of $y$ is equal to the $\os$-coordinate of the projection of $\widehat{y}_{\os}^{\Omega}$ to $\cC(\os)$, by construction of $\widehat{y}_{\os}^{\Omega}$.
We now use Proposition \ref{prop:existence-of-wiscal} to get a witness family which does not contain $\os$, since $\os$ is not an active subsurface.

Since the witness family $\Omega^{\prime}$ does not have $\os$ as a witness, we can do better than Proposition \ref{prop:complexity-length-inequality} when estimating $\mathfrak{L}(\widehat{y}_{\os}^{\Omega}, y)$.
We have from our hypothesis that $\hNP(\os) > \hNP(V)$, so there exists a constant $h$ such that $h < \hNP(\os)$ but $h > \hNP(V)$.
Using Theorem \ref{thm:contribution-set-resolution-distance}, we get the following estimate for $\mathfrak{L}(\widehat{y}_{\os}^{\Omega}, y)$.
\begin{align}
  \label{eq:eq3p}
\mathfrak{L}(\widehat{y}_{\os}^{\Omega}, y) \leq \left( h + \frac{K}{\mathbf{C}} \right) d_{\teich(\os)}(\widehat{y}_{\os}^{\Omega}, y) + K \mathbf{C}
\end{align}

From the triangle inequality for complexity length, we also have the following
\begin{align}
  \label{eq:eq4}
  \mathfrak{L}(p, y) \leq \mathfrak{L}(p, \widehat{y}_{\os}^{\Omega}) + \mathfrak{L}(\widehat{y}_{\os}^{\Omega}, y)
\end{align}
We plug in \eqref{eq:eq2} and \eqref{eq:eq3p} into \eqref{eq:eq4}.
\begin{align*}
  \mathfrak{L}(p, y) &\leq \left( \hNP(\os) + \frac{K}{\mathbf{C}} \right) (d_{\teich(\os)}(p, \widehat{y}_{\os}^{\Omega})) + \left( h + \frac{K}{\mathbf{C}} \right) d_{\teich(\os)}(\widehat{y}_{\os}^{\Omega}, y) +  2K \mathbf{C} \\
                     &= \left( \hNP(\os) + \frac{K}{\mathbf{C}} \right) (d_{\teich(\os)}(p, \widehat{y}_{\os}^{\Omega}) + d_{\teich(\os)}(\widehat{y}_{\os}^{\Omega}, y)) \\
                     &- \left( \hNP(\os) - h \right) d_{\teich(\os)}(\widehat{y}_{\os}^{\Omega}, y) \\
  &+  2K \mathbf{C}
\end{align*}

We now claim that $\widehat{y}_{\os}^\Omega$ is within a bounded distance of a point $q$ on $[p, y]$.
Consider the triple $(x, \widehat{y}_{\os}^{\Omega}, y)$: this is a strongly aligned tuple, i.e. its projections onto all curve complexes satisfy a coarse reverse triangle inequality (see Definition 3.21 of \cite{dowdall2023lattice}).
Lemma 9.10 of \cite{dowdall2023lattice} asserts the existence of $q$ on $[x,y]$ that satisfies the following properties.
\begin{itemize}
\item[-] For any subsurface $V$ such that $\overline{V}^{\Omega} = \os$, $d_V(\widehat{y}_{\os}^{\Omega}, q) \leq M$ for some fixed constant $M$.
\item[-] For any subsurfaces $V$ such that $\overline{V}^{\Omega} \swarrow \os$ or $\os \searrow \overline{V}^{\Omega}$, $q$ lies outside the active interval for $V$.
\end{itemize}
In the case where $\overline{V}^{\Omega} \swarrow \os$, we have that $d_V(q, y) \leq M$, since $q$ lies outside the active interval for $V$.
But note that by construction of the projection tuple, $\pi_V(\widehat{y}_{\os}^\Omega) = y$, thus $d_V(\widehat{y}_{\os}^\Omega, q) \leq M$.
Similarly, we have $d_V(\widehat{y}_{\os}^\Omega, q) \leq M$ for $\os \searrow \overline{V}^{\Omega}$ as well.
Thus, by Rafi's distance formula, we have that $\widehat{y}_{\os}^\Omega$ is within a bounded distance of $q$ on $[p, y]$: we have $d_{\teich(\os)}(p, \widehat{y}_{\os}^{\Omega}) + d_{\teich(\os)}(\widehat{y}_{\os}^{\Omega}, y) \leq R + J$, for some constant $J$.
Furthermore, $d_{\teich(\os)}(\widehat{y}_{\os}^{\Omega}, y) \geq \vepb R$, by the hypothesis of $y$ being a bad point, since $\widehat{y}_{\os}^{\Omega}$ is in the thick part of Teichmüller space.
This simplifies the expression for $\mathfrak{L}(x,y)$.
\begin{align*}
  \mathfrak{L}(x,y) &\leq \left( \hNP(\os) + \frac{K}{\mathbf{C}} \right) R - (\hNP(\os) - h)(\vepb) R + 2K \mathbf{C} \\
  &= R \left( \hNP(\os) - d + \frac{K}{\mathbf{C}} \right) + 2K\mathbf{C}
\end{align*}
Here, $d = \vepb \left( \hNP(\os) - h \right)$, which is a positive constant, since $\hNP(\os) > h$.
By picking $\mathbf{C}$ and $R$ large enough, we get the statement of the theorem, which proves the result in the first case of $\os$ being in the witness family.

When $\os$ is not in the witness family, we set $\widehat{y}_{\os}^{\Omega} = p$, and the rest of the proof follows identically.
\end{proof}


%% file: images/schematic.pdf_tex
\begingroup%
  \makeatletter%
  \providecommand\color[2][]{%
    \errmessage{(Inkscape) Color is used for the text in Inkscape, but the package 'color.sty' is not loaded}%
    \renewcommand\color[2][]{}%
  }%
  \providecommand\transparent[1]{%
    \errmessage{(Inkscape) Transparency is used (non-zero) for the text in Inkscape, but the package 'transparent.sty' is not loaded}%
    \renewcommand\transparent[1]{}%
  }%
  \providecommand\rotatebox[2]{#2}%
  \newcommand*\fsize{\dimexpr\f@size pt\relax}%
  \newcommand*\lineheight[1]{\fontsize{\fsize}{#1\fsize}\selectfont}%
  \ifx\svgwidth\undefined%
    \setlength{\unitlength}{491.09816664bp}%
    \ifx\svgscale\undefined%
      \relax%
    \else%
      \setlength{\unitlength}{\unitlength * \real{\svgscale}}%
    \fi%
  \else%
    \setlength{\unitlength}{\svgwidth}%
  \fi%
  \global\let\svgwidth\undefined%
  \global\let\svgscale\undefined%
  \makeatother%
  \begin{picture}(1,0.13551549)%
    \lineheight{1}%
    \setlength\tabcolsep{0pt}%
    \put(0,0){\includegraphics[width=\unitlength,page=1]{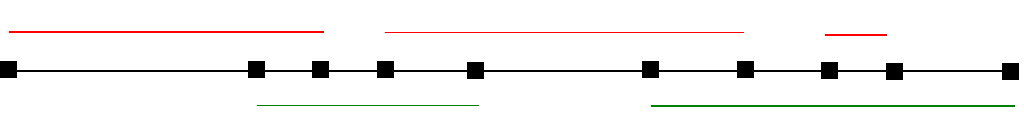}}%
    \put(0.00148406,0.04313039){\makebox(0,0)[lt]{\lineheight{1.25}\smash{\begin{tabular}[t]{l}$x$\end{tabular}}}}%
    \put(0.24171937,0.04313039){\makebox(0,0)[lt]{\lineheight{1.25}\smash{\begin{tabular}[t]{l}$z_1$\end{tabular}}}}%
    \put(0.30200714,0.04313039){\makebox(0,0)[lt]{\lineheight{1.25}\smash{\begin{tabular}[t]{l}$z_2$\end{tabular}}}}%
    \put(0.36378505,0.04313039){\makebox(0,0)[lt]{\lineheight{1.25}\smash{\begin{tabular}[t]{l}$z_3$\end{tabular}}}}%
    \put(0.45778045,0.04313039){\makebox(0,0)[lt]{\lineheight{1.25}\smash{\begin{tabular}[t]{l}$z_4$\end{tabular}}}}%
    \put(0.6272985,0.04313039){\makebox(0,0)[lt]{\lineheight{1.25}\smash{\begin{tabular}[t]{l}$z_5$\end{tabular}}}}%
    \put(0.7174027,0.04313039){\makebox(0,0)[lt]{\lineheight{1.25}\smash{\begin{tabular}[t]{l}$z_6$\end{tabular}}}}%
    \put(0.79987089,0.04313039){\makebox(0,0)[lt]{\lineheight{1.25}\smash{\begin{tabular}[t]{l}$z_7$\end{tabular}}}}%
    \put(0.86554008,0.04313039){\makebox(0,0)[lt]{\lineheight{1.25}\smash{\begin{tabular}[t]{l}$z_8$\end{tabular}}}}%
    \put(0.97855209,0.04313039){\makebox(0,0)[lt]{\lineheight{1.25}\smash{\begin{tabular}[t]{l}$y$\end{tabular}}}}%
    \put(0.14551856,0.12624235){\makebox(0,0)[lt]{\lineheight{1.25}\smash{\begin{tabular}[t]{l}$\pi_1$\end{tabular}}}}%
    \put(0.34557057,0.00293218){\makebox(0,0)[lt]{\lineheight{1.25}\smash{\begin{tabular}[t]{l}$\pi_2$\end{tabular}}}}%
    \put(0.53671445,0.12457039){\makebox(0,0)[lt]{\lineheight{1.25}\smash{\begin{tabular}[t]{l}$\pi_3$\end{tabular}}}}%
    \put(0.80477575,0.00940047){\makebox(0,0)[lt]{\lineheight{1.25}\smash{\begin{tabular}[t]{l}$\pi_4$\end{tabular}}}}%
    \put(0.81879625,0.12511896){\makebox(0,0)[lt]{\lineheight{1.25}\smash{\begin{tabular}[t]{l}$\pi_5$\end{tabular}}}}%
  \end{picture}%
\endgroup%

%% file: images/subordering-examples.pdf_tex
\begingroup%
  \makeatletter%
  \providecommand\color[2][]{%
    \errmessage{(Inkscape) Color is used for the text in Inkscape, but the package 'color.sty' is not loaded}%
    \renewcommand\color[2][]{}%
  }%
  \providecommand\transparent[1]{%
    \errmessage{(Inkscape) Transparency is used (non-zero) for the text in Inkscape, but the package 'transparent.sty' is not loaded}%
    \renewcommand\transparent[1]{}%
  }%
  \providecommand\rotatebox[2]{#2}%
  \newcommand*\fsize{\dimexpr\f@size pt\relax}%
  \newcommand*\lineheight[1]{\fontsize{\fsize}{#1\fsize}\selectfont}%
  \ifx\svgwidth\undefined%
    \setlength{\unitlength}{486.70986386bp}%
    \ifx\svgscale\undefined%
      \relax%
    \else%
      \setlength{\unitlength}{\unitlength * \real{\svgscale}}%
    \fi%
  \else%
    \setlength{\unitlength}{\svgwidth}%
  \fi%
  \global\let\svgwidth\undefined%
  \global\let\svgscale\undefined%
  \makeatother%
  \begin{picture}(1,0.10130323)%
    \lineheight{1}%
    \setlength\tabcolsep{0pt}%
    \put(0,0){\includegraphics[width=\unitlength,page=1]{subordering-examples.pdf}}%
    \put(0.00,0.01381817){\makebox(0,0)[lt]{\lineheight{1.25}\smash{\begin{tabular}[t]{l}$x$\end{tabular}}}}%
    \put(0.97030036,0.01381817){\makebox(0,0)[lt]{\lineheight{1.25}\smash{\begin{tabular}[t]{l}$y$\end{tabular}}}}%
    \put(0.44002148,0.01381817){\makebox(0,0)[lt]{\lineheight{1.25}\smash{\begin{tabular}[t]{l}$I_E$\end{tabular}}}}%
    \put(0.05608024,0.07185813){\makebox(0,0)[lt]{\lineheight{1.25}\smash{\begin{tabular}[t]{l}$I_{V_1}$\end{tabular}}}}%
    \put(0.22193379,0.09194654){\makebox(0,0)[lt]{\lineheight{1.25}\smash{\begin{tabular}[t]{l}$I_{V_2}$\end{tabular}}}}%
    \put(0.35753821,0.07191412){\makebox(0,0)[lt]{\lineheight{1.25}\smash{\begin{tabular}[t]{l}$I_{V_3}$\end{tabular}}}}%
    \put(0.68959865,0.07037313){\makebox(0,0)[lt]{\lineheight{1.25}\smash{\begin{tabular}[t]{l}$I_{W_1}$\end{tabular}}}}%
    \put(0.76664663,0.08732363){\makebox(0,0)[lt]{\lineheight{1.25}\smash{\begin{tabular}[t]{l}$I_{W_2}$\end{tabular}}}}%
    \put(0.87143186,0.06883215){\makebox(0,0)[lt]{\lineheight{1.25}\smash{\begin{tabular}[t]{l}$I_{W_3}$\end{tabular}}}}%
  \end{picture}%
\endgroup%

%% file: appendix.tex
\section{Geometry of $\teich(\no_g)$}
\label{sec:geom-of-teich}

In this section, we prove some standard results about the geometry of Teichmüller spaces of non-orientable surfaces that we use in Section \ref{sec:line-gap-compl}.
We do so by lifting the hyperbolic structures and markings on the non-orientable surfaces to their double covers, which give us points in the Teichmüller space and curve complex of the double cover.

The fact that these lifts are well-defined and respect the metric properties are encapsulated in the following two theorems.

\begin{theorem}[Isometric embedding of Teichmüller spaces (Theorem 2.1 of \cite{limitsetkhan})]
  \label{thm:i-embedding-teich-space}
  The map $i : \teich(\no_g) \to \teich(\os_{g-1})$ given by lifting the hyperbolic structure and marking from $\no_g$ to $\os_{g-1}$ is an isometric embedding.
  Furthermore, the image of $\teich(\no_g)$ in $\teich(\os_{g-1})$ is the subset of $\teich(\os_{g-1})$ is fixed by $\iota^{\ast}$, where $\iota^{\ast}$ is the map induced by the orientation reversing deck transformation $\iota$ on $\os_{g-1}$.
\end{theorem}

\begin{theorem}[Quasi isometric embedding of curve complexes (Lemma 6.3 from \cite{masur2013geometry})]
  \label{thm:qi-embedding-curve-complex}
  The map $\cC(\no_g) \to \cC(\os_{g-1})$ obtained by lifting curves in $\no_g$ to $\os_{g-1}$ is a quasi-isometric embedding.
\end{theorem}

We will use the above two theorems, along with Lemma \ref{lem:lifting-subsurfaces}, to reduce statements about the geometry of $\teich(\no_g)$ to statements about the geometry of $\teich(\os_{g-1})$.
However, we postpone the statement and the proof of Lemma \ref{lem:lifting-subsurfaces} until Section \ref{sec:dist-form-teichm}, since it's not required for Section \ref{sec:minskys-prod-regi}.

We set up some notation for this section.
\begin{itemize}
\item[-] $d(x,y)$ and $d(\wt{x}, \wt{y})$: Given points $x$ and $y$ in $\teich(\no_g)$, $d(x, y)$ is the distance in Teichmüller metric between them, and $d(\wt{x}, \wt{y})$ is the distance in $\teich(\os_{g-1})$ between their images, $\wt{x}$ and $\wt{y}$.
\item[-] $\pi_V(\mu_x)$ and $\pi_V(x)$: If $\mu_x$ is a marking/curve on a surface, the $\pi_V(\mu_x)$ denotes the subsurface projection to the subsurface $V$. If $x$ is a point in the Teichmüller space, the $\pi_V(x) = \pi_V(\mu_x)$, where $\mu_x$ is the Bers marking on $x$.
\item[-] $d_V(\mu_x, \mu_y)$ and $d_V(x, y)$: If $\mu_x$ and $\mu_y$ are markings/curves on a surface, and $V$ is a subsurface, then $d_V(\mu_x, \mu_y)$ refers to the curve complex distance between the subsurface projections of $\mu_x$ and $\mu_y$ in $\cC(V)$.
  When $x$ and $y$ are points in Teichmüller space, $d_V(x,y)$ refers to $d_V(\mu_x, \mu_y)$, where $\mu_x$ and $\mu_y$ are the Bers marking on $x$ and $y$.
\end{itemize}


\subsection{Minsky's Product Region Theorem}
\label{sec:minskys-prod-regi}

In this section, we prove a version of Minsky's product region theorem \cite[Theorem 6.1]{1077244446} for non-orientable surfaces.

We recall the following objects that were defined in Section \ref{sec:weak-conv-syst}.

\begin{enumerate}[(i)]
\item The multicurve $\gamma$ on $\no_g$.
\item The metric space $X_\gamma$, and the projection map $\Pi$.
\item The thin region $\teich_{\gamma \leq \vept}(\no_g)$.
\end{enumerate}

\begin{theorem}[Product region theorem for non-orientable surfaces]
  \label{thm:prno}
  For any $c >0$, there exists a small enough $\vept > 0$, such that the restriction of $\Pi$ to $\teich_{\gamma \leq \vept}(\no_g)$ is an isometry with additive error at most $c$, i.e.\ the following holds for any $x$ and $y$ in $\teich_{\gamma \leq \vept}(\no_g)$.
  \begin{align*}
    \left| d(x, y) - d_{X_{\gamma}}(\Pi(x), \Pi(y)) \right| \leq c
  \end{align*}
\end{theorem}

\begin{proof}
  We will prove this result by reducing the distance calculation in $\teich(\no_g)$ to a distance calculation in $\teich(\os_{g-1})$, where $\os_{g-1}$ is the orientation double cover, and invoking the classical product region theorem in that setting.

  We begin the proof by constructing some points in $\teich(\os_{g-1})$ and a multicurve on $\os_{g-1}$.
  Recall that $\teich(\no_g)$ isometrically embeds inside $\teich(\os_{g-1})$: let $\wt{x}$ and $\wt{y}$ denote the points in $\teich(\os_{g-1})$ that are the images of $x$ and $y$ under the embedding.
  Let $\wt{\gamma}$ denote the lift of the multicurve $\gamma$: if $\gamma_i$ is a two-sided curve, it will have two disjoint lifts in the cover, and if $\gamma_i$ is a one-sided curve, it will have single lift in the double cover.
  We have that the region $\teich_{\wt{\gamma} \leq \vept}(\os_{g-1}) \subset \teich(\os_{g-1})$ intersects the image of $\teich(\no_g)$ at the image of $\teich_{\gamma \leq \vept}(\no_g) \subset \teich(\no_g)$.
  Let $\iota$ denote the orientation reversing deck transformation on $\os_{g-1}$ which corresponds to the covering map.

  \begin{claim*}
    Let $\Pi_k$ denote the projection map from $\teich(\no_g)$ to the $k$\textsuperscript{th} component of $X_\gamma$, and $\wt{\Pi_k}$ denote the projection map from $\teich(\os_{g-1})$ to the lift of the $k$\textsuperscript{th} component of $\gamma$ to $\os_{g-1}$.
    This map is an isometric embedding.
    \begin{align*}
      d(\Pi_k(x), \Pi_k(y)) = d(\wt{\Pi_k}(\wt{x}), \wt{\Pi_k}(\wt{y}))
    \end{align*}
  \end{claim*}

  \begin{proof}[Proof of claim]
  We need to verify the claim on the three kinds of factors of $X_\gamma$.
  \begin{enumerate}[(i)]
  \item $\no_g \setminus \gamma$: The lift of $\no_g \setminus \gamma$ to $\os_{g-1}$ will have two components if $\no_g \setminus \gamma$ is orientable, which we call $S_1$ and $S_2$. Both $S_1$ and $S_2$ are homeomorphic to $\no_g \setminus \gamma$.
    If $\no_g \setminus \gamma$ is non-orientable, then its lift in $\os_{g-1}$ is the orientation double cover.

    In the first case, $\teich(\no_g \setminus \gamma)$ maps to the diagonal subspace in $\teich(S_1) \times \teich(S_2)$, and the metric on $\teich(S_1) \times \teich(S_2)$ is the $\sup$ metric.
    The space $\teich(\no_g)$ maps to the diagonal subspace because its image must be invariant under the map $\iota$, which isometrically swaps $S_1$ and $S_2$.
    This map is an isometric embedding, and thus for any points $x$ and $y$ in $\teich(\no_g \setminus \gamma)$, the distance between their images in $\teich(S_1) \times \teich(S_2)$ is the same as the distance in $\teich(\no_g \setminus \gamma)$.

    In the second case, we have that $\teich(\no_g \setminus \gamma)$ also isometrically embeds inside the Teichmüller space of its double cover, by Theorem \ref{thm:i-embedding-teich-space}, so the claim follows.
  \item $\gamma_i$ (for $\gamma_i$ two-sided): The lift of $\gamma_i$ in this case are two disjoint curves on $\os_{g-1}$, which are swapped by the deck transformation $\iota$.
    This means the $\mathbb{H}$-coordinate given by length and twist of $\gamma_i$ maps to the diagonal in $\mathbb{H} \times \mathbb{H}$, which correspond the length and twist around the two lifts.
    Since $\mathbb{H}$ mapped to the diagonal in $\mathbb{H} \times \mathbb{H}$ is an isometric embedding with $\sup$ metric, the claim follows in this case.
  \item $\gamma_i$ (for $\gamma_i$ one-sided): The lift of $\gamma_i$ in this case is a single curve $\wt{\gamma_i}$ on $\os_{g-1}$ which is left invariant by the deck transformation $\iota$.
    We will show that the twist coordinate around $\wt{\gamma_i}$ cannot be changed without leaving the image of $\teich(\no_g)$ in $\teich(\os_{g-1})$, i.e.\ any $\wt{x}$ and $\wt{y}$ have the same twist coordinate around $\wt{\gamma_i}$.
    Once we have established that, the claim will follow, since only the length coordinate of $\gamma_i$ can be changed, which corresponds to $\mathbb{R}_{>0}$.

    Suppose now that $x$ is a point in $\teich(\no_{g})$ and $\wt{x}$ the corresponding point in $\teich(\os_{g-1})$.
    Consider a pants decomposition on $\no_g$ that contains $\gamma_i$ as one of the curves. There is a unique one-sided curve $\kappa$ that intersects $\gamma_i$ and does not intersect any of the other pants curves.
    Let $\wt{\kappa}$ be the lift of $\kappa$ to $\os_{g-1}$: we will use this curve to measure twisting around $\wt{\gamma_i}$.
    Let $x^{\prime}$ be another point in $\teich(\os_{g-1})$ obtained by taking $\wt{x}$, and twisting by some amount around $\wt{\gamma_i}$, without changing the length of $\wt{\gamma_i}$.
    On $x^{\prime}$, the length of $\wt{\kappa}$ will be different from the length on $\wt{x}$.
    However, this means that $x^{\prime}$ is not contained in the image of $\teich(\no_g)$, since if it were, the length of $\wt{\kappa}$ would have to be the same, since that's the lift of the curve $\kappa$, whose length only depends on the length of $\gamma_i$.
  \end{enumerate}
  \end{proof}

  The following equality follows from the claim.
  \begin{align*}
    d_{X_{\wt{\gamma}}}(\Pi(\wt{x}), \Pi(\wt{y})) = d_{X_{{\gamma}}}(\Pi({x}), \Pi({y}))
  \end{align*}
  We also have that $\teich(\no_g)$ isometrically embeds into $\teich(\os_{g-1})$.
  \begin{align*}
    d(\wt{x}, \wt{y}) = d(x, y)
  \end{align*}
  And finally, have that the region $\teich_{\wt{\gamma} \leq \vept}(\os_{g-1}) \subset \teich(\os_{g-1})$ intersects the image of $\teich(\no_g)$ at the image of $\teich_{\gamma \leq \vept}(\no_g) \subset \teich(\no_g)$.
  Combining these three facts, and applying Minsky's product region theorem for orientable surfaces, the result follows.
\end{proof}


\subsection{Uniform Bounds for the Volume of a Ball}
\label{sec:unif-bounds-volume}


In this section, , we show that the for balls of fixed radius in $\core(\teich(\no_g))$, the $\nu_N$ volume of the ball is bounded above and below by constants that are independent of the center of the ball.

Let $\mathcal{P}$ be a pants decomposition for $\no_g$: recall the formula for $\nu_N$.

\begin{align*}
  \nu_N = \left( \bigwedge_{\text{$\gamma_i$ one-sided}} \coth(\ell(\gamma_i)) d\ell(\gamma_i) \right) \wedge \left( \bigwedge_{\text{$\gamma_i$ two-sided}} d\tau(\gamma_i) \wedge d\ell(\gamma_i) \right)
\end{align*}
Here $\ell(\gamma_i)$ denotes the length of the curve $\gamma_i$, and $\tau(\gamma_i)$ denotes the twist, when $\gamma_i$ is two-sided.

\begin{proposition}
  \label{prop:uniform-volume-bound}
  For any $\kappa > 0$, and $\vept > 0$ small enough, there exist positive constants\footnote{We will also consider $c_1$ and $c_2$ as functions of $\kappa$ elsewhere in the paper.} $c_1$ and $c_2$ (depending only on $\kappa$ and $\vept$) such the $\nu_N$ volume of a ball $B_{\kappa}^{\vept}(x)$ of radius $\kappa$ centered at $x \in \systole(\no_g)$ are bounded below and above by $c_1$ and $c_2$.
  \begin{align*}
    c_1 \leq \nu_N(B_{\kappa}^{\vept}(x)) \leq c_2
  \end{align*}
\end{proposition}
\begin{proof}
  Note that since the points we are considering lie in $\systole(\no_g)$, we have the following upper bound and lower bound for $\coth(\ell(\gamma_i))$, where $\gamma_i$ is a one sided curve.
  \begin{align}
    \label{eq:coth-bounds}
    1 \leq \coth(\ell(\gamma_i)) \leq \coth(\vept)
  \end{align}
  In particular, the $\nu_N$ volume of a ball can be bounded above and below by $\coth(\vept)\nu_N^{\prime}$ and $\nu_N^{\prime}$, where $\nu_N^{\prime} = \left( \bigwedge_{\text{$\gamma_i$ one-sided}} d\ell(\gamma_i) \right) \wedge \left( \bigwedge_{\text{$\gamma_i$ two-sided}} d\kappa(\gamma_i) \wedge d\ell(\gamma_i) \right)$.

  We now split up $\systole(\no_g)$ into two regions: $\thick(\no_g)$, and the complementary region.
  Since $\mcg(\no_g)$ acts cocompactly on $\thick(\no_g)$, and $\nu_N(B_{\kappa}^{\vept}(x))$ is continuous in $x$, the desired bounds hold in this region.
  It will therefore suffice to prove the bounds in the complementary region.

  Note that for any $x$ in the complementary region, there is some two-sided curve $\gamma$ that is short.
  By Theorem \ref{thm:prno}, the ball $B_{\kappa}^{\vept}(x)$ is contained in a product of balls, one in $\mathbb{H}$, and one in $\systole(\no_g \setminus \gamma)$.
  We pick $\gamma$ to be part of a pants decomposition $\mathcal{P}$, and write $\nu_N$ as follows.
  \begin{align}
    \label{eq:measure-of-product-is-product-of-measure}
    \nu_N = \left( d\kappa(\gamma) \wedge d\ell(\gamma) \right) \wedge \nu_N^{\no_g \setminus \gamma}
  \end{align}
  Here, $\nu_N^{\no_g \setminus \gamma}$ denotes the volume form on $\teich(\no_g \setminus \gamma)$.
  As a result, we have that the $\nu_N$ measure of a product of the two balls is the product of the corresponding measures of those balls.

  The measure of any ball of a fixed radius in $\mathbb{H}$ is constant, since $\mathbb{H}$ is homogeneous.
  The $\nu_N^{\no_g \setminus \gamma}$ measure of a ball in $\systole(\no_g \setminus \gamma)$ is again bounded above and below by fixed constants, by inducting on a surface of lower complexity.

  Since we have uniform bounds for both the terms in the product, we get uniform bounds for the measure of a ball in $\systole(\no_g)$.
\end{proof}

\subsection{Teichmüller Geodesics and Geodesics in the Curve Complex}
\label{sec:dist-form-teichm}

In this section, we will deduce some standard results about Teichmüller geodesics and the corresponding curve complex geodesics for non-orientable surfaces by reducing to the orientable case.
The following lemma will be the main tool for the reduction to the orientable case.
\begin{lemma}
  \label{lem:lifting-subsurfaces}
  Let $[x, y]$ be a Teichmüller geodesic segment in $\teich(\no)$, where $\no$ is a non-orientable surface, and $[\wt{x}, \wt{y}]$ be its image in $\teich(\os)$, where $\os$ is the orientable double cover of $\no$.
  Let $V$ be a subsurface of $S$: then the following statements hold for $d_V(\wt{x}, \wt{y})$.
  \begin{enumerate}[(i)]
  \item If $V$ is the lift of an orientable subsurface $W$ in $\no$, then $d_V(\wt{x}, \wt{y}) = d_{\iota(V)}(\wt{x}, \wt{y}) = d_W(x, y)$.
  \item If $V$ is the lift of a non-orientable subsurface $W$ in $\no$, then $d_V(\wt{x}, \wt{y}) \emul d_W(x, y)$.
  \item If $V$ is not a lift of a subsurface in $\no$, then there exists a uniform constant $k_0$, independent of $x$, $y$, and $V$, such that $d_V(\wt{x}, \wt{y}) \leq k_0$.
  \end{enumerate}
\end{lemma}
\begin{proof}
  We deal with the proof in cases.
  \begin{enumerate}[(i)]
  \item If $V$ is the lift of an orientable surface, we have that the covering map restricted to $V$ is a homeomorphism, and the same holds for $\iota(V)$, so the result follows in this case as well.
  \item If $V$ is the lift of a non-orientable subsurface $W$, then by Theorem \ref{thm:qi-embedding-curve-complex}, we have that $\cC(W)$ quasi-isometrically embeds into $\cC(V)$, and the result follows.
  \item If $V$ is not a lift at all, that means $V$ and $\iota(V)$ are transverse subsurfaces.
    By the Behrstock inequality, there exists a $k_0$ such that the following holds.
    \begin{align}
      \label{ineq:min-of-equal-variant-one}
      \min\left( d_V(\wt{x}, \partial \iota(V)), d_{\iota(V)}(\wt{x}, \partial V) \right) \leq \frac{k_0}{2}
    \end{align}
    But we also have that $\wt{x}$ is fixed by $\iota$, which gives us the following equality of curve complex distances.
    \begin{align}
      \label{eq:x-fixed-curve-1}
      d_V(\wt{x}, \partial \iota(V)) = d_{\iota(V)}(\wt{x}, \partial V)
    \end{align}
    Combining \eqref{ineq:min-of-equal-variant-one} and \eqref{eq:x-fixed-curve-1}, we get the following bound on the $\cC(V)$ distance between $\wt{x}$ and $\partial \iota(V)$.
    \begin{align*}
      d_V(\wt{x}, \partial \iota(V)) \leq \frac{k_0}{2}
    \end{align*}
    We have that the same bounds also hold for $\wt{y}$, so the result follows from the above inequality and the triangle inequality.
  \end{enumerate}
  This shows the result for all the cases and concludes the proof.
\end{proof}

We begin by proving the distance formula for points in Teichmüller space.
Let $x$ and $y$ be a pair of points in $\teich(\no_g)$, and let $\Gamma$ be the set of curves that are short on both $x$ and $y$, $\Gamma_x$ the set of curves that are only short on $x$, and $\Gamma_y$ the set of curves that are only short on $y$.
Let $\mu_x$ and $\mu_y$ be short markings on $x$ and $y$ respectively.
Let $\cC^+$ and $\cC^-$ denote the set of two-sided and one-sided curves on $\no_g$.
Finally, let $\left[ x \right]_k$ be the function which is $0$ for $x \leq k$, and identity for $x > k$.

\begin{theorem}[Distance formula]
  \label{thm:distance-formula}
  The distance between $x$ and $y$ in $\teich(\no_g)$ is given by the following formula.
  \begin{equation}
  \begin{aligned}
    d(x, y) &\emul \sum_{Y} \left[ d_Y(\mu_x, \mu_y) \right]_k + \sum_{\alpha \in \Gamma^c \cap \cC^+} \left[ \log(d_{\alpha}(\mu_x, \mu_y)) \right]_k \\
    &+ \max_{\alpha \in \Gamma \cap \cC^+} d_{\mathbb{H}_{\alpha}}(x, y) + \max_{\alpha \in \Gamma \cap \cC^-} d_{(\mathbb{R}_{>0})_{\alpha}}(x, y) \\
    &+ \max_{\alpha \in \Gamma_x} \log \frac{1}{\ell_x(\alpha)} + \max_{\alpha \in \Gamma_y} \log \frac{1}{\ell_y(\alpha)}
  \end{aligned}
  \label{eq:distance-formula-nonorient}
  \end{equation}
\end{theorem}
\begin{proof}
  Let $\wt{x}$ and $\wt{y}$ be the images of $x$ and $y$ in $\teich(\os_{g-1})$ under the isometric embedding map.
  Since $d(x,y) = d(\wt{x}, \wt{y})$, it will suffice to estimate $d(\wt{x}, \wt{y})$ using distances in the curve complexes.
  Let $\wt{\mu_x}$ and $\wt{\mu_y}$ be the lifts of $\mu_x$ and $\mu_y$.
  Both $\wt{\mu_x}$ and $\wt{\mu_y}$ are short markings on $\wt{x}$ and $\wt{y}$ respectively.
  We have by Rafi's distance formula \cite[Theorem 6.1]{rafi2007combinatorial}, the following estimate on $d(\wt{x}, \wt{y})$.
  \begin{equation}
  \begin{aligned}
    d(\wt{x}, \wt{y}) &\emul \sum_{Y} \left[ d_Y(\wt{\mu_x}, \wt{\mu_y}) \right]_k + \sum_{\alpha \in \wt{\Gamma}^c} \left[ \log(d_{\alpha}(\wt{\mu_x}, \wt{\mu_y})) \right]_k \\
    &+ \max_{\alpha \in \wt{\Gamma}} d_{\mathbb{H}_{\alpha}}(\wt{x}, \wt{y})\\
    &+ \max_{\alpha \in \wt{\Gamma_{\wt{x}}}} \log \frac{1}{\ell_{\wt{x}}(\alpha)} + \max_{\alpha \in \Gamma_{\wt{y}}} \log \frac{1}{\ell_{\wt{y}}(\alpha)}
  \end{aligned}
  \label{eq:distance-formula-orient}
  \end{equation}
  Here, $\wt{\Gamma}$, $\wt{\Gamma_{\wt{x}}}$, and $\wt{\Gamma_{\wt{y}}}$ are curves on $\wt{x}$ and $\wt{y}$ that are simultaneously short, short on $\wt{x}$ and not on $\wt{y}$, and short on $\wt{y}$ and not on $\wt{x}$ respectively.

 It will suffice to show that for a large enough choice of $k$, the right hand side of \eqref{eq:distance-formula-nonorient} is equal to the right hand side of \eqref{eq:distance-formula-orient}, up to an additive and multiplicative constant.
 We consider the first term in the right hand side of \eqref{eq:distance-formula-orient}, namely the sum over the non-annular subsurfaces $Y$.
 There are three possibilities for $Y$ in $\os_{g-1}$, which we deal with using Lemma \ref{lem:lifting-subsurfaces}.
 \begin{enumerate}[(i)]
 \item $Y$ is one component of a lift of an orientable subsurface $Z$ of $\no_g$: In this case we have $d_Y(\wt{\mu_x}, \wt{\mu_y}) = d_Z(\mu_x, \mu_y)$ (and the same equality with $Y$ replaced with $\iota(Y)$).
   Thus, for every term associated to an orientable non-annular subsurface $Z$ in \eqref{eq:distance-formula-nonorient}, we get two corresponding equal terms in \eqref{eq:distance-formula-orient}.
 \item $Y$ is the lift of a non-orientable subsurface $Z$ of $\no_g$: In this case, we have $d_Y(\wt{\mu_x}, \wt{\mu_y}) \emul d_Z(\mu_x, \mu_y)$.
 \item $Y$ is not a lift of a subsurface of $\no_g$: In this case, we have the following for some $k_0$.
   \begin{align*}
     d_Y(\wt{\mu_x}, \wt{\mu_y}) \leq k_0
   \end{align*}
   If we pick a threshold $k > k_0$, the subsurfaces $Y$ that do not arise from lifts will not contribute to the right hand side of \eqref{eq:distance-formula-orient}.
 \end{enumerate}

 We now do the same case analysis for annular subsurfaces: consider a curve $\alpha$ on $\os_{g-1}$ that is contained in $\Gamma^c$, i.e.\ it is not simultaneously short on $\wt{x}$ and $\wt{y}$. There are three possibilities for $\alpha$.
 \begin{enumerate}[(i)]
 \item $\alpha$ is one component of a lift of a two-sided curve $\gamma$ on $\no_g$: In this case, $\alpha$ and $\iota(\alpha)$ are disjoint, and the restriction of the covering map to these curves is a homeomorphism.
   We have $d_\alpha(\wt{\mu_x}, \wt{\mu_y}) = d_\gamma(\mu_x, \mu_y)$: consequently, for every term in $\Gamma^c \cap \cC^+$ in \eqref{eq:distance-formula-nonorient}, we have two equal terms in \eqref{eq:distance-formula-orient}.
 \item $\alpha$ is the lift of a one-sided curve on $\no_g$: In this case $\alpha = \iota(\alpha)$, but the transformation $\iota$ reverses orientation on the surface $\os_{g-1}$.
   That means $\wt(\mu_x)$ and $\wt{\mu_y}$ cannot have a relative twist between them along $\alpha$, because if they did, $\iota(\wt(\mu_x))$ and $\iota(\wt(\mu_y))$ would have the opposite twist.
   On the other hand $\wt{\mu_i} = \iota(\wt{\mu_i})$ for $i=x$ and $i=y$, which means the relative twist must be $0$.
   This proves that the $\alpha$ which are lifts of one-sided curves do not contribute to the second term of \eqref{eq:distance-formula-orient}.
 \item $\alpha$ is not a lift of a curve on $\no_g$: In this case $\alpha$ and $\iota(\alpha)$ intersect each other, and are not equal, which means they are transverse.
   We deal with this the same way we dealt with transverse non-annular subsurfaces, i.e.\ via the Behrstock inequality.
 \end{enumerate}
 This case analysis proves that the second terms on the right hand side of \eqref{eq:distance-formula-nonorient} and \eqref{eq:distance-formula-orient} are equal, up to an additive and multiplicative constant.

We now deal with the last three terms of \eqref{eq:distance-formula-orient}.
These terms deal with short curves on $x$ or $y$: we claim that the short curves must be lifts of either one-sided or two-sided curves in $\no_g$.
Suppose a curve $\alpha$ is short and not a lift. Then $\alpha$ has positive intersection number with $\iota(\alpha)$, but since $\iota$ is an isometry, $\iota(\alpha)$ must also be short.
For a sufficiently small threshold for what we call short, we can't have a short curve intersecting another short curve, which proves the claim that all the short curves arise as lifts.

Since the curves in $\wt{\Gamma}$ are all lifts, the third term of \eqref{eq:distance-formula-orient} can be split up into two terms: the lifts of the two-sided and one-sided curves.
For the two-sided curves, the distance calculation involves both the length and twist coordinate, and for the one-sided curves, only the length coordinate is involved.
This follows from Theorem \ref{thm:prno}.

Finally, the last two terms in \eqref{eq:distance-formula-orient} are the same as the last two terms of \eqref{eq:distance-formula-nonorient}, up to an additive error of $(6g) \cdot \log(2)$, since the lift of a short curve can double its length, and there are no more than $6g$ short curves.

We have shown that the right hand sides of \eqref{eq:distance-formula-nonorient} and \eqref{eq:distance-formula-orient} are equal, up to a multiplicative and additive constant, which proves the result.
\end{proof}

We now verify that Teichmüller geodesics can be broken up into \emph{active intervals associated to subsurfaces}, which are subintervals of the geodesic associated to each subsurface $V$, along which the projection to $V$ is large, and outside of which, the projection is bounded.
The following lemma  of \textcite[Lemma 3.26]{dowdall2023lattice} (which itself is a generalization of \textcite[Proposition 3.7]{rafi2007combinatorial}) describes the subsegments of $[x,y]$ along which the geodesic makes progress in the curve complex of a subsurface.

\begin{proposition}
  \label{thm:active-intervals}
  For each sufficiently small $\vept > 0$, there exists $0 < {\vept}^{\prime} < \vept$ and $M_{\vept} \geq 0$ such that for any subsurface $V \sqsubset S$, there's a (possibly empty) connected interval $\mathcal{I}_V^{\vept} \subset [x,y]$ such that the following five conditions hold.
  \begin{enumerate}[(i)]
  \item If $d_V(x, y) \geq M_{\vept}$, then $\mathcal{I}_V^{{\vept}}$ is a non-empty subinterval of $[x,y]$.
  \item $\ell_\alpha(z) < {\vept}$ for all $z \in \mathcal{I}_V^{{\vept}}$ and $\alpha \in \partial V$.
  \item For all $z \in [x,y] \setminus \mathcal{I}_V^{{\vept}}$, some component $\alpha$ of $\partial V$ has $\ell_\alpha(z) > {\vept}^{\prime}$.
  \item $d_V(w, z) \leq M_{\vept}$ for every subinterval $[w,z] \subset [x,y]$ if $[w,z] \cap \mathcal{I}_V^{\vept} = \varnothing$.
  \item For a pair of traverse subsurfaces $U$ and $V$, $\mathcal{I}_U^{\vept} \cap \mathcal{I}_V^{\vept} = \varnothing$.
  \end{enumerate}
\end{proposition}

\begin{proof}[Proof of Theorem \ref{thm:active-intervals} for non-orientable surfaces]
  Let $\no$ be the non-orientable surface, and $\os$ its double cover.
  We consider the image $[\wt{x}, \wt{y}]$ of the geodesic $[x,y]$ in $\teich(\os)$.
  We know that the result holds for $[\wt{x}, \wt{y}]$, although with ${\vept}^{\prime}$ replaced with $\frac{{\vept}^{\prime}}{2}$, since lifting can double the lengths of some curves.

  The main fact we need to verify is that the only subsurfaces $V$ that have non-empty $\mathcal{I}_V^{\vept}$ come from lifts.
  If $V$ is a subsurface of $\os$ that is not a lift, we use case (iii) of Lemma \ref{lem:lifting-subsurfaces} to conclude that $d_V(x, y) \leq k_0$ for some fixed constant $k_0$.
  Picking $M_{\vept} > k_0$ guarantees that the only subsurfaces for which $\mathcal{I}_V^{\vept}$ is non-empty arise from lifts, which proves the result for non-orientable surfaces.
\end{proof}

Finally, we show that the consistency and the realization results (\textcite{behrstock2012geometry}) hold for Teichmüller spaces of non-orientable surfaces as well.
We begin by recalling the definition of consistency.

\begin{definition}[Consistency]
  For a connected surface $S$, and a parameter $\theta \geq 1$, we say a tuple $(z_V) \in \prod_{V \sqsubset S} \cC(V)$ is $\theta$-consistent if the following two conditions holds for all pairs of subsurfaces $U$ and $V$.
  \begin{enumerate}[(i)]
  \item If $U \pitchfork V$, then
    \begin{align*}
      \min(d_U(z_U, \partial V), d_V(z_V, \partial U)) \leq \theta
    \end{align*}
  \item If $U \sqsubset V$, then
    \begin{align*}
      \min(d_U(z_U, \pi_U(z_V)), d_V(z_V, \partial U)) \leq \theta
    \end{align*}
  \end{enumerate}
\end{definition}

The following theorem (\textcite[Theorem 4.3]{behrstock2012geometry}) states that the projection from Teichmüller space to the curve complexes of all the subsurfaces is coarsely surjective onto the set of consistent tuples.

\begin{theorem}[Consistency and realization]
  \label{thm:consistency-realization}
  There is a constant $K \geq 1$, and function $\mathfrak{C}: \mathbb{R}_+ \to \mathbb{R}_+$ such that the following holds for any surface $S$.
  \begin{itemize}
  \item (Consistency) For every $x \in \teich(S)$, the projection tuple $(\pi_V(x))_{V \sqsubset S}$ is $K$-consistent.
  \item (Realization) For every $\theta$-consistent tuple $(z_V)_{V \sqsubset S}$, there exists a point $z \in \teich(S)$ such that $d_V(\pi_V(z), z_V) \leq \mathfrak{C}(\theta)$ for all $V$.
  \end{itemize}
\end{theorem}
\begin{proof}[Proof sketch of Theorem \ref{thm:consistency-realization} for non-orientable surfaces]
    We first show that the projection map is consistent, and then show consistent tuples lie coarsely in the image of the projection map.
  \begin{itemize}
  \item (Consistency) We map $x$ to $\wt{x}$ in the Teichmüller space of the double cover $\wt{S}$.
    By applying the theorem for orientable surfaces, we have the $(\pi_{W}(\wt{x}))_{W \sqsubset \wt{S}}$, and we restrict to the subsurfaces in the tuple which arise as lifts.
    These points lie in the image of the quasi-isometric embedding map from Theorem \ref{thm:qi-embedding-curve-complex}, which means consistency also holds for the tuples in $S$.
  \item (Realization) Given a $\theta$-consistent tuple $(z_V)_{V \sqsubset S}$, we construct a $\theta^{\prime}$-consistent tuple in the double cover $\wt{S}$.
    For subsurfaces of $\wt{S}$ that arise as lifts, we use the map from Theorem \ref{thm:qi-embedding-curve-complex}.
    For the subsurfaces $W$ that are not lifts, we set $z_W = \pi_W(\partial(\iota(W)))$.
    The fact that this is a $\theta^{\prime}$-consistent tuple follows from the Behrstock inequality (for some $\theta^{\prime} > \theta$)\footnote{A longer but a more thorough way of seeing this would be to verify that the Teichmüller space of a non-orientable surface satisfies the $9$ axioms for hierarchical hyperbolicity that are enumerated in \textcite{behrstock2019hierarchically}.}.
    We now use this point to construct $y \in \teich(\wt{S})$, and deduce that $y$ is coarsely fixed by $\iota^{\ast}$, i.e. the distance between $y$ and $\iota^{\ast}(y)$ is bounded above by a uniform constant $C$.
    Consider the midpoint $\wt{x}$ of the Teichmüller geodesic segment joining $y$ and $\iota^{\ast}(y)$: by Teichmüller's uniqueness theorem, $\wt{x}$ is fixed under the $\iota^{\ast}$-action.
    This means there is some $x \in \teich(S)$ whose image is $\wt{x}$, and therefore the projection maps are coarsely $(z_V)$.
  \end{itemize}
\end{proof}
